\documentclass[10pt]{article} 

\usepackage{listings}
\usepackage[colorlinks=false, backref]{hyperref}
\usepackage{amsmath,amsthm,amsfonts,amssymb,amscd,amsbsy}

\usepackage{float}
\usepackage{graphicx}

\usepackage[mathscr]{euscript}  
\usepackage{stmaryrd}
\usepackage{microtype}
\usepackage{booktabs}
\usepackage{cleveref}
\usepackage{bookmark}
\usepackage{mathrsfs}

\usepackage{emptypage}
\usepackage{tikz}
\usepackage{subfigure}
\usepackage[utf8]{inputenc}
\usepackage[T1]{fontenc}

\usepackage{mathabx}
\usepackage{extarrows}
\usepackage{fancyhdr}
\usepackage{textgreek}

\theoremstyle{plain}
\newtheorem{theorem}{\scshape Theorem}

\newtheorem{lemma}{\scshape Lemma}
\newtheorem{corollary}{\scshape Corollary}

\theoremstyle{definition}

\newtheorem{definition}{\scshape Definition}
\newtheorem{remark}{\scshape Remark}

\DeclareMathAlphabet{\mathpzc}{OT1}{pzc}{m}{it}
\DeclareMathAlphabet\mathbfcal{OMS}{cmsy}{b}{n}

\definecolor{grey}{rgb}{0.5,0.5,0.5}
\definecolor{lightgrey}{rgb}{0.9,0.9,0.9}
\definecolor{darkgreen}{rgb}{0,0.6,0}
\definecolor{orange}{rgb}{1,0.5,0}
\definecolor{lightpink}{rgb}{1,0.714,0.757}
\definecolor{lightorange}{rgb}{1,0.855,0.725}

\DeclareMathOperator{\spt}{spt}
\DeclareMathOperator{\dist}{dist}
\DeclareMathOperator{\GL}{GL^+(d, \mathbb{R})}
\DeclareMathOperator{\cof}{cof}

\def\bp{{\partial_1}}
\def\p{\partial}

\def\dx{ {\d}_x}

\def\zxg{ \xi_x ^{-\gamma} }
\def\zxgo{ \xi_x ^{-\gamma-1} }

\def\nxg{ \eta_x ^{-\gamma} }
\def\nxgo{ \eta_x ^{-\gamma-1} }

\def\gpo{ {\gamma+1} }

\def\deta{{\delta\hspace{-.03in}\eta} }

\def\aki{{a^k_i}}

\def\bp{ \partial^{\operatorname{tan} }}
\def\bp{{\bar \partial }}
\def\M{ {\widetilde M}}
\def\a{\mathrm{a} }

\def\rr{ \mathbb{R}  }

\def\ss{{\scriptstyle \mathcal{S} }}

\def\dd{{\dist }}
\def\d{{\mathbf{d} }}
\def\R{{\mathrm{R} }}

\def\P{{\mathscr{P} }}
\def\A{{\mathscr{A} }}
\def\B{{\mathscr{B} }}
\def\F{{\mathscr{F} }}

\def\ddd{{\cdot \cdot \cdot }}

\def\at{ \alpha^{-2} }
\def\ah{ \alpha^{-1/2} }
\def\agt{\alpha^{3-\gamma}}
\def\agtt{\alpha^{\frac{3-\gamma}{2}}}
\def\ee{ \boldsymbol{\eta}}
\def\dee{ \boldsymbol{{\delta\hspace{-.03in}\eta}}}
\def\vv{ \boldsymbol{v}}
\def\AA{ \boldsymbol{A}}
\def\aa{ \boldsymbol{a}}
\def\JJ{ \boldsymbol{J}}

\title{{\bf Global existence of near-affine solutions to the compressible Euler equations}}

\author{
 {\small {\bf Steve Shkoller}}
 \vspace{-.05 in}
\\{\footnotesize Department of Mathematics}
\vspace{-.05 in}
\\{\footnotesize University of California}
\vspace{-.05 in}
\\{\footnotesize Davis, CA 95616}
\vspace{-.05 in}
\\{\footnotesize email: {\it shkoller@math.ucdavis.edu}}
\and
{\small {\bf Thomas C. Sideris}}
\vspace{-.05 in}
\\{\footnotesize Department of Mathematics}
\vspace{-.05 in}
\\{\footnotesize University of California}
\vspace{-.05 in}
\\{\footnotesize Santa Barbara, CA 93106}
\vspace{-.05 in}
\\{\footnotesize email: {\it sideris@math.ucsb.edu}}
}
\date{ }

\begin{document}

\maketitle

\begin{abstract} We establish global existence of solutions to the compressible Euler equations, in the case that a finite volume of  ideal gas
expands into vacuum.  Vacuum states can occur with either smooth or singular sound speed, the latter corresponding to the so-called 
{\it physical vacuum singularity} when the enthalpy vanishes on the vacuum wave front like the distance function.   In this instance, the Euler equations
lose hyperbolicity and form a degenerate system of conservation laws, for which a local existence theory has only recently been developed.  
Sideris \cite{Sideris2017} found a class of expanding finite degree-of-freedom global-in-time {\it affine solutions}, obtained by solving 
nonlinear ODEs.   In three space dimensions, the stability of these affine solutions, and hence global existence of solutions,  was established 
by Had\v{z}i\'{c} \& Jang \cite{HaJa2016} with the pressure-density relation $p = \rho^\gamma$ with the constraint  that $1< \gamma\le {\frac{5}{3}} $.  
They asked if a different approach could go beyond the $\gamma > {\frac{5}{3}} $ threshold.    We provide an affirmative answer to their question, and
prove stability of affine flows and global existence for all $\gamma >1$, thus also establishing global existence for the shallow water equations when $\gamma=2$.
\end{abstract}

\section{Introduction}
\label{sec_introduction}
We consider the problem of global existence of solutions to the isentropic  compressible Euler equations of gas dynamics in $d$ space dimensions, with $d=1,2,3$.
The isentropic system consists of two basic conservation laws: the conservation of momentum and the conservation of mass.   Letting $u = (u_1,...,u_d)$ denote
the velocity vector and $\rho$ the density function, conservation of momentum and mass can be written, respectively, as 
$$
\p_t (\rho u)+  \operatorname{div} [ \rho u \otimes u + p(\rho) \operatorname{Id} ]=0 \ \text{ and } \ \p_t \rho + \operatorname{div} (\rho u) =0 \,,
$$
supplemented with the ideal gas law which states that the pressure function $p( \rho) = \rho^\gamma$ for $\gamma >1$.    Both $u$ and $\rho$ depend on the space coordinate
$x=(x_1,...,x_d)$ and time $t\ge 0$.   By definition, a vacuum state exists on the set
 $\{\rho(x,t)=0\}$.   The set $\Omega(t) := \{\rho(x,t) >0\}$ defines the location of the gas, and of particular interest is the evolution of the gas-vacuum boundary 
 $\Gamma(t):= \p  \Omega (t)$, also known as the ``front.''
 
 In the absence of vacuum, when $\rho(x,t) \ge \varrho>0$ in $ \mathbb{R}  ^d$,  the Euler equations are hyperbolic, and when
the initial density and  velocity are constant outside a compact set, classical $C^1$ solutions exist on a (possibly very short) time interval $[0,T]$ 
(see, for example, Majda \cite{Majda1984}).
Sideris \cite{Sideris1985} proved that even for small initial data, if the gas is  slightly compressed and out-going near the front $\Gamma(t)$,
 then $T< \infty $ (see also \cite{Chemin1990b}).
 
 In the presence of vacuum, strict hyperbolicity fails and  local-in-time existence was established in \cite{MaUkKa1986,Chemin1990}.  
As we shall describe below, vacuum states can occur with either smooth or singular sound speed.     For the easier case of smooth sound speed, 
Serre \cite{Serre1997} proved that there exist sufficiently small and smooth  initial density profiles with 
smooth initial velocities close to a linear field that cause gas particles to spread and the velocity to decay, thus leading to global-in-time solutions.  
  
The  case of
singular sound speed, also known as the {\it physical vacuum singularity}, is more difficult to analyze;  a priori estimates were given in 
Coutand, Lindblad, \& Shkoller \cite{CoLiSh2010}, and local existence was established by 
Coutand \& Shkoller \cite{CoSh2011,CoSh2012} and Jang \& Masmoudi \cite{JaMa2009, JaMa2015}. Later 
Coutand \& Shkoller  \cite{CoSh2014}  showed that there exists a large class of physical vacuum solutions  in which the vacuum boundary self-intersects in
finite-time, thus leading to a so-called splash singularity.   On the other hand, 
Sideris \cite{Sideris2017} has constructed a new class of global-in-time {\it affine solutions} for the physical vacuum case, which are
obtained by solving ODEs.  For three-dimensional flows,  Had\v{z}i\'{c} \& Jang \cite{HaJa2016} have constructed global solutions by proving the global
existence of small perturbations to these affine flows for the case that $1< \gamma \le {\frac{5}{3}} $; with their scheme,  if $\gamma > {\frac{5}{3}} $, a
so-called  ``anti-damping'' term
arises and their method is not applicable.  In particular,  Remark 1.3 of \cite{HaJa2016} states: ``It would be interesting to
understand whether one can go beyond the $\gamma > {\frac{5}{3}} $
threshold.''   

The purpose of this paper is to prove that, in fact, global solutions exist for all $\gamma >1$.  We introduce a simple strategy  which shows that there is no upper bound on $\gamma$, and thus prove global existence for the physical vacuum singularity
for the general ideal gas law.   This includes the important case of the shallow water equations corresponding to $\gamma=2$.

\subsection{The Euler equations with vacuum states as a free-boundary problem}  
We fix a time interval
$0 \le t \le T$.   In the presence of the physical vacuum singularity, solutions to the Euler equations are not  smooth across the front $\Gamma(t)$; thus, rather than studying
weak solutions on $ \mathbb{R}  ^d \times [0,T]$, we instead formulate the Euler equations as a free-boundary problem set on the {\it a priori} unknown domain
$\Omega(t)$.    We write 
the isentropic compressible Euler equations as 
\begin{subequations}
  \label{ceuler3d}
\begin{alignat}{2}
\rho[\p_t u+ u\cdot \nabla u] +\nabla  p(\rho)&=0  &&\text{in} \ \ \Omega(t) \,, \\
 \p_t \rho+ {\operatorname{div}}  (\rho u) &=0 
&&\text{in} \ \ \Omega(t) \,, \\
p &= 0 \ \ &&\text{on} \ \ \Gamma(t) \,, \\
\mathcal{V} (\Gamma(t))& = u \cdot n && \\
(\rho,u,\Omega)   &= (\rho_0,u_0 , \Omega_0) \ \  &&\text{on} \ \{t=0\} \,.
\end{alignat}
\end{subequations}
As noted above, the gas occupies the open, bounded subset
 $\Omega(t) \subset \mathbb{R}^d  $,   $\Gamma(t):= \partial\Omega(t)$ denotes
 the time-dependent vacuum boundary, $ \mathcal{V} (\Gamma(t))$ denotes the normal
 velocity (or speed)  of $\Gamma(t)$, and $n=n(\cdot , t)$ denotes the exterior unit normal vector to $\Gamma(t)$.

\subsection{Physical vacuum boundary condition}  \label{subsec_physicalvacuum}
The rate at which the density $\rho(x,t)$ vanishes  on $\Gamma(t)$ is extremely important.    
With the sound speed given by $c := \sqrt{\p p/ \p \rho}$ and $N$ denoting the {\it outward} unit normal to $\Gamma_0:=\partial \Omega_0$, satisfaction of the condition
\begin{equation}\label{phys_vac}
\frac{ \partial c_0^2}{ \partial N} < 0 \text{ on } \Gamma_0
\end{equation} 
defines a {\it physical vacuum} boundary (see \cite{Lin1987}, \cite{Liu1996}, \cite{LiYa1997}, \cite{LiYa2000},
\cite{LiSm1980}, \cite{XuYa2005}, \cite{LuXiZe2014}, \cite{LuZe2016}), where $c_0 = c|_{t=0}$ denotes the initial sound speed of the gas.   In this case, the sound speed is {\it not} a smooth
function on $ \mathbb{R}  ^d$.

 The physical vacuum condition (\ref{phys_vac}) is equivalent to the
requirement that
\begin{equation}\label{stab1}
\frac{ \p \rho_0^{\gamma-1}} { \partial N} < 0 \text{ on } \Gamma_0 \,,
\end{equation}
a condition necessary for the gas particles on the boundary to accelerate.
Since $ \rho_0 >0$ in $\Omega_0$, (\ref{stab1}) implies that for some positive constant $C$ and $x\in \Omega_0$ near the vacuum boundary $\Gamma_0$,
\begin{equation}\label{degen}
\rho_0^{\gamma-1}(x) \ge  C \dd(x, \Gamma_0) \,,
\end{equation} 
where $\dd$ denotes the distance function, and $\dd(x, \Gamma_0) $ is the distance from the point $x \in \Omega_0$ to the boundary $\Gamma_0$.   As noted
by Serre \cite{Serre2015},  the {\it physical vacuum} condition occurs when the front $\Gamma(t)$ 
 is accelerated by a force resulting from a  H\"{o}lder singularity of the sound speed $c$.  Vacuum states can occur when the sound speed is smooth or singular; 
the physical vacuum requires that $c$ have $ {\frac{1}{2}} $-H\"{o}lder regularity, and we shall be concerned with this singular scenario herein.   See also
\cite{LaMe2017} for the analysis of the shoreline problem using the 1-d shallow water equations.

\subsection{Affine solutions to compressible Euler}\label{sec:affine}
  The set of $d \times d$ matrices is denoted by $ \mathbb{M}  ^d$.
We let $ \GL =\{ \A \in \mathbb{M}  ^d \ : \ \det \A >0\}$.   With $\Omega_0 = B(0,1)=\{ |x| < 1\}$, 
Sideris \cite{Sideris2017} constructed  global-in-time affine solutions to \eqref{ceuler3d} as follows:
\begin{align}
\Omega(t) &= \A(t) B(0,1) \,, \ \ 
u(y,t) = \dot \A(t) \A(t)^{-1}  y \,, \ \ 
\rho(y,t) = \rho_0(| \A(t) ^{-1} y|)/\det A(t) \,. \label{affine_soln}
\end{align} 
where $(\A,\dot\A) \in C^0( [0 , \infty ); \GL)$ satisfies the ordinary differential equation
\begin{equation}\label{agen}
\ddot \A(t) = \det(\A(t))^{1- \gamma} \A(t) ^{-T}   \ \ t>0  \,,
\end{equation} 
with initial conditions at $t=0$ given by $\A(0)$ and $\dot \A(0)$.  
Thus, the vacuum boundary
 $\Gamma(t)$ evolves as a rotating ellipsoid.   We shall restrict our attention to affine motions for which the gas is expanding in all three directions; in 
 particular, we
 shall consider affine solutions $\A(t)$ whose determinant grows at the maximal rate of $(1+t)^3$,
as $t\to\infty$ (which  necessarily holds if $1< \gamma \le {\frac{5}{3}} $).  The next lemma gives conditions under which such solutions exist.  

\begin{lemma}\label{lemma_affine}
Suppose that $1<\gamma\le5/3$.  Let $\A(t)$ be an arbitrary (global) solution of the system \eqref{agen} in 3-d.
Then 
\begin{equation}
\label{detgr}
\det \A(t)\sim (1+t)^{3},\quad t>0.
\end{equation}
and
\begin{equation}
\label{d2}
|\A''(t)|\lesssim (1+t)^{-3\gamma+2}.
\end{equation}
In addition,  there exists a unique positive definite matrix $\A_1$ such that for all $t>0$
\begin{equation}
\label{d1}
|\A'(t)-\A_1|\lesssim (1+t)^{-3\gamma+3}
\end{equation}
and
\[
|\A(t)-\A_1t|\lesssim
\begin{cases}
(1+t)^{-3\gamma+4},&1<\gamma<4/3,\\
\log(2+t),&\gamma=4/3,\\
1,&4/3<\gamma\le5/3.
\end{cases}
\]

Suppose that $5/3<\gamma$.  Choose matrices $\A_1$, $\A_0$, with $\A_1$ positive definite.
Then there exists a unique (global) solution $\A(t)$ 
of the 3-d system \eqref{agen} such that 
 \eqref{detgr}, \eqref{d2},  \eqref{d1} hold, and also
\begin{equation}
\label{d0}
|\A(t)-(\A_1t+\A_0)|\lesssim (1+t)^{-3\gamma+4}.
\end{equation}

\end{lemma}

\begin{proof}
Assume that $1<\gamma\le5/3$.
Let $\A(t)$ be an arbitrary (global) solution of the system \eqref{agen} in 3-d.  
By Theorem 3 (27) of  \cite{Sideris2017}, we have that \eqref{detgr} holds.

By the energy estimate (see Lemma 4 in \cite{Sideris2017}), we have $|\A'(t)|\lesssim1$, and so we obtain that
\begin{equation}
\label{solgr}
|\A(t)|\lesssim 1+t,\quad t>0.
\end{equation}
Using \eqref{detgr} and \eqref{solgr}, we find that
\begin{equation} 
\label{adpbd}
|\A''(t)|= \det \A(t)^{-(\gamma-1)}|\A(t)^{-\top}| \\ = \det \A(t)^{-\gamma}|\cof \A(t)|
\lesssim (1+t)^{-3\gamma +2},
\end{equation} 
which establishes \eqref{d2}.  The right-hand side of \eqref{adpbd} lies in
$L^1[0,\infty)$, and so for any $t>0$, we may write 
\begin{equation}
\label{apid}
\A'(t)=\A'(0)+\int_0^\infty \A''(\tau)d\tau-\int_t^\infty \A''(\tau0)d\tau.
\end{equation}
We define
\begin{equation}
\label{a1def}
\A_1=\A'(0)+\int_0^\infty \A''(\tau)d\tau.
\end{equation}
The equations \eqref{apid} and \eqref{a1def} immediately yield \eqref{d1}.
The uniqueness of $\A_1$ follows from the uniqueness of limits.
The remaining estimate follows by integration of \eqref{d1}.

In the case $5/3<\gamma$,  the statements \eqref{d2}, \eqref{d1}, \eqref{d0} follow directly from Theorem 5 in \cite{Sideris2017}.
The statement \eqref{detgr} follows from \eqref{d0}.
%
\end{proof}

\begin{remark}\label{rem1}
Since $\A_1\in {\rm GL}^+(3,\rr)$, using the singular value decompostion,
there exist $U,V\in{\rm SL}(3,\rr)$ such that $U\A_1V$ is  positive
definite and diagonal.  Note that if $\A(t)$ is a solution of \eqref{agen}, then so too is $U\A(t)V$.  Therefore,
we may assume without loss of generality that $\A_1$ is a diagonal matrix.
 \end{remark}

The simplest affine solution corresponds to the case that
$$\A(t) = \alpha(t) \operatorname{Id} \,,$$
in which case we have an expanding spherical surface, and the scalar $ \alpha (t)$ is the solution of
\begin{equation}\label{affine1d}
\ddot{\alpha}(t) = \alpha ^{-d \gamma +(d-1)}(t) \,,
\end{equation} 
and  by Theorem 5 in \cite{Sideris2017}, 
\begin{align*} 
\alpha (t)  &\sim 1+t \text{ as } t \to +\infty  \,, \ \text{ and } \  \int_0^\infty \frac{1}{\alpha (s)^{1 + \delta }}ds \le C < \infty \ \text{ for }  \delta >0 \\
\dot\alpha (t) &\sim 1    \text{ as } t \to +\infty  \,, \ \text{ and }  \dot \alpha(t) \ge 0 \text{ for all  } t\ge 0 \,.
\end{align*} 

In order to analyze the  stability of the general affine solution,
 $$
\A(t) =   \left[\begin{matrix} a_{11}(t) & a_{12}(t) & a_{13}(t) \\  a_{21}(t) & a_{22}(t) & a_{23}(t) \\ a_{31}(t)& a_{32}(t) &  a_{33}(t)\end{matrix}\right] \,,
 $$
we shall make use of Lemma \ref{lemma_affine}.

\subsection{Main Result}
The precise norms for our analysis shall be defined later in Sections \ref{sec:1d} and \ref{sec:3d}  but we can, nevertheless, state the main result.

\begin{theorem}[Main result]\label{thm1}   For affine solutions \eqref{affine_soln} such that the assumptions of Lemma \ref{lemma_affine} are satisfied, 
and for space dimension $d=1,2,3$  and  all $\gamma>1$,  sufficiently small perturbations of these affine solution  exist globally-in-time. Moreover, the perturbed velocity 
field decays (pointwise) to zero as  time $t \to \infty $ and the vacuum boundary remains close the affine ellipsoid for all time and maintains its regularity.
\end{theorem} 

\subsection{Outline} Section \ref{sec:1d} is intended as an introduction to the analysis of the stability of affine flows, and we restrict our
analysis to the relatively simple study of 1d fluid motion, wherein some of the ideas that allow  for all $\gamma >1$ are explained. 

In Section \ref{sec:3d}, we treat the three-dimensional stability problem.   Our primary goal is to explain the stability analysis for $\gamma > {\frac{5}{3}} $. We 
first focus on the (shallow water) case that $\gamma=2$, and then explain the treatment for general $\gamma$.

\section{Notaton, weighted embeddings, and Hardy inequalities}\label{sec:frac}
\subsection{Notation}  For $1\le p\le \infty $,  write $\|F\|_{L^p}$ for the Lebesgue space $L^p(\Omega_0)$ norm, and for $s \ge 0$, 
 $\|F\|_s$ for the  $H^s(\Omega_0)$ Sobolev norm.   The function $x\mapsto \P(x)$ will be used to denote
$C\, x^q$ for some $q >1$, and a generic constant $C$.   We write $ f \lesssim g$ if $f \le C\, g$ for  a generic constant $C$.

\subsection{Inequalities with weights and fractional norms}
Using $\d$ to denote the distance function to
the boundary $\Gamma$,  for $s \ge 0$, we consider the
weighted Sobolev space  $H^k(\Omega_0,\d, s )$, with norm given by 
$$\| F\|^2 _{H^k(\Omega_0,\d,s )} =  \int_{\Omega_0}  \d^ s \sum_{ | \alpha | \le k} |  \nabla^\alpha  F (x)|^2 \,.$$
If  $s > 2(k-r) -1$, then  there is a constant $C>0$ depending only on $\Omega_0$, $k$,$r$, and $ s$ such that
\begin{equation}\label{w-embed}
\| F\|^2 _{H^r(\Omega_0,\d, s -2(k-r) )}  \le C \| F\|^2 _{H^k(\Omega_0,\d, s )}  \,.
\end{equation} 
See Section 8.8 in  \cite{Kufner1985} together with the interpolation theory in Chapter 5 of \cite{KuPe2003}.

We shall also make use of the following higher-order Hardy inequality (see \cite{CoSh2011,CoSh2012} for the proof):
\begin{lemma}[Hardy's inequality in higher-order form]\label{hardy}   
If $u\in H^k(  \Omega_0)\cap H^1_0(\Omega_0) $ for $k\ge 1$, and $\d(x) >0$ for $x \in \Omega_0$,  $\d\in H^{r}( \Omega _0)$, $r= \max(k-1,3)$, 
and $\d$ is the distance function to $\Gamma_0 $ near $\Gamma _0$,  then  $\frac{u}{\d}\in H^{k-1}( \Omega_0 )$ and 
\begin{equation}
\label{Hardy}
\left\|\frac{u}{\d}\right\|_{k-1}\le C \|u\|_k \,.
\end{equation}
\end{lemma}

We let $  \mathbb{R}  ^d_+ = \{ x \in \mathbb{R}  ^d \ : \ x_d > 0\}$, and for $ s \in \mathbb{R}  $, $[s]$ denotes the greatest integer less than or equal to $s$.
For $s>0$ and $s\notin \mathbb{N}  $, an equivalent norm on $H^s(\mathbb{R}  ^d_+)$ is given by
\begin{equation}\label{hsnorm}
\|u\|^2_{H^s( \mathbb{R}  ^d_+)} := \|u\|^2_{H^{[s]}(\mathbb{R}  ^d_+)} + \sum_{|\alpha| = [s]} \iint_{\mathbb{R}  ^d_+\times \mathbb{R}  ^d_+} \frac{|\nabla ^\alpha u(x) - \nabla ^\alpha u(y)|^2}{|x-y|^{d+2(s-[s])}}\hspace{1pt} \,dx dy \,.
\end{equation} 
For functions $u \in H^s( \mathbb{R}  ^d)$ one replaces the integrals over $ \mathbb{R}  ^d_+$ with integrals over $\mathbb{R}  ^d$.

For any smooth bounded domain $\Omega_0$, there exists a finite cover by open sets $\{ \mathcal{U} _i\}_{i=1}^K$ and a partition-of-unity $\{\zeta_i\}_{i=1}^K$ subordinate to this
open cover.  Each set $ \mathcal{U} _i$ which covers $\Gamma_0$ is diffeomorphic to the upper half of the unit ball $B^+$ in $ \mathbb{R}  ^3$ via the
chart $\theta_i: B^+ \to \mathcal{U} _i$, while interior sets $ \mathcal{U} _i$ are diffeomorphic to the unit ball $B$ via $\theta_i: B \to \mathcal{U} _i$.  Then for
$ u \in H^s(\Omega_0)$,
$\|u\|_s^2 = \sum_{i=1}^K \| \zeta_i\, u \circ \theta_i \|^2_{H^s( \mathbb{R}  ^d_+)}$,
where the $H^s( \mathbb{R}  ^d_+)$-norm is replaced by the $H^s( \mathbb{R}  ^d)$-norm if $ \mathcal{U} _i$ is an interior set.

\section{Global existence for 1-d compressible Euler equations}\label{sec:1d}
We include the analysis of flows in one space dimension as a {\it User's Guide} for the multi-dimensional setting to follow.     In one space dimension, the case
that $ 1< \gamma \le 3$ does not produce an ``anti-damping'' term in the sense of \cite{HaJa2016} and is thus in the same analysis regime as the three-dimensional
case of $ 1 < \gamma \le {\frac{5}{3}} $.   We shall explain how to treat all $\gamma > 1$.
\subsection{Eulerian free-boundary formulation}
In one space dimension, the isentropic compressible Euler equations \eqref{ceuler3d} can is written as
\begin{subequations}
  \label{ceuler}
\begin{alignat}{2}
\rho ( \p_t u + u \p_x u) + \p_x  p(\rho) &=0  &&\text{in} \ \ \Omega  (t) \,, \\
\p_t \rho+ \p_x (\rho u)&=0 
&&\text{in} \ \Omega (t) \,, \\
p &= 0 \ \ &&\text{on} \ \ \Gamma(t) \,,  \\
\mathcal{V} (\Gamma(t))& = u  && \\
(\rho,u,\Omega)   &= (\rho_0,u_0, \Omega_0 ) \ \  &&\text{on} \ \{t=0\}\,.
\end{alignat}
\end{subequations}
For our 1d analysis,
we shall use both $\p_x u$ and $u_x$ to denote the $x-$derivative of a function $u$.
\subsection{Lagrangian formulation}
\subsubsection{The Euler equations}
 We transform the system (\ref{ceuler}) into
Lagrangian variables.
We let $\xi(x,t)$ denote the ``position'' of the gas particle $x$ at time $t$.  Thus,
\begin{equation}
\nonumber
\begin{array}{c}
\partial_t \xi = u \circ \xi $ for $ t>0 $ and $
\xi(x,0)=x
\end{array}
\end{equation}
where $\circ $ denotes composition so that
$[u \circ \xi] (x,t):= u(\xi(x,t),t)\,.$  

%
%
As is shown in \cite{CoSh2011}, 
$\rho \circ \xi = \rho_0 \, \xi_x ^{-1}$; hence,
 the initial density function $\rho_0$ can be viewed as a parameter, and the compressible Euler equations can be written as a degenerate second-order
 wave equation for $\xi$:
\begin{subequations}
\label{ce0}
\begin{alignat}{2}
\rho_0 \p^2_t \xi +  \p_x (\rho_0^ \gamma \, \xi_x^{-\gamma}) &=0 \ \ && \text{ in } \Omega_0   \times (0,T] \,, \\
(\xi, \p_t\xi)  &=( \xi _0, \xi_1) \ \  \ \ && \text{ in } \Omega_0   \times \{t=0\} \,,  
\end{alignat}
\end{subequations}
with $\xi_0$ and $\xi_1$ denoting, respectively, the initial position and velocity on $\Omega_0$.   The initial density function must satisfy
 $ \rho _0^{ \gamma -1}(x) \ge C \dd( x, \Gamma_0 ) $ for $x \in \Omega_0$ near $\Gamma_0$.    We  let $\d(x)$ denote a particular
 distance function  to be defined below (and associated with our choice of affine solutions), and from  \eqref{degen}, 
 we set
$$
\rho_0(x) =\d^{\frac{1}{\gamma-1}} \,,
$$
so that \eqref{ce0} becomes
$$
\d^{\frac{1}{\gamma-1}}  \p^2_t \xi +  \p_x  \left(\d^{\frac{\gamma}{\gamma-1}} \xi_x^{-\gamma}\right) =0 \ \  \text{ in } \Omega_0  \times (0,T] \,,
$$
which is equivalent to
\begin{equation}\label{ce}
\p^2_t \xi + \frac{\gamma}{\gamma -1} \dx \zxg  -\gamma \d \zxgo \p_x^2\xi =0  \,.
\end{equation}

\subsubsection{Perturbations of affine solutions}   With the affine solution given by $\A(t)x=\alpha(t) x$, we shall denote by $\eta(x,t)$ the perturbation to the
affine flow, and consider
solutions to \eqref{ce0} of the form
\begin{equation}
\xi( x,t) = \alpha(t) \eta(x,t) \,. \nonumber
\end{equation} 
We define the velocity $v$ associated  to the perturbation $\eta$ by
$v(x,t) = \p_t \eta(x,t)$.
It follows that
$$
\p_t \xi(x, t) = \dot \alpha(t) \eta(x,t) + \alpha(t) \p_t \eta(x,t) \,.
$$
For simplicity, let us suppose
 that  at time $t=0$, $\eta(x,0) =e(x)$,  the identity map on $\Omega_0$, and that $\p_t\eta(x,0) = u_0(x)$; furthermore, we let $\alpha(0)=1$.  Then,
$$
\p_t \xi(x, 0) = \dot \alpha(0) x + u_0(x)\,.
$$
In order to analyze the stability of the affine solution, we shall assume that $u_0(x)$ is a small perturbation of the initial affine velocity $ \dot \alpha(0) x$.

Using the affine solution  \eqref{affine1d}, we see that the pair $(\eta,v)$ satisfies the following degenerate  and nonlinear wave equation:
\begin{subequations}
\label{ceuler0}
\begin{alignat}{2}
\p_t \eta &=v \ \ && \text{ in } \Omega_0   \times (0,T] \,, \\
\alpha^\gpo \p_t v + 2\alpha ^\gamma \dot\alpha v + \eta + \frac{\gamma}{\gamma-1} \dx \nxg  -\gamma \d \nxgo  \p^2_x \eta&=0 \ \ && \text{ in } \Omega_0   \times (0,T] \,, \\
(\eta, v)  &=( e, u_0) \ \  \ \ && \text{ in } \Omega_0   \times \{t=0\} \,,  
\end{alignat}
\end{subequations}
where $e(x)=x$ denotes the identity map, and 
$$\Omega_0 = (-1,1) \  \text{ and } \  \d(x)= {\frac{\gamma -1}{2\gamma}} (1-x^2) \ \text{ on } \overline \Omega_0 \,. $$
Note that the physical vacuum condition  \eqref{degen} is satisfied as  $|\d_x| = {\frac{\gamma-1}{\gamma}} $ on $\Gamma_0 = \{ -1, 1\}$ and that 
\begin{equation}\label{dxx}
|\d_x| \le {\frac{\gamma-1}{\gamma}} \ \text{ and } \ 
\d_{xx} = -{\frac{\gamma-1}{\gamma}}  \  \text{ on } \overline \Omega_0 \,.
\end{equation} 

For our analysis, we shall assume that
\begin{equation}\label{eta_bound}
\|  \eta_x(\cdot ,t) -1 \|_{L^\infty }  \le {\frac{1}{10}}   \ \text{ for all } t \ge 0 \,,
\end{equation} 
and prove later that  $\|  \eta_x(\cdot ,t) -1 \|_{L^\infty }  \le \epsilon  $ for all $t \ge 0$ and for $ 0 < \epsilon \ll {\frac{1}{10}} $.  The bound \eqref{eta_bound} 
implies that
\begin{equation}\label{ss0}
1 - {\frac{1}{10}} \le \eta_x (x, t) \le 1 + {\frac{1}{10}} \,,
\end{equation} 
and hence that
\begin{equation}\label{goodbound}
\left[1 + {\frac{1}{10}}\right]^{-1}  \le \eta_x^{-1}  (x, t) \le \left[1 - {\frac{1}{10}}\right]^{-1}  \,.
\end{equation}

\subsection{Smoothing the initial data}\label{subsec::initdata} To obtain energy estimates, we will work with a smooth sequence of solutions, obtained from smoothing
the initial data.

For 
$\kappa>0$, 
let $0 \le \varrho_ \kappa \in C^ \infty _0( \mathbb{R}^d )$ denote the standard family of mollifiers with $\spt (\varrho_ \kappa )
\subset \overline{ B(0, \kappa)}$,  and   let $\mathcal{E}_{\Omega_0}  $ denote a
Sobolev extension operator mapping $ H^s(\Omega_0)$ to $H^s( \mathbb{R}^d  )$ for $ s\ge 0$. 

With $0<\kappa_0\ll 1$ chosen small, 
for $\kappa\in (0,\kappa_0)$, we set
\begin{equation} 
u_0^ \kappa = \varrho_{ \frac{1}{|\text{ln} \kappa|}} *\mathcal{E}_ {\Omega_0} (u_0)\,,   \label{u0kappa}
\end{equation} 
so that for any fixed $\kappa \in (0, \kappa_0 )$, $u_0^ \kappa \in C^ \infty (\overline 
\Omega )$.  More importantly, due to the standard properties of  convolution, we also have the following estimates:
\begin{equation}
\nonumber
\forall s\ge 1\,,\forall\kappa\in(0,\kappa_0)\,, \ \ \ \|u_0^\kappa\|_s\le C_s |\text{ln}\kappa|^s \|u_0\|_0\,.
\end{equation}

\subsubsection{Global existence for $ \gamma=2$}
Because of the particular interest in the shallow water equations (see \cite{LaMe2017}) and the fact that $\gamma < 3$ does not produce a so-called
``anti-damping'' term, 
we shall first explain the proof of global existence in the case that $\gamma=2$; the proof for  general $\gamma >1$ will be given below.
The momentum equation (\ref{ceuler0}b) takes the form 
\begin{equation}\label{euler-gamma2}
\alpha^3 \p_t v + 2\alpha ^2 \dot\alpha v + \eta +  2\dx \eta_x^{-2} - 2 \d \eta_x^{-3}   \p^2_x \eta=0 \ \ \text{ in } \Omega_0   \times (0,T] \,.
\end{equation} 
We define the near-identity map
$$
\deta(x,t) = \eta(x,t) -x \,.
$$
\begin{definition}[Norm for $\gamma=2$ in 1-d]
We use the following higher-order energy function for case that $\gamma=2$:
\begin{align*}
e_2(t) = \| \d^ {\frac{7}{2}} \p_x^6 \deta( \cdot , t)\|_0^2 + \sum_{b=0}^5 \left(
 \| \alpha(t) ^ {\frac{3}{2}}  \d^{\frac{1+b}{2}}  \p_x^b v(\cdot , t)\|_0^2 +  \|  \d^{\frac{1+b}{2}}  \p_x^b \deta(\cdot , t)\|_0^2
\right) \,.
 \end{align*} 
The norm is given by
\begin{equation}\nonumber
E_2(t) =
e_2(t) +  \| \alpha^ {\frac{3}{2}}(t)  v( \cdot ,t) \|_{W^ {1,\infty} }^2  + \|   \deta( \cdot ,t) \|_{W^ {1,\infty} }^2
+  \sum_{b=0}^3 \left( \| \alpha(t) ^ {\frac{3}{2}} \d^b  v( \cdot ,t) \|_{2+b}^2  + \|\d^b  \deta( \cdot ,t) \|_{2+b}^2\right) \,.
\end{equation} 
\end{definition} 
\noindent
We remark that it suffices to use $e_2(t)$ for the analysis, but including all of the terms in $E_2(t)$ allows for a somewhat more transparent approach to controlling
the error terms in energy estimates.

From  \eqref{goodbound}, we have that
\begin{equation}   \label{g2bound}
\frac{1}{2} \le \eta_x ^{-3} (x,t)  \ \text{ and } \ \eta_x ^{-4} (x,t)  \le 2 \ \text{ for } t \ge 0 \,,
\end{equation} 
and we shall make use of these bounds in energy estimates.

\begin{theorem}  For $\gamma=2$ and  $ \epsilon >0$ taken sufficiently small, if the initial data satisfies
 $e_2(0) <  \epsilon $, then there exists a global unique solution of the 1-d Euler equations \eqref{ceuler},  such that $E_2(t) \le \epsilon $ for all $t \ge 0$.
 Furthermore, $\xi( x,t) = \alpha(t) \eta(x,t) $ is a global solution to the 1-d Euler equations \eqref{ceuler}.
\end{theorem} 
\begin{proof}
{\it Step 1. A sequence of smooth  solutions.}
We consider initial conditions (\ref{ceuler0}c) given by
\begin{equation}\label{smooth1}
(\eta, v)  =( e, u_0^\kappa) \ \  \ \  \text{ in } \Omega_0   \times \{t=0\}   \,,
\end{equation} 
where $u_0^\kappa$ is given by \eqref{u0kappa}.  Notice that \eqref{ceuler0} is equivalent to  \eqref{ce0};  by the local-in-time well-posedness theorem 
 in \cite{CoSh2011}, using the smooth initial data \eqref{smooth1}, there exist a smooth solution $(\eta_\kappa, v_\kappa)$ on a short time-interval $[0,T]$
 (where $T$ depends on $\kappa$),
  such that for each $t \in [0,T]$,
 $\eta_\kappa( \cdot , t) \in H^7(\Omega_0)$,  $v_\kappa( \cdot , t) \in H^6(\Omega_0)$, and $\p_t v_\kappa( \cdot , t) \in H^5(\Omega_0)$.   We are thus able to consider higher-order energy estimates, as we can rigorously differentiate the sequence $(\eta_\kappa, v_\kappa)$ as many time as required.   For notational
 clarity, we will drop the subscript $\kappa$ and simply write $\eta$ and $v$.

\vspace{.05in}
\noindent
{\it Step 2. Energy estimates.}
We let $\p_x^5$ act on \eqref{euler-gamma2}, and letting $v_t:= \p_t v$, we find that
\begin{equation}\label{px5}
\alpha^3 \p_x^5 v_t + 2\alpha ^2 \dot\alpha  \p_x^5v +  (1+ 14 \eta_x^{-3})\p_x^5 \eta  - 2 \p_x \left( 6 \d_x  \eta_x^{-3} \p_x^5 \eta + \d  \eta_x^{-3} \p_x^6 \eta  \right) 
=\R_0+\R_1+\R_2+ \R_3  \,,
\end{equation} 
where the lower-order  
remainder terms for the fifth space differentiated problem are given by
\begin{align*} 
\R_0 & = 42 \eta_x^{-4} \p_x^2 \eta \, \p_x^4 \eta \,, \\
\R_1 & = -6 \p_x \left(  5 \d_x \eta_x^{-4} \p_x^2 \eta \p_x^4 \eta 
 + \d  \eta_x^{-4}\p_x^2 \eta \p_x^5 \eta \right)  + 27 \p_x \left(\eta_x^{-4} \p_x^2\eta \p_x^3 \eta\right)\,,\\
\R_2 & = -6 \p_x^2 \left(  
4\d_x \eta_x^{-4} \p_x^2 \eta \p_x^3 \eta 
+ \d  \eta_x^{-4}\p_x^2 \eta \p_x^4 \eta \right)  + 15\p_x^2\left(\eta_x^{-4} \p_x^2 \eta \p_x^2 \eta\right) \,, \\
\R_3 & = - 6\p_x^3 \left( 
3 \d \eta_x^{-4} \p_x^2 \eta \p_x^2 \eta + \d \eta_x^{-4}\p_x^2 \eta \p_x^3 \eta \right)  \,.
\end{align*}  
H\"{o}lder's inequality  and the Sobolev embedding theorem show that
\begin{equation}\nonumber
\| \d^3 (\R_0+ \R_1+\R_2+ \R_3)\|_0 \lesssim E_2(t) + \P(E_2(t)) \,.
\end{equation} 
We compute the $L^2$ inner-produce of equation \eqref{px5} with $\d^6 \p_x^5v$, and use that 
$$\p_x (\d^6 \p_x^5 v) = \d^5 (6 \d_x \p_x^5 v + \d \p_x^6 v) $$
to obtain the following identity:
\begin{align*} 
& \frac{d}{dt} \left( \| \alpha ^ {\frac{3}{2}} \d^3 \p_x^5 v\|_0^2 +  \int_{\Omega_0} (1+ 14 \eta_x^{-3}) \d^6 |\p_x^5 \eta|^2 dx \right) 
+ 4 \dot\alpha \|  \alpha \d^2 \p_x^4 v\|_0^2 \\
& \qquad 
+ 2 \frac{d}{dt} \int_{\Omega_0} \d^5 \eta_x^{-3} \left( 6\d_x \p_x^5 \eta + \d \p_x^6\eta\right)^2 dx
+  42\int_{\Omega_0}  \eta_x^{-4} v_x \d^6 |\p_x^5 \eta|^2 dx \\
& \qquad  + 6 \int_{\Omega_0} \d^5 \eta_x^{-4}  v_x\left( 6\d_x \p_x^5 \eta + \d \p_x^6\eta\right)^2 dx 
= 2 \int_{\Omega_0} \d^3  ( \R_0+\R_1+\R_2+ \R_3) \, \d^2 \p_x^5 v\, dx  \,,
\end{align*} 
and integrating in time, noting that $\dot\alpha(t) \ge 0$, and using 
 \eqref{g2bound}, we see that
\begin{align*} 
& \| \alpha ^ {\frac{3}{2}} \d^3 \p_x^5 v( \cdot , t) \|_0^2 + 8 \| \d^3 \p_x^5 \eta( \cdot , t)\|_0^2 
+   \int_{\Omega_0}  \left( 6\d_x \d^{2.5} \p_x^5 \eta + \d^{3.5} \p_x^6\eta\right)^2 dx \\
&  \qquad \lesssim  \|  \d^2 \p_x^4 u_0 \|_0^2
+   \int_0^t \|v_x\|_{L^\infty} \int_{\Omega_0} \left( 3 \d^{2.5}\d_x \p_x^5 \eta + \d^{3.5} \p_x^6\eta\right)^2 dx  ds \\
& \qquad \qquad 
+   \int_0^t \|v\|_{L^ \infty }  \int_{\Omega_0}  \d^6 |\p_x^5 \eta|^2 dxds
+  \int_0^t \int_{\Omega_0} \d^3  \left|(\R_0+ \R_1+\R_2+ \R_3) \, \d^3 \p_x^5 v\right|\, dx ds \,,
\end{align*} 
where we have used the fact that $\p_x^a \eta(x,0)=0$ for $a \ge 2$.

As $(\eta, v)$ are sufficiently smooth (thanks to Step 1), we can integrate-by-parts to find that
\begin{align*} 
 \int_{\Omega_0}  \left( 6\d_x \d^{2.5} \p_x^5 \eta + \d^{3.5} \p_x^6\eta\right)^2 dx  
 & = 
  \int_{\Omega_0}  \left( 36\d_x^2 \d^{5} |\p_x^5 \eta|^2 + 6 \d^6 \d_x  \p_x(|\p_x^5\eta|^2)+ \d^{7}| \p_x^6\eta| ^2\right) dx \\
  & = 
  \int_{\Omega_0}  \left( -6 \d^6 \d_{xx}|\p_x^5\eta|^2  + \d^{7}| \p_x^6\eta| ^2\right) dx \\
  & = 
  \int_{\Omega_0}  \left( 3 \d^6 |\p_x^5\eta|^2  + \d^{7}| \p_x^6\eta| ^2\right) dx \,,
\end{align*} 
the last equality following from \eqref{dxx}. Hence, since $\|v_x\|_{L^ \infty } \le 2 \|v_x\|_1$ on $\Omega_0$, 
\begin{align} 
& \| \alpha ^ {\frac{3}{2}} \d^3 \p_x^5 v( \cdot , t) \|_0^2 + \| \d^3 \p_x^5 \deta( \cdot , t)\|_0^2 + \| \d^{3.5} \p_x^6 \deta( \cdot , t)\|_0^2 
\nonumber \\
&  \lesssim  e_2(0)+   \int_0^t \alpha ^ {-\frac{3}{2}} \| \alpha ^ {\frac{3}{2}} v\|_{2}\int_{\Omega_0}  \left(   |\d^3\p_x^5\deta|^2  + | \d^{3.5}\p_x^6\deta| ^2\right) dx ds 
\nonumber\\
& \qquad \qquad \qquad \qquad \qquad 
+  \int_0^t \alpha ^ {-\frac{3}{2}} \|\d^3  (\R_0+ \R_1+\R_2+ \R_3) \|_0 \,    \| \alpha ^ {\frac{3}{2}} \d^3 \p_x^5 v\|_0 ds  \,,
\nonumber\\
&  \lesssim e_2(0)+ \sup_{s \in [0,t]} P(E_2(s))  \int_0^t \alpha ^ {-\frac{3}{2}}ds  \lesssim e_2(0) + \sup_{s \in [0,t]} P(E_2(s)) \,, \label{g25}
\end{align} 
where we have used the fact that $\int_0^ t \alpha ^ {-\frac{3}{2}}(s)ds \le C <\infty$  uniformly for $t \ge 0$.

Next for $ b=2,3,4$, we let $\p_x^b$ act on \eqref{euler-gamma2} and compute the $L^2$ inner-product with $\d^{b+1} \p_x^b v$.
 Following identically the strategy given above
for the $\p_x^5$-problem, we find that
\begin{align} 
 \| \alpha ^ {\frac{3}{2}} \d^{2.5} \p_x^4 v( \cdot , t) \|_0^2 + \| \d^{2.5} \p_x^4 \deta( \cdot , t)\|_0^2 + \| \d^{3} \p_x^5 \deta( \cdot , t)\|_0^2 
 &\lesssim e_2(0) + \sup_{s \in [0,t]} P(E_2(s)) \,, \label{g24} \\
 \| \alpha ^ {\frac{3}{2}} \d^2 \p_x^3 v( \cdot , t) \|_0^2 + \| \d^2 \p_x^3 \deta( \cdot , t)\|_0^2 + \| \d^{2.5} \p_x^4 \deta( \cdot , t)\|_0^2 
 &\lesssim e_2(0)+ \sup_{s \in [0,t]} P(E_2(s)) \,, \label{g23} \\
  \| \alpha ^ {\frac{3}{2}} \d^{1.5} \p_x^2 v( \cdot , t) \|_0^2 + \| \d^{1.5} \p_x^2 \deta( \cdot , t)\|_0^2 + \| \d^{2} \p_x^3 \deta( \cdot , t)\|_0^2 
 &\lesssim e_2(0) + \sup_{s \in [0,t]} P(E_2(s)) \,. \label{g22}
 \end{align}

 We next consider the $ \p_x^b$-problems with $b=0$ and $1$.
Since
 $$
 \eta_x^{-2} = \left( 1+ \deta_x\right)^{-2} = 1 - 2 \deta_x  + \P(\deta_x) \,,
 $$
the momentum equation
 \eqref{euler-gamma2} can be written as
 \begin{align} 
 0&= 
 \alpha^3 \d\p_t v + 2\alpha ^2 \dot\alpha \d v + \d\, \deta + \d \, x+\p_x \left[ \d^2 (1-2 \deta_x) \right] +\p_x \left[ \d^2 \P(\deta_x)\right] \nonumber  \\
 &=
  \alpha^3 \d\p_t v + 2\alpha ^2 \dot\alpha \d v + \d\,\deta + \d \, x  + 2\d \, \d_x-2 \p_x (\d^2 \deta_x )+\p_x \left[ \d^2 \P(\deta_x)\right] \nonumber\\
   &=
  \alpha^3 \d\p_t v + 2\alpha ^2 \dot\alpha \d v + \d\, \deta -2 \p_x (\d^2 \deta_x )+\p_x \left[ \d^2 \P(\deta_x)\right] \,, \nonumber
\end{align} 
 where we have used the fact that $\d= {\frac{1}{4}} (1-x^2)$ and so $ x + 2\d_x =0$.   We then compute the $L^2$ inner-product of  this equation with 
 $v= \p_t \deta$ to
 find that
 \begin{align*} 
\frac{d}{dt}  \int_{\Omega_0} \left[ \alpha^3 \d v^2  + \d\, \deta^2 + 2\d^2 \deta_x^2 \right] dx      +   4\int_{\Omega_0} \dot\alpha \alpha^ 2  \d v^2 dx 
= \int_{\Omega_0} \d^2 \P(\deta_x) v_x dx \,.
\end{align*} 
Integrating from $0$ to $t$ and using the fact that both $\deta$ and $\deta_x$ vanish at $t=0$, we have that
\begin{align} 
&\| \alpha^{\frac{3}{2}}  \d^ {\frac{1}{2}}  v(\cdot ,t)\|_0^2   + \|\d^ {\frac{1}{2}} \, \deta( \cdot ,t) \|_0^2  + \| \d\,  \deta_x^2(\cdot ,t) \|_0^2 \nonumber   \\
 &\qquad \le \|  \d^ {\frac{1}{2}}  u_0\|_0^2    + \int_0^t \alpha^{- \frac{3}{2}}(s)  \int_{\Omega_0} \d^2 \P(\deta_x)  \alpha^{\frac{3}{2}} (s) v_x dx \, ds
  \lesssim e_2(0) + \sup_{s \in [0,t]} \P(E_2(s)) \,. \label{g2p0}
\end{align} 
 
For the $\p_x$-problem, we write
 \eqref{euler-gamma2} as
\begin{align*} 
   \alpha^3 \p_t v + 2\alpha ^2 \dot\alpha v + \deta -4\d_x\, \deta - 2\d\, \deta_x
   +\p_x \left[ \d^2 \P(\deta_x)\right]  =0
\end{align*} 
and differentiate to find that
 $$
\alpha^3 \p_x v_t + 2\alpha ^2 \dot\alpha \p_xv + \p_x\deta -2\p_x\left[ 2\d_x\, \deta + \d\, \deta_x \right]
+\p_x \left[ \d^2 \P(\deta_x)\right]  =0 \,.
 $$
Computing the $L^2$ inner-product  of this equation with $\d^2  \p_x v$ shows that
\begin{align*} 
\frac{d}{dt}  \int_{\Omega_0} \left[ \alpha^3 \d^2 v_x^2  + \d^2 \deta_x^2 + 2\d\left(2\d_x\, \deta + \d\, \deta_x\right)^2 \right] dx      +   4\int_{\Omega_0} \dot\alpha \alpha^ 2  \d^2 v_x^2 dx 
= \int_{\Omega_0} \d^2 \P(\deta_x) v_x dx \,.
\end{align*} 
Integrating from $0$ to $t$, the same argument used to obtain \eqref{g2p0} provides the bound
\begin{align} 
\| \alpha^{\frac{3}{2}}  \d  \p_x v(\cdot ,t)\|_0^2 + \|  \d  \p_x \deta(\cdot ,t)\|_0^2
  \lesssim e_2(0) + \sup_{s \in [0,t]} \P(E_2(s)) \,.\label{g2p1}
  \end{align}

 \vspace{.05in}
\noindent
{\it Step 3. Building the higher-order norm $E_2(t)$.} From \eqref{g25}--\eqref{g2p1}, we have that
$e_2(t)  \lesssim e_2(0) + \sup_{s \in [0,t]} P(E_2(s))$.  Thus,
$$
\alpha^3(t) \sum_{ a=0}^5\int_{\Omega_0} \d^6  |\p_x^a v(x, t)|^2 dx \lesssim e_2(0) + \sup_{s \in [0,t]} P(E_2(s)) \,,
$$
and hence by the embedding \eqref{w-embed},
$$
\alpha^3(t) \left(\|v( \cdot , t)\|_2^2 + \sum_{ a=0}^3\int_{\Omega_0} \d^2  |\p_x^a v(x, t)|^2 dx 
+ \sum_{ a=0}^4\int_{\Omega_0} \d^4  |\p_x^a v(x, t)|^2 dx
\right)\lesssim e_2(0) + \sup_{s \in [0,t]} P(E_2(s)) \,.
$$
It follows that
\begin{equation}\label{v-energy1}
\sup_{s \in [0,t]}  \sum_{b=0}^3 \| \alpha^{\frac{3}{2}}\d^bv( \cdot , s)\|_{2+b}^2 
 \lesssim e_2(0)+ \sup_{s \in [0,t]} P(E_2(s)) \,.
\end{equation} 

Similarly,  the embedding estimate \eqref{w-embed} also shows that
\begin{equation}\label{eta-energy1}
\sup_{s \in [0,t]} 
\left(\| \d^{\frac{7}{2}}  \p_x^6 \deta( \cdot , t)\|_0^2
+
 \sum_{b=0}^3 \| \d^b \deta( \cdot , s)\|_{2+b}^2  
\right)  \lesssim e_2(0)+ \sup_{s \in [0,t]} P(E_2(s)) \,.
\end{equation} 
Then, from the Sobolev embedding theorem, 
\begin{equation}\label{ev-energy}
\sup_{s \in [0,t]} 
\left(\| \deta( \cdot , t)\|_{ W^{1, \infty }}^2
+
\|\alpha^{\frac{3}{2}} v( \cdot , t)\|_{ W^{1, \infty }}^2
\right)  \lesssim e_2(0)+ \sup_{s \in [0,t]} P(E_2(s)) \,.
\end{equation} 

 \vspace{.05in}
\noindent
{\it Step 4. Bound for $E_2(t)$ and global existence.}  The inequalities \eqref{v-energy1}--\eqref{ev-energy} show that
$$
\sup_{s \in [0,t]} E_2(s)
  \lesssim e_2(0) + \sup_{s \in [0,t]} P(E_2(s)) \text{ for } t \in [0,T]\,,
$$
where $T$ is independent of $ \kappa$, since $e_2(0)$ is  independent of $\kappa$.   By assumption $e_2(0) \le \epsilon$.
By choosing $ \epsilon $ sufficiently small, we have that $ E_2(t) \lesssim \epsilon$ for all $t \in [0,T]$.  Then, by the standard continuation argument,
\begin{equation}\nonumber
 E_2(t) \lesssim \epsilon \ \ \text{ for all } t \ge 0\,.
\end{equation} 
Hence, by the  Sobolev embedding theorem, $\| \eta_x -1\|_{L^ \infty } \lesssim \epsilon $ so that  setting $\varepsilon = C\, \epsilon $, 
$$
1- \varepsilon \le \eta_x(x,t) \le 1+ \varepsilon  \  \text{ for } \ t \ge 0\,,
$$
thus verifying \eqref{ss0}.

 \vspace{.05in}
\noindent
{\it Step 5. The limit as $\kappa \to 0$.} In order to produce our energy bounds, we have used the 
smooth sequence $(\eta_\kappa, v_\kappa)$.  Since we have established that
$
\| \alpha(t) ^ {\frac{3}{2}} \d^3 v_\kappa( \cdot ,t) \|_5^2 +  \| \d^3 (\eta_\kappa(\cdot ,t) -x) \|_5^2 \lesssim \epsilon $ with
$ \epsilon $ independent of $\kappa$, we have compactness, and it is a standard argument to show that 
$\eta_k \to \eta$ in  $C^2(\overline\Omega_0 \times [0,T])$ for any $T < \infty $, where $\eta$ is a solution of \eqref{ceuler0}, and
$E_2(t) \lesssim \epsilon $.
 \end{proof}

 \subsubsection{Global existence for $ \gamma>3$}\label{sec:1dgamma}
We now explain how to treat the case where $\gamma>3$.
The momentum equation (\ref{ceuler0}b) takes the form 
\begin{equation}\label{euler-gamma4a}
\alpha^\gpo \p_t v + 2\alpha ^\gamma \dot\alpha v + \eta + \frac{\gamma}{\gamma-1} \dx \nxg  -\gamma \d \nxgo  \p^2_x \eta =0 \ \ \text{ in } \Omega_0   \times (0,T] \,.
\end{equation} 
and it appears that energy estimates might
 produce an ``anti-damping'' term,  which is explained by focusing on only the first two terms of the equation: $\alpha^{\gamma+1} \p_t v + 2\alpha ^\gamma \dot\alpha v$.
Computing the $L^2$ inner-product with $v$ shows that
$$
{\frac{\alpha ^{\gamma+1}}{2}}  \frac{d}{dt} \|v\|_0^2 +  2\alpha ^{\gamma} \dot\alpha \|v\|_0^2 = {\frac{1}{2}}\frac{d}{dt} \left(\alpha ^{\gamma+1} \|v\|_0^2 \right) 
+ \left(2- \frac{\gamma+1}{2}\right) \alpha ^\gamma \dot\alpha \|v\|_0^2 \,.
$$
If $\gamma>3$, then $ 2 - \frac{\gamma +1}{2} < 0$ which
 produces the so-called anti-damping effect described in \cite{HaJa2016}.   In order to avoid such anti-damping, it is 
necessary that
$ 2 - \frac{\gamma +1}{2} \ge 0$ which only holds for $\gamma \le 3$ in 1-d (and  for $\gamma \le\frac{5}{3} $ in 3-d).     

Thus, in order to close energy estimates when $\gamma>3$ we shall divide (\ref{ceuler0}b) by
 $ \alpha ^{\gamma -3}$.
We are thus lead to consider instead the equation
 \begin{equation}\label{euler-gamma4}
 \alpha^4 \p_t v + 2\alpha ^3 \dot\alpha v + \agt \eta + \frac{\gamma}{\gamma-1} \agt \dx \nxg  -\gamma \agt \d \nxgo  \p^2_x \eta =0
  \ \ \text{ in } \Omega_0   \times (0,T] \,.
\end{equation} 

Continuing to denote the near-identity map $ \deta(x,t) = \eta(x,t) -x$.
\begin{definition}[Norm for $\gamma >3$ in 1-d] The higher-order energy function for 
  $\gamma>3$ is given by
\begin{align}
e_\gamma(t) &= 
\| \agtt \d^{\frac{5\gamma-4}{2\gamma-2}}  \p_x^5 \deta( \cdot , t)\|_0^2
+ \sum_{a=0}^4 \left(  \| \alpha(t)^2 \d^{\frac{1+a(\gamma-1)}{2\gamma -2}}
\p_x^a v( \cdot ,t) \|_0^2 +  \| \d^{\frac{1+a(\gamma-1)}{2\gamma -2}}\p_x^a \deta( \cdot ,t) \|_0^2  \right)
\,, \nonumber
\end{align} 
and we set
$$
E_\gamma(t) = e_\gamma(t) + \|\deta(\cdot , t)\|^2_{W^{1, \infty }}+\|\alpha^2(t) v(\cdot , t)\|^2_{W^{1, \infty }}+ \sum_{b=0}^2 \left( \|\alpha^2(t) \d^b v(\cdot , t)\|^2_{\frac{4\gamma-5}{2\gamma -2} + b}
+  \| \d^b \deta(\cdot , t)\|^2_{\frac{4\gamma-5}{2\gamma -2} + b}\right)\,.
$$
\end{definition}

From  \eqref{goodbound}, we have that
\begin{equation}   \label{g4bound}
\frac{1}{2} \le \eta_x ^{-\gamma-1} (x,t)  \ \text{ and } \ \eta_x ^{-\gamma-2} (x,t)  \le 2 \ \text{ for } t \ge 0 \,.
\end{equation}

\begin{theorem}  For $\gamma>3$ and  $ \epsilon >0$ taken sufficiently small, if the initial data satisfies
 $e_\gamma(0) <  \epsilon $, then there exists a global unique solution of the Euler equations \eqref{ceuler},  such that $E_\gamma(t) \le \epsilon $ for all $t \ge 0$.
Furthermore, $\xi( x,t) = \alpha(t) \eta(x,t) $ is a global solution to the 1-d Euler equations \eqref{ceuler}.
\end{theorem} 
\begin{proof}
{\it Step 1.}
Just as in the proof  for $\gamma=2$, we smooth the initial data.

\vspace{.05in}
\noindent
{\it Step 2. Energy estimates.}
We let $\p_x^4$ act on \eqref{euler-gamma4}, and letting $v_t:= \p_t v$, we find that
\begin{align}
&\alpha^4 \p_x^4 v_t + 2\alpha ^3 \dot\alpha  \p_x^4v + \agt [1+(6\gamma -3) \eta_x^{-\gamma-1}]\p_x^4 \eta   \nonumber \\
&\qquad\qquad\qquad
- \agt \gamma \p_x \left( \frac{4\gamma -3}{\gamma-1}  \d_x  \eta_x^{-\gamma -1} \p_x^4 \eta + \agt \d  \eta_x^{-\gamma -1} \p_x^5 \eta  \right) 
=\R_0+\R_1+\R_2+ \R_3  \,,  \label{px5ii}
\end{align} 
where the lower-order  remainder terms for the fourth space differentiated problem are given by
\begin{align*} 
\R_0 & =   - \agt (6 \gamma ^2 + 3 \gamma  -3) \eta_x^{ - \gamma -2} \p_x^2 \eta \p_x^3 \eta \,, \\
\R_1 & = -\agt  \gamma ( \gamma +1) \p_x \left(\frac{3\gamma -2}{ \gamma -1} \d_x  \eta_x^{-\gamma -2}  \p_x^2 \eta \p_x^3\eta  
+ \d  \eta_x^{-\gamma -2}  \p_x^2 \eta \p_x^4 \eta  \right) +(3 \gamma ^2 + \gamma + 1) \p_x\left( \eta_x^{ \gamma -2} \p_x^2 \eta \p_x^2 \eta \right)\,,\\
\R_2 & = -\agt \gamma ( \gamma +1) \p_x^2 \left(\frac{2\gamma -1}{\gamma-1} \d_x  \eta_x^{-\gamma-2}  \p_x^2 \eta \p_x^2\eta  + \d  \eta_x^{-\gamma -2}  \p_x^2 \eta \p_x^3 \eta \right)\,,\\
\R_3 & =  - \agt \gamma ( \gamma +1) \p_x^3 \left( \d  \eta_x^{-\gamma -2}  \p_x^2 \eta \p_x^2 \eta \right)\,,\\
  \,.
\end{align*}  

We compute the $L^2$ inner-produce of equation \eqref{px5ii} with $\d^{\frac{4\gamma-3}{\gamma -1}} \p_x^4v$, use the fact that 
$$\p_x (\d^{\frac{4\gamma-3}{\gamma -1}} \p_x^4 v) =\d^{\frac{3\gamma-2}{\gamma -1}}\left(\frac{4\gamma-3}{\gamma-1}  \d_x \p_x^4 v + \d \p_x^5 v\right) \,,$$
and with $ \dot \alpha(t) > 0$, we
obtain the following inequality:
\begin{align*} 
& \frac{d}{dt} \left( \| \alpha^2 \d^{\frac{4\gamma-3}{2\gamma -2}} \p_x^4 v\|_0^2 +  \agt  \int_{\Omega_0} \d^{\frac{4\gamma-3}{\gamma -1}} (1+ 21 \eta_x^{-5})  |\p_x^4 \eta|^2 dx \right)  \\
& \ 
+(\gamma+1) (6\gamma -3)
\agt  \int_{\Omega_0} \d^{\frac{4\gamma-3}{\gamma -1}} \eta_x^{-\gamma -2} \, v_x\,  |\p_x^4 \eta|^2 dx   
\\
& \ 
+ \gamma \frac{d}{dt} \agt \int_{\Omega_0}\d^{\frac{3\gamma-2}{\gamma -1}}\eta_x^{-5} \left( \frac{4\gamma-3}{\gamma-1} \d_x \p_x^4 \eta + \d \p_x^5\eta\right)^2 dx
\\
& \   + \gamma(\gamma+1) \int_{\Omega_0}\d^{\frac{3\gamma-2}{\gamma -1}}\eta_x^{-6}  v_x\left( \frac{4\gamma-3}{\gamma-1} \d_x \p_x^4 \eta + \d \p_x^5\eta\right)^2 dx 
\le 2 \int_{\Omega_0} \d^{\frac{4\gamma-3}{2\gamma -2}} \left| ( \R_0+\R_1+\R_2+ \R_3) \, \d^{\frac{4\gamma-3}{2\gamma -2}} \p_x^4 v\right|\, dx  \,.
\end{align*} 
Integrating in time and using \eqref{dxx} and  \eqref{g4bound}, we see that
\begin{align*} 
& \| \alpha ^2\d^{\frac{4\gamma-3}{2\gamma -2}} \p_x^4 v( \cdot , t) \|_0^2 
+  \| \agtt \d^{\frac{4\gamma-3}{2\gamma -2}} \p_x^4 \eta( \cdot , t)\|_0^2 
+ \agt  \int_{\Omega_0}  \left( \frac{4\gamma-3}{\gamma-1} \d_x \d^{\frac{3\gamma-2}{2\gamma-2}} \p_x^4 \eta + \d^{\frac{5\gamma-4}{2\gamma-2}} \p_x^5\eta\right)^2 dx \\
&  \qquad \lesssim  \|  \d^{\frac{4\gamma-3}{2\gamma -2}} \p_x^4 u_0 \|_0^2
+ \agt \int_0^t \|v_x\|_{L^\infty} \int_{\Omega_0}  \left( \frac{4\gamma-3}{\gamma-1} \d_x \d^{\frac{3\gamma-2}{2\gamma-2}} \p_x^4 \eta +  \d^{\frac{5\gamma-4}{2\gamma-2}} \p_x^5\eta\right)^2  dx  ds \\
& \qquad \qquad 
+  \agt \int_0^t \|v\|_{L^ \infty }  \int_{\Omega_0}  \d^{\frac{4\gamma-3}{\gamma -1}} |\p_x^4 \eta|^2 dxds
+   \int_0^t \int_{\Omega_0} \d^{\frac{4\gamma-3}{2\gamma -2}}  \left|(\R_0+ \R_1+\R_2+ \R_3) \, \d^{\frac{4\gamma-3}{2\gamma -2}} \p_x^4 v\right|\, dx ds \,,
\end{align*} 
where we have used the fact that $\p_x^a \eta(x,0)=0$ for $a \ge 2$.

Again, we use the fact that we are working with a smooth sequence $(\eta, v)$, and so  we can integrate-by-parts to find that
\begin{align*} 
 & \int_{\Omega_0}  \left( \frac{4\gamma-3}{\gamma-1} \d_x \d^{\frac{3\gamma-2}{2\gamma-2}} \p_x^4 \eta + \d^{\frac{5\gamma-4}{2\gamma-2}} \p_x^5\eta\right)^2 dx  \\
 & \qquad = 
  \int_{\Omega_0}  \left( {\frac{(4\gamma -3)^2}{(\gamma-1)^2}} \d_x^2\d^{\frac{3\gamma-2}{\gamma -1}}|\p_x^4 \eta|^2 
  + \frac{4\gamma-3}{\gamma-1}  \d^{\frac{4\gamma-3}{\gamma -1}} \d_x  \p_x(|\p_x^4\eta|^2)+ \d^{\frac{5\gamma -4}{\gamma-1}}| \p_x^5\eta| ^2\right) dx \\
  &\qquad = 
  \int_{\Omega_0}  \left( -\frac{4\gamma-3}{\gamma-1}  \d^{\frac{4\gamma-3}{\gamma -1}} \d_{xx}|\p_x^4\eta|^2  +  \d^{\frac{5\gamma -4}{\gamma-1}}| \p_x^5\eta| ^2\right) dx 
  \\
  & \qquad 
  = 
  \int_{\Omega_0}  \left( \frac{4\gamma-3}{\gamma} \d^{\frac{4\gamma-3}{\gamma -1}} |\p_x^4\eta|^2+ \d^{\frac{5\gamma-4}{\gamma-1}}| \p_x^5\eta| ^2\right) dx \,,
\end{align*} 
the last equality following from \eqref{dxx}. It follows that
\begin{align*} 
&\| \alpha ^2\d^{\frac{4\gamma-3}{2\gamma -2}} \p_x^4 v( \cdot , t) \|_0^2 +  \| \agtt \d^{\frac{4\gamma-3}{2\gamma -2}} \p_x^4 \eta( \cdot , t)\|_0^2  
+ \| \agtt \d^{\frac{5\gamma-4}{2\gamma-2}} \p_x^5 \eta( \cdot , t)\|_0^2 \\
&  \ \lesssim  \|  \d^{\frac{4\gamma-3}{2\gamma -2}} \p_x^4 u_0 \|_0^2
+   \int_0^t \|v_x\|_{L^\infty}   \left(  \|\agtt \d^{\frac{4\gamma-3}{2\gamma -2}} \p_x^4 \eta( \cdot , t)\|_0^2  + + \| \agtt \d^{\frac{5\gamma-4}{2\gamma-2}} \p_x^5 \eta( \cdot , t)\|_0^2\right)  ds \\
& \
+   \int_0^t \|v\|_{L^ \infty }  \int_{\Omega_0}  \d^{\frac{4\gamma-3}{\gamma -1}} |\agtt \p_x^4 \eta|^2 dxds
+   \int_0^t \int_{\Omega_0} \d^{\frac{4\gamma-3}{2\gamma -2}}  \left|(\R_0+ \R_1+\R_2+ \R_3) \, \d^{\frac{4\gamma-3}{2\gamma -2}} \p_x^4 v\right|\, dx ds \,.
\end{align*} 
Using the fact that $\int_0^t \alpha (s) ^{-2} ds\le C < \infty $ for all $t \ge 0$, we have that
\begin{align*} 
&  \int_0^t \|v_x\|_{L^\infty}   \left(  \|\agtt \d^{\frac{4\gamma-3}{2\gamma -2}} \p_x^4 \eta( \cdot , t)\|_0^2  +  \| \agtt \d^{\frac{5\gamma-4}{2\gamma-2}} \p_x^5 \eta( \cdot , t)\|_0^2\right)  ds  \\
& \ = \int_0^t\alpha ^{-2} \| \alpha ^2 v_x\|_{L^\infty}   \left(  \|\agtt \d^{\frac{4\gamma-3}{2\gamma -2}} \p_x^4 \eta( \cdot , t)\|_0^2  +  \| \agtt \d^{\frac{5\gamma-4}{2\gamma-2}} \p_x^5 \eta( \cdot , t)\|_0^2\right)  ds \lesssim \sup_{s \in [0,t]} P(E_\gamma(s)) \,,
\end{align*} 
and similarly that
\begin{align*} 
& \int_0^t \|v\|_{L^ \infty }  \int_{\Omega_0}  \d^{\frac{4\gamma-3}{\gamma -1}} |\agtt \p_x^4 \eta|^2 dxds 
 =  \int_0^t \alpha ^{-2} \|\alpha ^2 v\|_{L^ \infty }  \int_{\Omega_0}  \d^{\frac{4\gamma-3}{\gamma -1}} |\agtt \p_x^4 \eta|^2 dxds  \lesssim \sup_{s \in [0,t]} P(E_\gamma(s))\,.
\end{align*} 

Furthermore, since each $\R_a$, $a=0,1,2,3$ is at least quadratic in $E(t)$, H\"{o}lder's inequality  and the Sobolev embedding theorem shows that
\begin{equation}\nonumber
\| \d^{\frac{4\gamma-3}{2\gamma -2}} (\R_0+ \R_1+\R_2+ \R_3)\|_0 \le \agtt [E_\gamma(t)+ \P(E_\gamma(t))] \,,
\end{equation} 
since only $d^{ \frac{5\gamma-4}{2\gamma-2}} \p_x^5 \deta( \cdot , t)$ is weighted by $ \agtt(t)$ in $e_\gamma(t)$, while the lower-order derivatives of $\deta$ have no 
$\alpha$-time-weights.   Hence,
\begin{align*} 
& \int_0^t \int_{\Omega_0} \d^{\frac{4\gamma-3}{2\gamma -2}}  \left|(\R_0+ \R_1+\R_2+ \R_3) \, \d^{\frac{4\gamma-3}{2\gamma -2}} \p_x^4 v\right|\, dx ds 
  \le  \int_0^t  \alpha^{1-\gamma} \P(E_\gamma(s))  \| \alpha ^2 \d^{\frac{4\gamma-3}{2\gamma -2}} \p_x^4 v\|_0
  \lesssim \sup_{s \in [0,t]} P(E_\gamma(s) \,.
 \end{align*} 
 We have thus established that
 \begin{align} 
 \| \alpha ^2\d^{\frac{4\gamma-3}{2\gamma -2}} \p_x^4 v( \cdot , t) \|_0^2 +  \| \agtt \d^{\frac{4\gamma-3}{2\gamma -2}} \p_x^4 \deta( \cdot , t)\|_0^2  + \|\agtt d^{ \frac{5\gamma-4}{2\gamma-2}} \p_x^5 \deta( \cdot , t)\|_0^2
  \lesssim e_\gamma(0) + \sup_{s \in [0,t]} P(E_\gamma(s)) \,. \label{g44}
\end{align}

Next, we let $\p_x^3$ act on \eqref{euler-gamma4} and compute the $L^2$ inner-product with $\d^{(3\gamma-2)/(\gamma-1)} \p_x^3 v$ and then
let $\p_x^2$ act on \eqref{euler-gamma4} and compute the $L^2$ inner-product with $\d^{(2\gamma-1)/(\gamma-1)} \p_x^2 v$.  Following identically the strategy given above
for the $\p_x^4$-problem, we find that
\begin{align} 
 \| \alpha ^2\d^{\frac{3\gamma-2}{2\gamma-2}}  \p_x^3 v( \cdot , t) \|_0^2 +  \| \agtt \d^{\frac{3\gamma-2}{2\gamma-2}} \p_x^3 \eta( \cdot , t)\|_0^2  
 &\lesssim e_\gamma(0) + \sup_{s \in [0,t]} P(E_\gamma(s)) \,, \label{g43} \\
  \| \alpha ^2\d^{\frac{2\gamma-1}{2\gamma -2}} \p_x^2 v( \cdot , t) \|_0^2 +  \| \agtt \d^{\frac{2\gamma-1}{2\gamma -2}} \p_x^2 \eta( \cdot , t)\|_0^2 
 &\lesssim e_\gamma(0)+ \sup_{s \in [0,t]} P(E_\gamma(s))
  \,. \label{g42}
 \end{align} 

 Just as in the case that $\gamma=2$, we explain the modifications required for the lower-order energy estimates.  We  again set
 $\deta = \eta-x$; then,
 $ \eta_x^{-4} = \left( 1+ \deta_x\right)^{-\gamma} = 1 - \gamma \deta_x  + \P(\deta_x) $
and 
 \eqref{euler-gamma4a} is written as
 \begin{align} 
 0&= 
 \d^{\frac{1}{\gamma-1}} \left( \alpha^{\gamma+1} \p_t v + 2\alpha ^\gamma \dot\alpha  v +  \deta + x \right)
 +\p_x \left[ \d^{\frac{\gamma}{\gamma-1}} (1-4 \deta_x) \right] +\p_x \left[  \d^{\frac{\gamma}{\gamma-1}} \P(\deta_x)\right] \nonumber  \\
 &=
  \d^{\frac{1}{\gamma-1}} \left( \alpha^{\gamma+1} \p_t v + 2\alpha ^\gamma \dot\alpha  v +  \deta \right)  +  \d^{\frac{1}{\gamma-1}} x 
 + \frac{\gamma}{\gamma-1} \d^{\frac{1}{\gamma-1}}  \, \d_x
  -\gamma \p_x ( \d^{\frac{\gamma}{\gamma-1}} \deta_x )+\p_x \left[  \d^{\frac{\gamma}{\gamma-1}} \P(\deta_x)\right] \nonumber\\
   &=
   \d^{\frac{1}{\gamma-1}} \left( \alpha^{\gamma+1} \p_t v + 2\alpha ^\gamma \dot\alpha  v +  \deta \right)  
  -\gamma \p_x ( \d^{\frac{\gamma}{\gamma-1}} \deta_x )+\p_x \left[  \d^{\frac{\gamma}{\gamma-1}}\P(\deta_x)\right] 
   \,, \nonumber
\end{align} 
 where we have used the fact that $ x + \frac{\gamma}{\gamma-1} \d_x =0$.   We then divide by $ \alpha^{\gamma-3} $ and 
 compute the $L^2$ inner-product of  this equation with 
 $v= \p_t \deta$ to
 find that
 \begin{align*} 
 &
\frac{d}{dt}  \int_{\Omega_0} \left[ \alpha^4 \d^{\frac{1}{\gamma-1}}  v^2  + \agt \d^{\frac{1}{\gamma-1}}  \deta^2 + \gamma \agt \d^{\frac{\gamma}{\gamma-1}} |\deta_x|^2 \right] dx  \\
& \qquad\qquad
+ (\gamma-3)   \alpha ^(2-\gamma)  \dot \alpha  \int_{\Omega_0} \left[  \d^{\frac{1}{\gamma-1}}  \deta^2+ \d^{\frac{\gamma}{\gamma-1}} |\deta_x|^2 \right] dx 
=\agt \int_{\Omega_0} \d^{\frac{\gamma}{\gamma-1}} \P(\deta_x) v_x dx \,.
\end{align*} 
Integrating from $0$ to $t$, and using the fact that both $\deta$ and $\deta_x$ vanish at $t=0$ and that $ \alpha ^{3-\gamma} (t) \le 1$, we have that
\begin{align} 
&\| \alpha^2 \d^{\frac{1}{2\gamma-2}}   v(\cdot ,t)\|_0^2  
\le \|  \d^{\frac{1}{2\gamma-2}}   u_0\|_0^2    + \int_0^t \alpha^{- 2}(s)  \int_{\Omega_0} \d^2 \P(\deta_x)  \alpha^{2} (s) v_x dx \, ds
  \lesssim e_\gamma(0) + \sup_{s \in [0,t]} \P(E_\gamma(s)) \,. \label{g4p0}
\end{align} 
 
To proceed with energy estimates for the first differentiated problem, we write the momentum equation as
$
   \alpha^4 \p_t v + 2\alpha ^3 \dot\alpha v + \agt  \deta - \frac{\gamma}{\gamma-1} \d_x\, \deta - \agt \gamma\d\, \deta_x
   +\agt \p_x \left[\d^{\frac{\gamma}{\gamma-1}} \P(\deta_x)\right]  =0$
and differentiate to find that
 $$
\alpha^4 \p_x v_t + 2\alpha ^3 \dot\alpha \p_x v + \agt \p_x\deta 
-\gamma \agt \p_x\left[ \frac{4}{3} \d_x\, \deta + \d\, \deta_x \right]
+\agt \p_x \left[ \d^{\frac{\gamma}{\gamma-1}} \P(\deta_x)\right]  =0 \,.
 $$
Computing the $L^2$ inner-product  of this equation with $\d^{\frac{\gamma}{\gamma-1}} \p_x v$ shows that
\begin{align*} 
&\frac{d}{dt}  \int_{\Omega_0} \left[ \alpha^4 \d^{\frac{\gamma}{\gamma-1}} v_x^2  + \agt  \d^{\frac{\gamma}{\gamma-1}} \deta_x^2 
+ \gamma \agt \d^{\frac{1}{\gamma-1}} \left(\frac{\gamma}{\gamma-1} \d_x\, \deta + \d\, \deta_x\right)^2 \right] dx   
 \\
& \qquad  
= \agt \int_{\Omega_0} \d^{\frac{\gamma}{\gamma-1}}\P(\deta_x) v_x dx \,.
\end{align*} 
Integrating from $0$ to $t$, the same argument used to obtain \eqref{g4p0} provides the bound
\begin{align} 
\| \alpha^2 \d^{\frac{\gamma}{2\gamma-2}} \p_x v(\cdot ,t)\|_0^2
  \lesssim e_\gamma(0) + \sup_{s \in [0,t]} \P(E_\gamma(s)) \,.\label{g4p1}
  \end{align}

 \vspace{.05in}
\noindent
{\it Step 3. Building the higher-order norm $E_\gamma(t)$.} From \eqref{g44}--\eqref{g4p1}, we have that
$$
\sup_{s \in [0,t]}   e_\gamma(t) \lesssim e_\gamma(0) + \sup_{s \in [0,t]} \P(E_\gamma(s)) \,,
$$
and hence by the embedding \eqref{w-embed},
$$ \alpha^4(t) \| v( \cdot , t)\|_{1}^2 + \alpha^4(t)  \sum_{ b=0}^2\int_{\Omega_0} \d^{\frac{1}{\gamma-1}}   |\p_x^b v(x, t)|^2 dx
\lesssim e_\gamma(0) + \sup_{s \in [0,t]} \P(E_2(s)) \,.
$$
From Theorem 5.3 in \cite{KuPe2003} and the definition of the fractional $H^s$ norm \eqref{hsnorm}, we  have that
\begin{equation}\label{sobemb}
\|\alpha^2(t)\p_x v(\cdot ,t) \|^2_{\frac{2\gamma-3}{2\gamma -2}} \lesssim e_\gamma(0) + \sup_{s \in [0,t]} \P(E_\gamma(s)) \,,
\end{equation} 
and hence $ \|\alpha^2(t) v(\cdot , t)\|^2_{(4\gamma-5)/(2\gamma -2)}\lesssim e_\gamma(0) + \sup_{s \in [0,t]} \P(E_\gamma(s))$.    In fact, \eqref{w-embed} and
Theorem 5.3 in \cite{KuPe2003} show that
$$
 \sum_{b=0}^2 \|\alpha^2(t) \d^b v(\cdot , t)\|^2_{\frac{4\gamma-5}{2\gamma -2} + b}
 \lesssim e_\gamma(0) + \sup_{s \in [0,t]} \P(E_\gamma(s)) \,.
$$

Now, returning to the inequality \eqref{sobemb},
since for $\gamma>3$,  $(3\gamma-2)/(2\gamma -2) > 1/2$,  by the Sobolev embedding theorem, it follows that
$$
\|\alpha^2(t) v(\cdot , t)\|^2_{W^{1, \infty }}\lesssim e_\gamma(0)+ \sup_{s \in [0,t]} \P(E_\gamma(s))
$$

The fundamental theorem of calculus shows that $\eta(x,t) -x = \int_0^t v(x, s) ds$ and so for $a=0,...,4$ and $\beta\ge 0$,
$$|\d^\beta \p_x^a\deta(x,t)| = \int_0^t  \alpha^{-2}(s) \,  |\alpha^{2}(s)\d^\beta \p_x^a v(x, s)| ds  \lesssim  \max_{ s \in [0,t]} |\alpha^{2} \d^\beta \p_x^a v(x , s)|\,,$$
so that, 
\begin{equation}\label{ss1}
\| \deta(\cdot, t)   \|^2_{W^{1, \infty }} +  \sum_{a=0}^4 \|  \d^{(1+3a)/(2\gamma -2)}\p_x^a \deta( \cdot ,t) \|_0^2  
+ \sum_{b=0}^2 \| \d^b \deta(\cdot , t)\|^2_{(4\gamma-5)/(2\gamma -2) + b}
\lesssim e_\gamma(0)+ \sup_{s \in [0,t]} \P(E_\gamma(s)) \,.
\end{equation}

 \vspace{.05in}
\noindent
{\it Step 4. Bound for $E_\gamma(t)$ and global existence.}  We have shown that
$$
\sup_{s \in [0,t]} E_\gamma(s)
  \lesssim e_\gamma(0) + \sup_{s \in [0,t]} P(E_\gamma(s)) \text{ for } t \in [0,T]\,,
$$
where $T$ is independent of $ \kappa$, since $e_\gamma(0)$ is  independent of $\kappa$.   By assumption $e_\gamma(0) \le \epsilon$.
By choosing $ \epsilon $ sufficiently small, we have that $ E_\gamma(t) \lesssim \epsilon$ for all $t \in [0,T]$.  Then, by the standard continuation argument,
\begin{equation}\nonumber
 E_\gamma(t) \lesssim \epsilon \ \ \text{ for all } t \ge 0\,.
\end{equation} 
from which we have that
$$
1- \varepsilon \le \eta_x(x,t) \le 1+ \varepsilon  \  \text{ for } \ t \ge 0\,,
$$
thus verifying \eqref{ss0}.

 \vspace{.05in}
\noindent
{\it Step 5. The limit as $\kappa \to 0$.} Proceeding in the same manner as for the case $\gamma=2$ concludes the proof. \end{proof}

\subsubsection{Global existence for all $1< \gamma \le 3$} 
As the parameter $\gamma \to 1$, it is necessary to study more and more space differentiated problems in order to employ the Sobolev embedding theorem to
ensure that $\|v_x ( \cdot , t) \|^2_{L ^ \infty }$ is bounded. 
If  $2< \gamma \le 3$, we can continue to use the same number of space differentiated problems  as for the case $\gamma >3$, but we do not
have the $\agt$-weight present, and the higher-order energy is given by
\begin{align}
e_\gamma(t) &= 
\|\d^{\frac{5\gamma-4}{2\gamma-2}}  \p_x^5 \deta( \cdot , t)\|_0^2
+ \sum_{b=0}^4 \left(  \| \alpha(t)^{\frac{\gamma+1}{2}}  \d^{\frac{1+b(\gamma-1)}{2\gamma -2}}
\p_x^b v( \cdot ,t) \|_0^2 +  \| \d^{\frac{1+b(\gamma-1)}{2\gamma -2}}\p_x^b \deta( \cdot ,t) \|_0^2  \right)
\,.
\nonumber
\end{align} 
As we explained in obtaining \eqref{sobemb}, due to \eqref{w-embed} and  Theorem 5.3 in \cite{KuPe2003} and the definition of the fractional $H^s$,
$\|\alpha^2(t)\p_x v(\cdot ,t) \|^2_{(3\gamma-2)/(2\gamma -2)} \lesssim e_\gamma(t)$
and by the Sobolev embedding theorem,
$
\|\alpha^2(t) v(\cdot , t)\|^2_{W^{1, \infty }}\lesssim e_\gamma(t)
$
whenever  $(3\gamma-2)/(2\gamma -2) > 1/2$, which holds if $ \gamma >2$.

Thus, for  $1<\gamma \le 2$, we must increase the number, $a$, of space differentiated problems.   The higher-order energy $e_\gamma(t)$ ensures that
$$
\d^{\frac{1+ a(\gamma-1)}{2\gamma -2}}  \p_x^a v( \cdot , t) \in L^2(\Omega_0).
$$
In order for $v( \cdot , t) \in H^s(\Omega_0)$ for $s > {\frac{3}{2}} $, according to \eqref{w-embed}, it is necessary that
$$
{\frac{1+ a(\gamma-1)}{2\gamma -2}}   < a - {\frac{3}{2}}  \text{ which implies that } a > {\frac{3\gamma -2}{\gamma -1}} \,.
$$

\begin{definition}[Norm for $1< \gamma \le 3$ in 1-d]
For $1 < \gamma \le 3$ and for $a > {\frac{3\gamma -2}{\gamma -1}}$, we define the higher-order energy function as
\begin{align}
e_\gamma(t) &= 
\| \d^{\frac{1+(a+1)(\gamma-1)}{2\gamma -2}} \p_x^{a+1} \deta( \cdot , t)\|_0^2
+ \sum_{b=0}^a \left(  \| \alpha(t)^{\frac{\gamma+1}{2}}  \d^{\frac{1+b(\gamma-1)}{2\gamma -2}}
\p_x^bv( \cdot ,t) \|_0^2 +  \| \d^{\frac{1+b(\gamma-1)}{2\gamma -2}}\p_x^b \deta( \cdot ,t) \|_0^2  \right) \,,
\,.
\nonumber
\end{align} 
and for the integer $L < a - \frac{a(\gamma-1)-1}{2\gamma -2} $, the 
norm is given by
\begin{align*} 
E_\gamma(t) & = e_\gamma(t) + \|\deta(\cdot , t)\|^2_{W^{1, \infty }}+\|\alpha^{\frac{\gamma+1}{2}}  v(\cdot , t)\|^2_{W^{1, \infty }} \\
& \qquad\qquad
+ \sum_{b=0}^L \left( \|\alpha^ {\frac{\gamma+1}{2}}  \d^b v(\cdot , t)\|^2_{\frac{a(\gamma-1)-1}{2\gamma -2} + b}
+  \| \d^b \deta(\cdot , t)\|^2_{\frac{a(\gamma-1)-1}{2\gamma -2} + b}\right)\,.
\end{align*} 
\end{definition} 
Having defined $e_\gamma(t)$ and $E_\gamma(t)$, the identical proof as for the case $\gamma=2$ provides us with the following
\begin{theorem}  For $1<\gamma \le3$ and  $ \epsilon >0$ taken sufficiently small, if the initial data satisfies
 $e_\gamma(0) <  \epsilon $, then there exists a global unique solution of the Euler equations \eqref{ceuler},  such that $E_\gamma(t) \le \epsilon $ for all $t \ge 0$.
 Furthermore, $\xi( x,t) = \alpha(t) \eta(x,t) $ is a global solution to the 1-d Euler equations \eqref{ceuler}.
\end{theorem}

\section{Global existence for 3-d compressible Euler equations}\label{sec:3d}
We now examine the multi-dimensional problem.
The case that $\gamma=2$ is of particular interest as it coincides with the shallow water equations.  Herein, we prove
global existence and the stability of affine solutions for $\gamma>1$.   While we write our proofs in the 3d setting, 
all of our results also hold for 2d fluids.

\subsection{Lagrangian formulation of the  Euler equations}
We define the Lagrangian flow of the velocity $u$ by
\begin{equation}
\nonumber
\begin{array}{c}
\partial_t \xi = u \circ \xi $ for $ t>0 $ and $
\xi(x,0)=x
\end{array}
\end{equation}
where $\circ $ denotes composition so that
$[u \circ \xi] (x,t):= u(\xi(x,t),t)\,.$  We set 
\begin{align*}
B &= [\nabla  \xi]^{-1}  \text{ (inverse of deformation tensor)}, \\
\mathcal{J}  &= \det \nabla  \xi  \text{ (Jacobian determinant)}, \\
b &= \mathcal{J} \, B  \text{ (transpose of cofactor matrix)}. 
\end{align*}

Since $ \xi (0,x)=x$, 
from \cite{CoSh2012},  we have that
$\rho \circ \xi = \rho_0  \mathcal{J}  ^{-1}$, and the Euler equations in Lagrangian coordinates can be written as
\begin{subequations}
\label{ce03d}
\begin{alignat}{2}
\rho_0 \p^2_t \xi^i +   b^k_i (\rho_0^ \gamma \, \mathcal{J} ^{-\gamma}),_k &=0 \ \ && \text{ in } \Omega_0   \times (0,T] \,, \\
(\xi, \p_t\xi)  &=(e, \xi_1) \ \  \ \ && \text{ in } \Omega_0   \times \{t=0\} \,,  
\end{alignat}
\end{subequations}
with $e(x)=x$ and $\xi_1$ denoting, respectively, the initial position and velocity on $\Omega_0$.   The initial density function must satisfy
 $ \rho _0^{ \gamma -1}(x) \ge C \dd( x, \Gamma_0 ) $ for $x \in \Omega_0$ near $\Gamma_0$.

\subsection{Perturbations of 3-d affine solutions for $\gamma=2$} 
We begin our analysis with the case that
$\gamma=2$,  the shallow water equations, which already places us in the ``anti-damping'' analysis regime. 

 For a particular choice of the distance function $\d(x)$ to be made precise below and associated to the affine
flow, we set $\rho_0(x)=\d(x)$, and write
\eqref{ce03d} as
\begin{equation}\label{ce3d}
\d  \p^2_t\xi^i+  b^k_i  \left(\d^2 \mathcal{J} ^{-2}\right),_k =0 \ \  \text{ in } \Omega_0   \times (0,T] \,,
\end{equation} 
The case  $\gamma=2$  is particularly nice to start with precisely because the physical vacuum condition \eqref{degen} allows us to set $\rho_0 = \d$, with no 
fractional powers on the distance function, which we find somewhat elegant.  As we shall explain, however,  there is no fundamental difference in the analysis between $\gamma=2$ and general $\gamma$.

\subsubsection{Perturbation of the simplest affine solutions}
As noted in Section \ref{sec:affine}, 
the simplest affine solutions are given by $A(t)=\alpha(t)   \operatorname{Id} $, and for the most part, the general stability theory can be reduced to the
analysis of this type of motion.

We set $\Omega_0 = \{x \in \mathbb{R}^3  \ : \ |x| < 1\}$, the open unit ball, and we define
the distance function to $\Gamma_0 = \{ |x|=1\}$ to be 
\begin{equation}\label{dist3d}
\d(x) = {\frac{1}{4}} (1- |x|^2) \,.
\end{equation}

We continue to denote  the perturbation  of the affine flow by $\eta(x,t)$.
  In particular, we assume that solutions $ \xi $ to \eqref{ce03d} have the form $\xi (x,t)= A(t) \eta(x,t)$ and
define
\begin{align*} 
v &= \p_t\eta   \text{ (Lagrangian perturbation velocity)},  \\
A &= [\nabla  \eta]^{-1}  \text{ (inverse of perturbed deformation tensor)}, \\
J &= \det \nabla  \xi  \text{ (Jacobian determinant of perturbation)}, \\
a &= J \, A  \text{ (transpose of cofactor matrix of perturbation)}. 
\end{align*}

 For the simplest affine motions, we consider
solutions to \eqref{ce3d} of the form
\begin{equation}\nonumber
\xi( x,t) = \alpha(t) \eta(x,t) \,.
\end{equation} 
With $v(x,t) = \p_t \eta(x,t)$,
it follows that
$$
\p_t \xi(x, t) = \dot \alpha(t) \eta(x,t) + \alpha(t) v(x,t) \,.
$$
For simplicity, let us suppose
 that  at time $t=0$, $\eta(x,0) =e(x)$,  the identity map on $\Omega_0$, and that $\p_t\eta(x,0) = u_0(x)$; furthermore, we let $\alpha(0)=1$.  Then,
$$
\p_t \xi(x, 0) = \dot \alpha(0) x + u_0(x)\,.
$$
In order to analyze the stability of the affine solution, we shall assume that $u_0(x)$ is a small perturbation of the initial affine velocity $ \dot \alpha(0) x$.

From \eqref{affine1d}, we have that  $\ddot \alpha = \alpha^{-4}$, and we see that the perturbation $(\eta,v)$ satisfies the following degenerate wave equation:
\begin{subequations}
\label{ceuler03D}
\begin{alignat}{2}
\p_t \eta &=v \ \ && \text{ in } \Omega_0   \times (0,T] \,, \\
\alpha^5\p_t v^i + 2\alpha ^4 \dot\alpha v^i + \eta^i + \d ^{-1} a^k_i  \left(\d^2 J^{-2}\right),_k &=0 \ \ && \text{ in } \Omega_0   \times (0,T] \,, \\
(\eta, v)  &=( e, u_0) \ \  \ \ && \text{ in } \Omega_0   \times \{t=0\} \,.
\end{alignat}
\end{subequations}
We note that  there is no loss of generality  in choosing the initial condition $\eta(x,0)$ to be the identity map $e$; see Remark \ref{rem:data} below.

As in \cite{CoLiSh2010,CoSh2012}, to derive a simple formula for the evolution of $ \operatorname{curl} \eta$, it is
more convenient to  write (\ref{ceuler03D}b) as
\begin{equation}\label{for_vorticity}
\alpha^5\p_t v^i + 2\alpha ^4 \dot\alpha v^i + \eta^i + 2 A^k_i \left(\d J^{-1}\right),_k =0 \ \  \text{ in } \Omega_0   \times (0,T] \,.
\end{equation}

\subsubsection{Perturbation of  general affine solutions} 
We next considering perturbations of general affine solutions to \eqref{agen}, satisfying the hypotheses of Lemma \ref{lemma_affine}. 
  In this case, 
 solutions to \eqref{ce3d} take the form
\begin{equation}\nonumber
\xi( x,t) = \A(t) \eta(x,t) \text{ or in components, } \xi ^i(x,t) =\A ^i_j (t)\eta^j(x,t) \,.
\end{equation} 

For perturbations $\eta(x,t)$ of a general affine solution,
the momentum equation \eqref{ceuler03D} is replaced by
\begin{equation}\label{ceuler_gen}
(\A^T)^i _s \A^s_j \p_t v^j + 2(\A^T)^i _s \dot\A^s_j v^j + {\frac{1}{\det \A}}  \eta^i +{\frac{1}{\det \A}}  \d ^{-1} a^k_i  \left(\d^2 J^{-2}\right),_k =0  \,.
\end{equation} 
or equivalently, as 
\begin{equation}\label{vort_gen}
(\A^T)^i _s \A^s_j \p_t v^j + 2(\A^T)^i _s \dot\A^s_j v^j + {\frac{1}{\det \A}}  \eta^i + {\frac{2}{\det \A}}  A^k_i  \left(\d J^{-1}\right),_k =0  \,.
\end{equation} 
The presence of the matrix $(\A^T)^i _s \A^s_j $ in the above equations  shall require a  minor generalization of the analysis of the simple case that $\A(t) = \alpha (t) \operatorname{Id} $.

\subsection{The perturbed velocity in Eulerian coordinates}  With $u$ continuing to denote the Euler velocity solving the Euler equations, 
we now let $w = v \circ \eta ^{-1} $ denote the  Eulerian perturbed velocity.
Note that since
$$
\p_t\xi = \A \, v + \dot\A \, \eta =  \A \, v + \dot\A \, \A ^{-1} \xi \text{ and } u = \p_t \xi \circ \xi ^{-1} \,,
$$
we see that
$$
u(y,t)- \dot \A(t) \A(t) ^{-1} y = \A(t) v(\xi ^{-1} (y,t),t) =  \A(t) v(\eta ^{-1} (\A(t) ^{-1} y,t),t) =\A(t) w (\A(t) ^{-1} y,t) \,.
$$
It follows that 
$$
w (\A(t) ^{-1} y,t) = \A(t) ^{-1} u(y,t) - \A(t) ^{-1} \dot \A(t) \A(t) ^{-1} y \,,
$$
or equivalently with $y= \A(t) z$,
$$
w (z,t) = \A(t) ^{-1} u(  \A(t) z,t) - \A(t) ^{-1} \dot \A(t) z \,.
$$

Furthermore, the density solving Euler is given by $ \rho \circ (\A(t) \eta) = \rho_0/( \det \nabla \eta \det \A(t))$ or equivalently with $\xi = A(t)\eta$,
we see that $\rho =\left( \rho_0/(\det \nabla \xi)\right) \circ \xi ^{-1} $.

\subsection{Eulerian and Lagrangian derivatives}  We use coordinates  $x=(x_1,x_2,x_3)$ on $\Omega_0$.
The gradient  of a function $F$ is
 $$\nabla =\left(\frac{\p}{\p x_1}, \frac{\p}{\p x_2},\frac{\p}{\p x_3}\right) 
  \,.$$ 
  We let $ \bp = x \times \nabla $ so that
 \begin{equation}\label{bp}
  \bp = \left(  x^2\p_3 - x^3 \p_2, x^3 \p_1 - x_1 \p_3, x_1 \p_2 - x_2 \p_1\right)  \,.
\end{equation} 
  Notice that
  $$
  \bp_i \d =0\ \text{ for } \ i=1,2,3\,.
  $$
  
 The $k$th-component of the $1$st partial derivative of  a function $F$ will be denoted by $F,_k = \frac{ \p F}{ \p x_k}$.   
 Higher-order $j$th partial derivatives will be written as
 $$
 F,_{r_1 \cdot\cdot\cdot r_j} := \frac{\partial^j F}{\partial x_{r_1}\cdot\cdot\cdot \partial x_{r_j}} \,,
 $$
 where $j \ge 1$ is an integer and for $i=1,...,j$, each $r_i=1,2$.   We shall sometimes use
  the notation $ \nabla ^j F$ to mean $ F,_{r_1 \cdot\cdot\cdot r_j}$, when only the derivative count is important.   Similarly, we shall write
  $$
  \bp^j := \bp_{r_1} \ddd \bp_{r_j} \,.
  $$
  

The divergence of a vector field $V$ is 
$$
\operatorname{div} V = V^1,_1 + V^2,_2 + V^3,_3 \,,
$$
and 
$$
 \operatorname{curl} V = \left( V^3,_2 - V^2,_3\,, V^1,_3 - V^3,_1 \,, V^2,_1 - V^1,_2\right) \,.\,.
$$
Throughout the paper, we will make use of the permutation symbol 
\begin{equation}\label{permutation}
\varepsilon_{ijk} = \left\{\begin{array}{rl}
1, & \text{even permutation of } \{ 1, 2, 3\}, \\
-1, & \text{odd permutation of } \{ 1, 2, 3\}, \\
0, & \text{otherwise}\,,
\end{array}\right.
\end{equation} 
This allows us to write the $i$th component of the curl of a vector-field $V$ as
$$
[ \operatorname{curl}  V]_i = \varepsilon_{ijk} V^k,_j  \text{ or equivalently } \operatorname{curl} V = \varepsilon_{\cdot jk} V^k,_j
$$
which agrees with our definition above, but is notationally convenient.

The $i$th component of the Lagrangian gradient $ \nabla _\eta$ is defined by
$$
[\nabla _\eta f ]_i = f,_k A^k_i \,.
$$
We will also define the Lagrangian divergence and curl operators as follows:
\begin{equation}\label{lagrangian_div}
\operatorname{div} _\eta W = A^j_i W^i,_j \,,
\end{equation} 
and 
\begin{equation}\label{lagrangian_curl}
\operatorname{curl} _\eta W = \varepsilon_{\cdot jk} A^r_j W^k,_r \,.
\end{equation}

Finally, we shall use the Einstein summation convention, wherein repeated Latin indices $i,j,k,$ etc., are summed
from $1$ to $3$, so for example $A^k_i F,_k = \sum_{k=1^2} A^k_i F,_k$ for $i=1,3$.


\subsection{The cofactor matrix and the Jacobian determinant}
\subsubsection{The Jacobian, cofactor matrix, and Piola identity} 
The cofactor $a$ can be written as the matrix
\begin{equation}\label{a3d}
a=\left[ 
\begin{matrix}
\eta,_2\times \eta,_3 \\
\eta,_3\times \eta,_1 \\
\eta,_1\times \eta,_2
\end{matrix}
\right]\,.
\end{equation} 
 It is thus easy to verify that the columns of every cofactor matrix
are divergence-free and satisfy the so-called Piola identity
\begin{equation}\label{piola}
a^k_i, _k =0\,.
\end{equation}
The identity (\ref{piola}) will play a vital role in our energy estimates. (Note that we use
the notation cofactor for what is commonly termed the {\it adjugate matrix}, or the transpose
of the cofactor.)

In 3-d, the Jacobian determinant satisfies the identity
\begin{equation}\label{J3d}
J = {\frac{1}{3}}  \eta^r,_s a^s_r \,.
\end{equation}

\subsubsection{Differentiating the Jacobian determinant and inverse deformation tensor}  For any partial derivative $\p_j$,  $j=1,2,3$,
the following identities will be
useful to us:
\begin{alignat}{2}
&\hspace{.09in} \p_j J=  a^s_r \eta^r,_{sj}  \,,  \ \ 
 &&\hspace{.09in} \partial_t  J= a^s_r   v^r,_s \,, \label{J1} \\
 &\p_j  A^k_i = - A^k_r  \eta^r,_{sj} A^s_i  \,, \ \ 
 &&\p_t A^k_i  = - A^k_r v^r,_s A^s_i \,, \label{Adiff1}
\end{alignat}

\subsubsection{Differentiating the cofactor matrix}  Using (\ref{J1})--(\ref{Adiff1}) and the fact that $a= J\, A$,
we find that
\begin{align}
\p_r a^k_i &= \eta^j,_{sr} J
[A^s_j A^k_i - A^k_j A^s_i]  
\,, \label{a1}\\
\partial_t  a^k_i &= v^r,_s  J^{-1} 
[a^s_r a^k_i - a^s_i a^k_r]  
\,. \label{a2}
\end{align}

\subsubsection{Two important  identities  for energy estimates}
The following identity is from \cite{CoLiSh2010,CoSh2011}:
\begin{lemma}\label{lemma_aenergy} For any multi-index $ \alpha $, 
$$
 - \p^\alpha   \eta^j,_m (A^m_j A^k_i - A^k_j A^m_i)  \,  \p^\alpha v^i,_{k} ={\frac{1}{2}}  \frac{d}{dt} \left( |  \nabla_\eta \ \p^\alpha \eta |^2 
- | \operatorname{div}_\eta \p^\alpha  \eta|^2 - 2 | \operatorname{curl}_\eta \p^\alpha \eta |^2\right)  \,.
$$
\end{lemma} 
\begin{proof} 
We compute that
\begin{align*} 
 & \p^\alpha   \eta^j,_m (A^m_j A^k_i - A^k_j A^m_i)  \,  \p^\alpha v^i,_{k} \\
 & \qquad =  \p^\alpha   \eta^j,_m A^m_j   \p^\alpha v^i,_{k}A^k_i   - \p^\alpha   \eta^j,_m A^m_i  \p^\alpha v^i,_{k}A^k_j  \\
 & \qquad
  = \operatorname{div} _\eta \p^ \alpha \eta\, \operatorname{div}_ \eta \p^ \alpha v
    - \p^\alpha   \eta^j,_m A^m_i \p^\alpha v^j,_{k}A^k_i 
 +  \left(\p^\alpha   \eta^j,_m A^m_i  - \p^\alpha   \eta^i,_m A^m_j  \right)  \left( \p^\alpha v^j,_{k}A^k_i        -    \p^\alpha  v^i,_{k}A^k_j   \right) \\
  & \qquad
  =   \operatorname{div} _\eta \p^ \alpha \eta\, \operatorname{div} _\eta \p^ \alpha v - \nabla _\eta \p^ \alpha \eta \, \nabla _\eta \p^ \alpha v
 +  \p^\alpha   \eta^j,_m A^m_i  \left( \p^\alpha v^j,_{k}A^k_i        -    \p^\alpha  v^i,_{k}A^k_j   \right)  \\
  & \qquad
  = - \nabla _\eta \p^ \alpha \eta \, \nabla _\eta \p^ \alpha v + \operatorname{div} _\eta \p^ \alpha \eta\, \operatorname{div} _\eta \p^ \alpha v
  +2 \operatorname{curl}_ \eta \p^ \alpha \eta\, \operatorname{curl}_ \eta \p^ \alpha v \,.
\end{align*} 
\end{proof} 

\begin{lemma} \label{lemma_atan} For any multi-index $\alpha $, we define the matrix  
$ G^ \alpha  
=\left[ 
\begin{matrix}
\p^ \alpha \eta,_2\times \eta,_3 \\
\p^ \alpha \eta,_3\times \eta,_1 \\
\p^ \alpha \eta,_1\times \eta,_2
\end{matrix}
\right]$. 
Then,
$$
x_k\,  [G^ \alpha ]^k_{( \cdot )} = -(\bp \p^\alpha \eta^3\cdot \bp \eta^2 \,, \   \bp \p^ \alpha \eta^1 \cdot \bp \eta^3 \,, \  \bp \p^ \alpha \eta^2\cdot \bp \eta^1 )^T \,.
$$
\end{lemma} 
\begin{proof} 
Using the identity \eqref{bp}, we  find that 
$$x_k\,  [G^ \alpha ]^k_{( \cdot )}  = - ( \bp \p^ \alpha \eta^3 \cdot \nabla \eta^2\,, \bp\p^\alpha  \eta^1 \cdot \nabla \eta^3\,,  \bp \p^ \alpha  \eta^2\cdot \nabla \eta^1 )^T \,.$$   Since $\bp \p^ \alpha \eta^j \cdot x =0$ for
$j=1,2,3$, the $\nabla $ operator can be replaced with the $\bp$ operator, and we arrive at the desired formula.
\end{proof}

\subsection{Global existence in 3-d and stability of simple affine solutions} 
\def\Ee{   {\mathscr E}_2 }
\def\Eeg{   {\mathscr E}_\gamma }
We shall first consider perturbations of the simple affine solutions $\A(t) = \alpha (t) \operatorname{Id} $, and later treat more general affine flows.

\subsubsection{Some basic inequalities}
The time-dependent vacuum boundary $\Gamma(t):= \alpha(t) \eta(\Gamma_0,t)$
 is a closed surface whose geometry is controlled by the $\eta(x,t)$.     For our analysis,  we define the near identity map
 $$
 \deta (x,t) = \eta(x,t) -x \,.
 $$
 We shall assume that for $ \vartheta >0$ taken sufficiently small,
\begin{equation}\label{eta_bound2d}
\| \nabla  \deta(\cdot ,t)  \|_{L^\infty }  \le \vartheta  \ \text{ for all } t \ge 0 \,,
\end{equation} 
and prove later that  $\| \nabla  \deta(\cdot ,t)   \|_{L^\infty }  \le C\epsilon  $ for all $t \ge 0$ and for $ 0 < C \epsilon \ll \vartheta $.
Since
$$
\| \operatorname{Id} - A\|_{L^ \infty } = \| A\, (\nabla \eta - \operatorname{Id})\|_{L^ \infty }  \le \vartheta  \|A\|_{L^ \infty }  \,,
$$
then  $ (1-\vartheta) \|A\|_{L^ \infty } \le \| \operatorname{Id}\|_{L^ \infty } $ so that $\|A\|_{L^ \infty } \le  \frac{1}{1-\vartheta} \| \operatorname{Id}\|_{L^ \infty } $  and hence
\begin{equation}\label{Abound}
\|A-  \operatorname{Id} \|_{L^ \infty } \le {\frac{\vartheta}{1-\vartheta}} \lesssim \vartheta \ 
 \text{ and } \ \| A \, A^T - \operatorname{Id} \|_{L^ \infty } \le {\frac{3\vartheta}{1-\vartheta}} \lesssim \vartheta  \ \text{ for all } t \ge 0 \,.
\end{equation} 

The bound \eqref{eta_bound2d} 
implies that
\begin{equation}\label{ss10}
1 - {\frac{1}{10}} \le J(x, t) \le 1 + {\frac{1}{10}} \,,
\end{equation} 
and hence that
\begin{equation}\label{Jbound}
\left[1 + {\frac{1}{10}}\right]^{-1}  \le J ^{-1}  (x, t) \le \left[1 - {\frac{1}{10}}\right]^{-1}  \,.
\end{equation}

At time $t=0$, 
the physical vacuum condition  \eqref{degen} is satisfied as  $| \nabla \d| = 1/2 $ on $\Gamma_0 $ and that 
\begin{equation}\label{grad_d}
|\nabla \d| \le {\frac{1}{2}}  \ \text{ and } \  \d,_r = - {\frac{1}{2}} x,_r \ \text{ and } \ 
\d,_{r_1r_2} = -{\frac{1}{2}} \delta_{r_1r_2}\  \text{ on } \overline \Omega_0 \,,
\end{equation} 
where $\delta_{r_1r_2}$ denote the Kronecker delta.   Thanks to \eqref{Jbound}, we see that the physical vacuum condition continues to be
satisfied as long as the solution exists.

Finally, we shall use the fact that\footnote{
By \eqref{a3d},
$a- \operatorname{Id} =  \operatorname{div} \deta \operatorname{Id} - \operatorname{Diag}[ \deta^1,_1, \deta^2,_2, \deta^3,_3]
+ \mathcal{Q} ( \nabla \deta)$ which we write as $ \mathcal{L} ( \nabla \deta) + \mathcal{Q} ( \nabla \deta)$ with $ \mathcal{L} $ and $ \mathcal{Q} $
respectively denoting linear and quadratic functions of their arguments.
The identity \eqref{J3d} then shows that
$J = 1 + \operatorname{div} \deta +  \mathcal{Q} ( \nabla \deta)  + \mathcal{C} ( \nabla \deta)$, with $ \mathcal{C} $ denote a cubic function of its argument.
Hence,  $J ^{-2}  = 1- 2\operatorname{div} \deta + \P( \nabla \deta)$.
}
\begin{equation}\label{aJ2}
(a - \operatorname{Id}) J^{-2} =  \mathcal{L} ( \nabla \deta) + \P( \nabla \deta)\,,
\end{equation} 
where $ \mathcal{L} $ is linear function of its argument.

\subsubsection{The norm used for the $\gamma=2$ analysis}

\begin{definition}[Norm for  ${\mathbf \gamma=2}$ norm]\label{defE}
We define the $\gamma=2$ norm,  $\Ee(t):=\Ee\left( \eta( \cdot , t) \,, v( \cdot , t)\right) $, by
\begin{align} 
\Ee(t) &=\sum_{b=0}^8\sum_{\a=0}^{8-b} \left( \|\d^{{\frac{b+1}{2}} }  \nabla ^b \bp^\a \deta( \cdot ,t)\|_0^2  
+  \|\alpha ^2 \d^{{\frac{b+1}{2}} }  \nabla ^b \bp^a v( \cdot ,t)\|_0^2 \right) 
 +\sum_{b=1}^9  \| \alpha^{-1/2} \d^{{\frac{b+1}{2}} } \nabla ^b \bp^{9-b} \deta( \cdot ,t)\|_0^2 \nonumber \\
 & 
+\sum_{\a=0}^7  \left(\| \alpha^2 \bp^\a v( \cdot ,t) \|_{3.5-\a/2}^2 +  \|  \bp^\a \deta(\cdot ,t)  \|_{3.5-\a/2}^2 \right)
+\sum_{\a=0}^8   \| \alpha^{-1/2} \bp^\a \deta(\cdot ,t)  \|_{4-\a/2}^2  \nonumber \\
 &
+\sum_{b=1}^5 \sum_{\a=0}^{10-2b}    \left\| \alpha^{-1/2} \d^b \nabla^b \bp^\a \deta(\cdot ,t)  \right\|_{4-\a/2}^2 
+\sum_{b=1}^4 \sum_{a=0}^{9-2b} \left( \| \alpha ^2 \d^b \nabla ^b \bp^\a v\|^2_{3.5-\a/2} +  \|  \d^b \nabla ^b \bp^\a \deta\|^2_{3.5-\a/2}\right)
\nonumber
\\
&
+\sum_{b=0}^8
 \| \alpha ^2\d^{\frac{b+2}{2}}  \nabla^b \bp^{8-b} \operatorname{curl} _\eta v( \cdot , t)\|_0^2
\,.\nonumber
\end{align} 
\end{definition} 

\begin{remark} Only the first two terms and the last term  of $\Ee(t)$ are essential in the definition of the higher-order norm.   All of the other terms in $\Ee(t)$
follow from weighted embeddings and the Sobolev embedding theorem.  Moreover,   we have defined $\Ee(t)$ to have the fewest derivatives necessary for the Sobolev
embedding theorem to ensure that $\nabla \p_t v$ is pointwise bounded, a requirement for closing estimates for weighted derivatives of $ \operatorname{curl} \deta$.

More generally, we can define for all $K \ge 8$, the $K$-dependent higher-order norm as
\begin{align} 
\Ee^K(t) &=\sum_{b=0}^K\sum_{\a=0}^{K-b} \left( \|\d^{{\frac{b+1}{2}} }  \nabla ^b \bp^\a \deta( \cdot ,t)\|_0^2  
+  \|\alpha ^2 \d^{{\frac{b+1}{2}} }  \nabla ^b \bp^a v( \cdot ,t)\|_0^2 \right) 
 +\sum_{b=1}^{K+1}  \| \alpha^{-1/2} \d^{{\frac{b+1}{2}} } \nabla ^b \bp^{K+1-b} \deta( \cdot ,t)\|_0^2 \nonumber \\
&+
\sum_{b=0}^K
 \| \alpha ^2\d^{\frac{b+2}{2}}  \nabla^b \bp^{K-b} \operatorname{curl} _\eta v( \cdot , t)\|_0^2
\,.\nonumber
\end{align} 

\end{remark}
\begin{remark}\label{rem:data} With regards to the initial data for $\eta(x,0)$ which determines the initial domain of the problem, we can
 replace $\eta(x,0)=e$ with the initial condition $\eta(x,0)=\eta_0(x)$ for any near-identity embedding $\eta_0: \Omega_0 \to 
\mathbb{R} ^d$, such that $\Ee( \eta_0, u_0) < \epsilon $ for $ \epsilon >0$ chosen sufficiently small.   The only modification to our analysis requires the
replacement of the Jacobian
determinant $J$ by the ratio $J/(\det \nabla \eta_0)$.
\end{remark}

\subsubsection{Curl estimates}  For the case of perturbations of simple affine solutions $\A(t) = \alpha (t) \operatorname{Id} $, the curl estimates use
formula which is 
similar to that used in \cite{CoLiSh2010, CoSh2012}.

\begin{lemma}[Curl estimates]\label{lemma_curl} For all $t\ge 0$, 
\begin{align*} 
  &\underbrace{ \sum_{b=0}^7\sum_{\a=0}^{7-b} }_{b+\a>0} \|\d^{{\frac{b+2}{2}} } \operatorname{curl}   \nabla ^{b} \bp^\a \deta ( \cdot ,t)\|_0^2  
 +\sum_{b=0}^8  \| \alpha^{-1/2} \d^{{\frac{b+2}{2}} }\operatorname{curl}  \nabla ^b \bp^{8-b} \deta( \cdot ,t)\|_0^2  \lesssim 
  \Ee(0)  + \sup_{s \in [0,t]} \P (\Ee(s)) \,.
 \end{align*} 
\end{lemma} 
Note that in 2-D,  by defining $ \operatorname{curl} \deta= \nabla ^\perp \cdot \deta$ with $\nabla ^\perp= (-\p_2, \p_1)$, the lemma also holds.
\begin{proof}  We let  $ \operatorname{curl} _\eta$ act on 
equation \eqref{for_vorticity}, written as
$\alpha^5\p_t v + 2\alpha ^4 \dot\alpha v + \eta + 2  \nabla _\eta \left(\d J^{-1}\right)=0$, and using the fact that
$ \operatorname{curl} _\eta$ annihilates $ \nabla _\eta$ and that $ \operatorname{curl} _\eta \eta =0$, we obtain that
$$
\operatorname{curl} _\eta \p_t v + 2 \frac{\dot \alpha}{ \alpha } \operatorname{curl}_\eta v =0 \,.
$$

We commute $\p_t$ with $ \operatorname{curl} _\eta$ and find that
$$
\p_t (\operatorname{curl} _\eta v ) +2 \frac{\dot \alpha}{ \alpha } \operatorname{curl}_\eta v = \varepsilon_{ \cdot kj} v^k,_r A^r_l v^l,_m A^m_j \,,
$$
and so using the $ \alpha ^2$ integrating-factor, we have that
\begin{equation}\label{curl_need}
\alpha ^2 \operatorname{curl} _\eta v ( \cdot , t)= \operatorname{curl} u_0 +  \varepsilon_{ \cdot kj} \int_0^t  \alpha ^2  v^k,_r A^r_l v^l,_m A^m_j  \, dt' \,.
\end{equation} 
Then, since  $\operatorname{curl} _\eta v = \operatorname{curl} v +  \varepsilon_{ \cdot jk }v^k,_r (A^r_j - \delta ^r_j)$, 
\begin{equation}\nonumber
 \operatorname{curl}  v ( \cdot , t)=  \at \operatorname{curl} u_0 + \varepsilon_{ \cdot kj }v^k,_r (A^r_j - \delta ^r_j) +  \varepsilon_{ \cdot kj}\at \int_0^t  \alpha ^2  v^k,_r A^r_l v^l,_m A^m_j  \, dt' \,.
\end{equation} 
We set $\beta(t) =  \int_0^t \alpha(t') ^{-2} dt' \le C  < \infty$ for $ t \ge 0$.   Applying the fundamental theorem of calculus once again, and using the fact that
$ \operatorname{curl} \eta = \operatorname{curl} \deta$, shows that
\begin{equation}\label{sss2}
 \operatorname{curl}  \deta ( \cdot , t)=  \beta(t) \operatorname{curl} u_0 + \varepsilon_{ \cdot kj } \int_0^t v^k,_r (A^r_j - \delta ^r_j) dt'
 +  \varepsilon_{ \cdot kj}  \int_0^t \at \int_0^{t'}  \alpha ^2  v^k,_r A^r_l v^l,_m A^m_j  \, dt'' dt' \,.
\end{equation} 

For $0\le b \le 7$ and $ 0\le \a \le  7-b$ with $b+\a>0$, we have that
\begin{align*} 
\d^{{\frac{b+2}{2}} }\nabla^b\bp^\a \operatorname{curl}  \deta ( \cdot , t)
&=  \beta(t) \d^{{\frac{b+2}{2}} }\nabla^b\bp^\a  \operatorname{curl} u_0 
+ \varepsilon_{ \cdot kj } \int_0^t \d^{{\frac{b+2}{2}} }\nabla^b\bp^\a \left[ v^k,_r (A^r_j - \delta ^r_j)\right] dt' \\
& \qquad
 +  \varepsilon_{ \cdot kj}  \int_0^t \at \int_0^{t'}  \alpha ^2 \d^{{\frac{b+1}{2}} } \nabla^b\bp^\a \left[ v^k,_r A^r_l v^l,_m A^m_j \right] \, dt'' dt' \,.
 \end{align*} 
From the definition of $\Ee(t)$, for $0\le b \le 7$ and $ 0\le \a \le 7-b $ and $b+\a>0$, 
$\d^{{\frac{b+2}{2}} }\nabla^b\bp^\a \operatorname{curl}  \deta$ and 
$ \alpha ^2 \d^{{\frac{b+2}{2}} }\nabla^b\bp^\a  \nabla  v$ have the same 
regularity.   We have that
\begin{align*} 
\|\d^{{\frac{b+2}{2}} }\nabla^b\bp^\a \operatorname{curl}  \deta ( \cdot , t)\|_0^2
&=  \beta(t)\int_{\Omega_0} \d^{b+2} \nabla^b\bp^\a  \operatorname{curl} u_0  \ \nabla^b\bp^\a \operatorname{curl}  \deta dx  +
\mathcal{J}^{b,\a} _1 +\mathcal{J}^{b,\a}_2 \,,
\end{align*} 
where
\begin{align*} 
\mathcal{J}_1^{b,\a} &= \varepsilon_{ \cdot kj } \int_{\Omega_0}  \int_0^t \at \d^{b+2}\nabla^b\bp^\a \left[ \alpha ^2 v^k,_r (A^r_j - \delta ^r_j)\right] dt' 
 \ \nabla^b\bp^\a \operatorname{curl}  \deta dx  \,, \\
\mathcal{J}_2^{b,\a} & 
 =  \varepsilon_{ \cdot kj} \int_{\Omega_0}   \int_0^t \at \int_0^{t'} \at \d^{b+2}  \nabla^b\bp^\a \left[ \alpha ^2v^k,_r A^r_l  \alpha ^2v^l,_m A^m_j \right] \, dt'' dt'
  \ \nabla^b\bp^\a \operatorname{curl}  \deta dx  \,.
\end{align*} 
Rather than requiring the precise structure of the partial derivatives, it is only the derivative count that is important in estimating $ \mathcal{J} _1^{b,\a}$ and 
$ \mathcal{J} _2^{b,\a}$; hence, we shall write these integrals as
\begin{align*} 
\mathcal{J}_1^{b,\a} &=  \int_{\Omega_0}  \int_0^t \at \d^{b+2} \nabla^b\bp^\a \left[ \alpha ^2 \nabla v \ \nabla \deta \right] dt' 
 \ \nabla^b\bp^\a \operatorname{curl}  \deta dx  \,, \\
\mathcal{J}_2^{b,\a} & 
 =  \int_{\Omega_0}   \int_0^t \at \int_0^{t'} \at \d^{b+2} \nabla^b\bp^\a  \left[ \alpha ^2 \nabla v \P(A)  \alpha ^2 \nabla v \right]  \, dt'' dt'
  \ \nabla^b\bp^\a  \operatorname{curl}  \deta dx  \,,
\end{align*}

\vspace{.1 in}
\noindent
{\it The case that $b=0$.}  We begin with the case that  $1\le a\le 7$, and
\begin{align*} 
\mathcal{J}_1^{0,\a} &=  \int_{\Omega_0}  \int_0^t \alpha^{-2} \d \bp^{\a} \left[ \alpha ^2 \nabla v \ \nabla \deta \right] dt' 
\ \d  \bp^{\a} \operatorname{curl}  \deta dx  \,, \\
\mathcal{J}_2^{0,\a} & 
 =  \int_{\Omega_0}   \int_0^t \at \int_0^{t'} \alpha^{-2} \d  \bp^{\a}  \left[ \alpha ^2 \nabla v \P(A)  \alpha ^2 \nabla v \right]  \, dt'' dt'
  \  \d  \bp^{\a}  \operatorname{curl}  \deta dx  \,.
\end{align*} 
We estimate the case that $\a=7$ whose analysis also covers the cases $1\le \a \le 6$.  We use the  integrability-in-time of $ \alpha ^{-2}$
together with the Cauchy-Schwartz inequality to estimate the integral $\mathcal{J}_1^{0,7}$ as follows:
\begin{align*} 
\mathcal{J}_1^{0,7} \lesssim \sup_{s \in [0,t]}  \left\| \d \bp^{7} [ \alpha ^2 \nabla v \  \nabla \deta]  \right\|_0 
\|  \d  \bp^7 \operatorname{curl}  \deta\|_0
\end{align*} 
Then, using the product rule and H\"{o}lder's inequality, we have that
{\footnotesize
\begin{alignat*}{2}
&\| \d \alpha ^2 \bp^7 \nabla v \,  \ah\nabla \deta\|_0 
 \lesssim  \| \d \alpha ^2 \bp^7 \nabla v \|_0 \  \| \nabla \deta\|_{L^ \infty } \,, \ 
&&\| \d \alpha ^2 \bp^6\nabla  v \, \ah \bp \nabla \deta\|_0 
 \lesssim  \| \d \alpha ^2 \bp^6\nabla v \|_0 \ \| \bp \nabla \deta\|_{L^ \infty } \,,  \\
&\| \d \alpha ^2 \bp^5\nabla  v \, \ah \bp^2 \nabla \deta\|_0 
 \lesssim  \| \alpha ^2 \bp^5\nabla  v \|_0\  \| \d \bp^2 \nabla \deta\|_{L^ \infty }\,,  \ \ 
&&\| \d \alpha ^2 \bp^4 v \, \ah \bp^3 \nabla \deta\|_0 
 \lesssim  \| \alpha ^2 \bp^4 v \|_0  \ \|\d  \bp^3 \nabla \deta\|_{L^ \infty }  \,,\\
&\| \d \alpha ^2 \bp^3 v \, \ah \bp^4 \nabla \deta\|_0 
\lesssim  \| \alpha ^2 \bp^3 v \|_{L^6}  \ \|\d  \bp^4 \nabla \deta\|_{L^ 3 }  \,,
&&\| \d \alpha ^2 \bp^2 v \, \ah \bp^5 \nabla \deta\|_0 
 \lesssim  \| \alpha ^2 \bp^2 v \|_{L^ \infty }  \ \|\d  \bp^5 \nabla \deta\|_0  \,, \\
&\| \d \alpha ^2 \bp v \, \ah \bp^6 \nabla \deta\|_0 
 \lesssim  \| \alpha ^2 \bp v \|_{L^ \infty }  \ \|\d  \bp^6 \nabla \deta\|_0  \,,
&&\| \d \alpha ^2  v \, \ah \bp^7 \nabla \deta\|_0 
 \lesssim  \| \alpha ^2  v \|_{L^ \infty }  \ \|\d  \bp^7 \nabla \deta\|_0   \,,
\end{alignat*} 
}
and hence
$$\mathcal{J}_1^{0,7} \lesssim \sup_{s \in [0,t]} \P(\Ee(s)) \,.$$
The  integral  $\mathcal{J}_2^{0,7}$ is estimated using the same H\"{o}lder pairs.  Thus, with the Cauchy-Schwarz inequality, we have
shown that for $0\le \a\le 7$, 
\begin{align*} 
\|\d  \bp^\a \operatorname{curl}  \deta ( \cdot , t)\|_0^2
&=  \beta(t)\int_{\Omega_0} \d^{2} \bp^\a  \operatorname{curl} u_0  \ \bp^\a \operatorname{curl}  \deta dx  +
\mathcal{J}^{0,\a} _1 +\mathcal{J}^{0,\a}_2  
\lesssim \Ee(0) +  \sup_{s \in [0,t]} \P(\Ee(s)) \,.
\end{align*}

We now examine the difficult case where   $\bp^8$ acts on $ \operatorname{curl} \deta$.
From \eqref{sss2}, we also have that
\begin{align} 
\| \ah \d \bp^8 \operatorname{curl}  \deta ( \cdot , t)\|_0^2
&=  \beta(t)\int_{\Omega_0} \d^{2} \bp^8  \operatorname{curl} u_0  \  \alpha^{-1} \bp^8 \operatorname{curl}  \deta dx  +
\mathcal{J}^{0,8} _1 +\mathcal{J}^{0,8}_2 \,, \label{sss1}
\end{align} 
where
\begin{align*} 
\mathcal{J}_1^{0,8} &=   \int_{\Omega_0}  \int_0^t  \d \ \bp^8 \left[ \nabla v \  \nabla \deta \right] dt' 
 \  \alpha^{-1} \d \bp^8 \operatorname{curl}  \deta dx  \,, \\
\mathcal{J}_2^{0,8} & 
 =  \int_{\Omega_0}   \int_0^t \alpha^{-2} \int_0^{t'}  \d\  \bp^8 \left[ \alpha ^2 \nabla v \P(A)  \nabla v \right]  \, dt'' dt'
  \ \alpha ^{-1}   \d \bp^8 \operatorname{curl}  \deta dx  \,.
\end{align*} 
We begin with the estimate for $\mathcal{J}_1^{0,8} $, which we write as 
\begin{align*} 
\mathcal{J}_{1}^{0,8} &=
\sum_{c=0}^8 \int_{\Omega_0}  \int_0^t \alpha^{-2} \d\ \alpha ^2 \bp^{8-c} \nabla v \ \bp^c \nabla \deta dt' 
 \ \alpha(t)^{-1} \d \  \bp^8 \operatorname{curl}  \deta dx\\
 & 
 =
 \sum_{c=0}^8 \int_{\Omega_0}  \int_0^t \alpha^{-2} \d\ \alpha ^2 \bp^{8-c} \nabla v \ \bp^c \nabla \deta \, 
 \  \alpha(t')^{-1/2} \sqrt{\frac{\alpha (t')}{\alpha (t)} } \, dt' \   \alpha(t)^{-1/2}\d \bp^8 \operatorname{curl}  \deta dx \,.
\end{align*}  
We shall use the that fact $ \sqrt{\frac{\alpha (t')}{\alpha (t)}} < C < \infty $ for all $t \ge 0$ with the constant $C$ independent of $t$ (this follows from the
condition $\dot \alpha(t) \ge 0$).

In the case that $c=8$, 
since $ \alpha ^{-2}$ is integrable in time, 
\begin{align*}
|\mathcal{J}_{1,c=8}^{0,8} | & \lesssim  \sup_{s \in [0,t]}  \|  \ah  \d \bp^{8} \nabla \deta\|_0  \| \alpha ^2 \nabla v\|_{L^ \infty } \|  \ah \d \bp^8 \operatorname{curl}  \deta\|_0 \lesssim \sup_{s \in [0,t]} \P(\Ee(s)) \,.
 \end{align*} 
In the case that $0 \le c\le 7$, 
we integrate-by-parts in time to obtain that  
\begin{align*} 
\mathcal{J}_{1,c}^{0,8} 
&=
- \int_{\Omega_0}  \int_0^t  \d  \bp^{8-c} \nabla \deta \  \bp^c \nabla v \, dt' 
 \  \alpha(t) ^{-1} \d \bp^c \operatorname{curl}  \deta dx 
 +  \int_{\Omega_0}  \d \bp^{8-c} \nabla \deta  \  \bp^c \nabla \deta 
 \  \alpha(t) ^{-1} \d \bp^8 \operatorname{curl}  \deta dx \\
 &=
- \int_{\Omega_0}  \int_0^t  \alpha^{-2} \d \  \alpha(t')^{-1/2}\bp^{8-c} \nabla \deta \  \alpha ^2 \bp^c \nabla v \,  \sqrt{\frac{\alpha (t')}{\alpha (t)} } \, dt' 
 \ \alpha(t) ^{-1/2} \d \bp^8 \operatorname{curl}  \deta dx  \\
 & \qquad\qquad\qquad\qquad
 +  \int_{\Omega_0}   \ah  \d \bp^{8-c} \nabla \deta  \  \bp^c \nabla \deta 
 \ \ah \d \bp^8 \operatorname{curl}  \deta dx \,.
 \end{align*} 
Since $ \alpha ^{-2}$ is integrable in time, it follows that for $c=0,1$,
\begin{align*}
|\mathcal{J}_{1,c}^{0,8} | & \lesssim  \sup_{s \in [0,t]}  \|  \ah  \d \bp^{8-c} \nabla \deta\|_0  \| \alpha ^2 \bp^c \nabla v\|_{L^ \infty } \|  \ah \d \bp^8 \operatorname{curl}  \deta\|_0 \lesssim \sup_{s \in [0,t]} \P(\Ee(s)) \,.
 \end{align*} 
 For $c=2,3,4$, we estimate $  \ah  \d \bp^{8-c} \nabla \deta \in L^6(\Omega_0)$ and $ \alpha ^2 \bp^c \nabla v \in L^3(\Omega_0)$,
for $c=5,6$, we estimate $  \ah  \bp^{8-c} \nabla \deta \in L^6(\Omega_0)$ and $ \alpha ^2\d \bp^c \nabla v \in L^3(\Omega_0)$, and 
for $c=7$, we estimate $  \ah  \bp \nabla \deta \in L^ \infty (\Omega_0)$ and $ \alpha ^2\d \bp^7 \nabla v \in L^2(\Omega_0)$.
 Therefore, $|\mathcal{J}_{1}^{0,8} |  \lesssim \sup_{s \in [0,t]} \P(\Ee(s))$.

 We next consider the integral $\mathcal{J}_2^{0,8} $; in particular, 
we need only analyze the integral
 \begin{align*} 
 \mathcal{J}_{2,(i)}^{0,8} & 
 =   \int_{\Omega_0}   \int_0^t \alpha^{-2} \int_0^{t'}  \d^{2}  \alpha ^2\bp^{8}  \nabla v \P(A)  \nabla v \, dt'' dt'
  \ \alpha(t) ^{-1}  \bp^8 \operatorname{curl}  \deta dx 
\end{align*} 
 as all the other cases are easily estimated by  H\"{o}lder's inequality.   We must again integrate-by-parts in time, and we write the integral as
$
 \mathcal{J}_{2,(i)}^{0,8}  = \mathcal{K} _1 +\mathcal{K} _2 + \mathcal{K} _3$
 where
\begin{align*} 
\mathcal{K} _1 & =  -   \int_{\Omega_0}   \int_0^t \alpha^{-2} \int_0^{t'}  \sqrt{\frac{\alpha(t'')}{\alpha (t)}}  \alpha (t'')^{-1/2}\d  \bp^{8} \nabla \deta \
 \P(A)\  \p_t\left[   \alpha ^2 \nabla v\right] \, dt'' dt'
  \ ( \alpha(t) ^{-1/2} \d \bp^8 \operatorname{curl}  \deta) dx  \,, \\
  \mathcal{K} _2 
  & =
   -   \int_{\Omega_0}   \int_0^t \alpha^{-2} \int_0^{t'}  \sqrt{\frac{\alpha(t'')}{\alpha (t)}}  \alpha (t'')^{-1/2}\d \bp^{8} \nabla \deta \ \P(A)  \ \nabla v \ \alpha ^2 \nabla v \, dt'' dt'
  \ ( \alpha ^{-1/2}  \d \bp^8 \operatorname{curl}  \deta) dx  \,, \\
\mathcal{K} _3 & = 
  \int_{\Omega_0}   \int_0^t \alpha^{-2} \ \     (\sqrt{\frac{\alpha(t')}{\alpha (t)}}  \alpha (t')^{-1/2} \d \bp^{8} \nabla \deta) \ \P(A)   \alpha ^2 \nabla v\, dt'
  \ (\alpha(t) ^{-1/2} \d \bp^8 \operatorname{curl}  \deta )dx \,.
\end{align*} 
The integral $ \mathcal{K} _2 $ and $ \mathcal{K} _3$ are easy to estimate (by virtue of the integrability of $\at$ in time) as follows:
 \begin{align*} 
  \mathcal{K} _2 
  & =
   -   \int_{\Omega_0}   \int_0^t \alpha^{-2} \int_0^{t'}\at \ \sqrt{\frac{\alpha(t'')}{\alpha (t)}}   (\alpha (t'')^{-1/2}  \d \bp^{8} \nabla \deta)  \ \P(A)  \  \alpha ^2 \nabla v \ \alpha ^2 \nabla v \, dt'' dt'
  \ (\ah \d \bp^8 \operatorname{curl}  \deta) dx  \\
  & \le  \sup_{s \in [0,t]}  \| \ah \d \bp^{8} \nabla \deta  \|_0 \| \alpha ^2 \nabla v \|_{L^ \infty }^2 \| \ah \d \bp^8 \operatorname{curl}  \deta\|_0
  \le \sup_{s \in [0,t]} \P(\Ee(s)) \,,
 \end{align*} 
and similarly
$$
\mathcal{K} _3 \le \sup_{s \in [0,t]}  \| \ah \d \bp^{8} \nabla \deta  \|_0 \| \alpha ^2 \nabla v \|_{L^ \infty } \| \ah \d \bp^8 \operatorname{curl}  \deta\|_0
\le \sup_{s \in [0,t]} \P(\Ee(s)) \,.
$$
To estimate $ \mathcal{K} _1$, we write \eqref{for_vorticity} as
$$\alpha^{2.5} \p_t(  \alpha ^2  v) + \ah \deta + \ah x  + 2\ah  A \, \nabla  (\d J ^{-1} )=0 \,,$$
and hence
$$
\alpha^{2.5} \p_t(  \alpha ^2  \nabla v ) = -  \ah \left( \nabla \deta - \operatorname{Id}  +2  \nabla \left(A \,\nabla  (\d J ^{-1} \right) \right)\,.
$$
Substitution of this identity show that the integral $ \mathcal{K}_1 =  \mathcal{K}_{1 (i)}+  \mathcal{K}_{1 (ii)}$, where
\begin{align*} 
\mathcal{K} _{1(i)} & =     \int_{\Omega_0}   \int_0^t \alpha^{-2} \int_0^{t'} \alpha^{-2.5} (\ah \d  \bp^{8} \nabla \deta) \ \P(A) \left[  \ah \nabla \eta \right] \, dt'' dt'
  \ (\ah \d  \bp^8 \operatorname{curl}  \deta) dx  \,, \\
  \mathcal{K} _{1(ii)} & =    2 \int_{\Omega_0}   \int_0^t \alpha^{-2} \int_0^{t'} \alpha^{-2.5} \d^{2}  \bp^{8} \nabla \deta \, \P(A) \, 
  \ah\left[  \nabla \left( A \,\nabla (\d J ^{-1} )\right) -  \operatorname{Id} 
  \right] \, dt'' dt'
  \ \alpha ^{-1}  \bp^8 \operatorname{curl}  \deta dx  \,.
\end{align*} 
Since by the Sobolev embedding theorem,
 $ \sup_{s \in [0,t]}   \| \ah \nabla \eta(\cdot , t)\|_{L^ \infty } \lesssim \sup_{s \in [0,t]}  \Ee^{\frac{1}{2}} (s)$, we see that
 \begin{align*} 
\mathcal{K} _{1(i)} \lesssim  \| \ah \d  \bp^{8} \nabla \deta\|_0 \| \ah  \nabla \eta\|_{L^ \infty }  \| \ah \d  \bp^8 \operatorname{curl}  \deta\|_0
\lesssim \sup_{s \in [0,t]} \P(\Ee(s)) \,.
\end{align*}

Now, with \eqref{grad_d},
\begin{align*} 
2 A \, \nabla  (\d J ^{-1} ) & = -  x \, A\, J ^{-1} - 2\ah  J^{-1} \d\,\nabla _\eta ( \operatorname{div} _\eta \, \eta) \\
&=  x \, A\, J ^{-1} +   x \, (A\, J ^{-1} - \operatorname{Id} ) - 2  J^{-1} \d\,\nabla _\eta ( \operatorname{div} _\eta \, \eta)
\end{align*} 
and so
$$
 \left[ \nabla \left( A \,\nabla (\d J ^{-1} )\right) -\operatorname{Id} \right]= 
 \left[ (A \, J^{ ^{-1} }  - \operatorname{Id} )+ x J^{-1} \P(A) \nabla ^2 \, \eta (1  + \d \nabla ^2 \eta) + 
J^{-1} \P(A) \nabla ( \d \nabla ^2 \eta) \right] \,.
$$
It follows from \eqref{Abound}, \eqref{Jbound}, and \eqref{aJ2} that 
$$
\| \ah \nabla \left( A \,\nabla (\d J ^{-1} )\right)\|_{L^ \infty } \lesssim \ah \left[\| \nabla \deta\|_{L^ \infty } + \| \nabla ^2\deta\|_{L^ \infty } +   \| \nabla ^2\deta\|_{L^ \infty } 
 \| \d \nabla ^2\deta\|_{L^ \infty }  +\| \nabla ( \d \nabla ^2 \deta)\|_{L ^ \infty }\right] \,.
$$
Since $\d \nabla ^2 \deta = \nabla ( \d \nabla \deta ) + {\frac{1}{2}} x\nabla \deta$, then
$$\nabla ( \d \nabla ^2 \deta) = \nabla ^2 (\d \nabla \deta)  +{\frac{1}{2}} x\nabla^2 \deta + {\frac{1}{2}} \nabla \deta$$ 
and hence since
$  \| \alpha^{-1/2} \deta(\cdot ,t)  \|_{4} $ and $\| \alpha^{-1/2} \d \nabla  \deta(\cdot ,t)  \|_{4}$ are both bounded by $\Ee(t)^ {\frac{1}{2}} $, the Sobolev embedding
theorem shows that
$$
\| \ah \nabla \left( A \,\nabla (\d J ^{-1} )\right)\|_{L^ \infty } \lesssim \Ee(t)^ {\frac{1}{2}}  + \Ee(t) \,.
$$
Therefore,
 \begin{align*} 
\mathcal{K} _{1(ii)} \lesssim  \| \ah \d  \bp^{8} \nabla \deta\|_0 \| \ah \nabla \left( A \,\nabla (\d J ^{-1} )\right)\|_{L^ \infty }  \| \ah \d  \bp^8 \operatorname{curl}  \deta\|_0
\lesssim \sup_{s \in [0,t]} \P(\Ee(s)) \,.
\end{align*} 
We have thus estimated the integral $ \mathcal{K} _1$ which together with our bounds for $\mathcal{K} _2 $ and $ \mathcal{K}_3$ shows that
$
 \mathcal{J}_{2,(i)}^{0,8}  \lesssim \sup_{s \in [0,t]} \P(\Ee(s))$, and hence that
$
  \mathcal{J}_{2}^{0,8}  \lesssim \sup_{s \in [0,t]} \P(\Ee(s)$.
 
From \eqref{sss1}, and the Cauchy-Schwarz inequality, we have shown that
\begin{align*} 
\| \ah \d \bp^8 \operatorname{curl}  \deta ( \cdot , t)\|_0^2 \lesssim \Ee(0) + \sup_{s \in [0,t]} \P(\Ee(s)) \,,
\end{align*} 
and this concludes the estimates for the case that $b=0$.

 \vspace{.1 in}
\noindent
{\it The case that $b=1$.} We start with the case that  $0\le a\le 6$, and 
\begin{align*} 
\|\d^{{\frac{3}{2}} }\nabla^b\bp^\a \operatorname{curl}  \deta ( \cdot , t)\|_0^2
&=  \beta(t)\int_{\Omega_0} \d^{3} \nabla^b\bp^\a  \operatorname{curl} u_0  \ \nabla^b\bp^\a \operatorname{curl}  \deta dx  +
\mathcal{J}^{1,\a} _1 +\mathcal{J}^{1,\a}_2 \,,
\end{align*} 
where
\begin{align*} 
\mathcal{J}_1^{1,\a} &=  \int_{\Omega_0}  \int_0^t \alpha^{-2} \d^{3} \nabla\bp^\a \left[ \alpha ^2 \nabla v \ \nabla \deta \right] dt' 
 \ \nabla\bp^\a \operatorname{curl}  \deta dx  \,, \\
\mathcal{J}_2^{1,\a} & 
 =  \int_{\Omega_0}   \int_0^t \at \int_0^{t'} \alpha^{-2} \d^{3} \nabla\bp^\a  \left[ \alpha ^2 \nabla v \P(A)  \alpha ^2 \nabla v \right]  \, dt'' dt'
  \  \nabla\bp^\a  \operatorname{curl}  \deta dx  \,,
\end{align*} 
We analyze the case that $\a=6$, which also covers the cases $0\le \a \le 5$.  Again, we use the  integrability-in-time of $ \alpha ^{-2}$
together with the Cauchy-Schwartz inequality to estimate the integral $\mathcal{J}_1^{1,6}$ as follows:
\begin{align*} 
\mathcal{J}_1^{1,6} \lesssim \sup_{s \in [0,t]}  \left\| \d^{3/2} \nabla  \bp^{6} [ \alpha ^2 \nabla v \ \ah \nabla \deta]  \right\|_0 
\|  \d^{3/2}\nabla  \bp^{6} \operatorname{curl}  \deta\|_0 \,.
\end{align*} 
Next, we write
$$
 \d^{3/2} \nabla  \bp^{6} [ \alpha ^2 \nabla v \ \ah \nabla \deta] 
 =  \d^{3/2}  \bp^{6} [ \alpha ^2 \nabla^2 v \ \ah \nabla \deta]  + \d^{3/2}  \bp^{6} [ \alpha ^2 \nabla v \ \ah \nabla^2 \deta]   \,.
$$
We estimate the $L^2$-norm of  $\d^{3/2}  \bp^{6} [ \alpha ^2 \nabla^2 v \ \ah \nabla \deta] $ as follows:
{\footnotesize
\begin{alignat*}{2}
&\| \d^{3/2} \alpha ^2 \bp^6 \nabla^2 v \,  \nabla \deta\|_0 
 \lesssim  \| \d^{3/2} \alpha ^2 \bp^6 \nabla^2 v \|_0 \  \|  \nabla \deta\|_{L^ \infty } \,, \ \ 
&&\|  \alpha ^2\d^{3/2} \bp^5 \nabla^2 v \,   \bp \nabla \deta\|_0 
 \lesssim  \|\alpha ^2 \d^{3/2}  \bp^5 \nabla^2 v \|_0 \  \|  \bp \nabla \deta\|_{L^ \infty }  \,, \\
&\| \d^{3/2} \alpha ^2 \bp^4 \nabla^2 v \,   \bp^2 \nabla \deta\|_0 
 \lesssim  \| \alpha ^2  \d \bp^4 \nabla^2 v \|_{L^3} \  \|  \bp^2 \nabla \deta\|_{L^ 6 }   \,,  \ \ \ 
&&\| \d^{3/2} \alpha ^2 \bp^3 \nabla^2 v \,   \bp^3 \nabla \deta\|_0 
 \lesssim  \| \alpha ^2 \bp^3 \nabla^2 v \|_0 \  \| \d  \bp^3 \nabla \deta\|_{L^ \infty }   \,, \\
&\| \d^{3/2} \alpha ^2 \bp^2 \nabla^2 v \,    \bp^4 \nabla \deta\|_0 
 \lesssim  \| \alpha ^2 \bp^2 \nabla^2 v \|_{L^3} \  \| \d  \bp^4 \nabla \deta\|_{L^ 6}   \,,  \ \ \ 
&&\| \d^{3/2} \alpha ^2 \bp \nabla^2 v \,    \bp^5 \nabla \deta\|_0 
 \lesssim  \| \alpha ^2 \bp \nabla^2 v \|_{L^ 3} \  \| \d \bp^5 \nabla \deta\|_{L^6}    \,, \\
&\| \d^{3/2} \alpha ^2 \bp^2 \nabla^2 v \,   \bp^6 \nabla \deta\|_0 
 \lesssim  \| \alpha ^2  \nabla^2 v \|_{L^ 6 } \  \| \d  \bp^6 \nabla \deta\|_{L^3}  \,.
\end{alignat*}
}
The  $L^2$-norm of  $\d^{3/2}  \bp^{6} [ \alpha ^2 \nabla v \ \ah \nabla^2 \deta] $ is estimated using the same H\"{o}lder pairs. Hence,
$\mathcal{J}_1^{1,6}  \lesssim \Ee(0) + \sup_{s \in [0,t]} \P(\Ee(s)) $.
 We use H\"{o}lder's inequality in the
identical way to estimate the integral $\mathcal{J}_2^{1,6}$ as well.  Hence, with Cauchy-Schwarz,   for $0\le a\le 6$, 
\begin{align*} 
\|\d^{{\frac{3}{2}} }\nabla^b\bp^\a \operatorname{curl}  \deta ( \cdot , t)\|_0^2 \lesssim \Ee(0) + \sup_{s \in [0,t]} \P(\Ee(s)) \,.
\end{align*} 
In the hardest case that $\a=7$, we have that
\begin{align} 
\| \ah \d ^{3/2} \bp^7 \nabla  \operatorname{curl}  \deta ( \cdot , t)\|_0^2
&=  \beta(t)\int_{\Omega_0} \d^{3}  \nabla\bp^7 \operatorname{curl} u_0  \  \ah \nabla \bp^7 \operatorname{curl}  \deta dx  +
\mathcal{J}^{1,7} _1 +\mathcal{J}^{1,7}_2 \,, \label{sss3}
\end{align} 
where
\begin{align*} 
\mathcal{J}_1^{1,7} &=  \int_{\Omega_0}  \int_0^t  \d^{3} \bp^7 \left[ \nabla^2 v \ \nabla \deta +  \nabla v \ \nabla^2 \deta\right] dt' 
 \ \alpha  ^{-1} \nabla\bp^7 \operatorname{curl}  \deta dx  \,, \\
\mathcal{J}_2^{1,7} & 
 =  \int_{\Omega_0}   \int_0^t \at \int_0^{t'}  \d^{3} \bp^7  \left[ \alpha ^2 \nabla^2 v \P(A)   \nabla v +  \nabla ^2\deta \  \alpha ^2 \nabla v\,  \P(A)\,   \nabla v   \right]  \, dt'' dt'
  \ \alpha ^{-1}  \nabla\bp^7  \operatorname{curl}  \deta dx  \,
\end{align*} 
The most challenging terms in  $\mathcal{J}_1^{1,7} $ are written as
\begin{align*} 
\mathcal{J}_{1(i)}^{1,7} 
&=  \sum_{c=0}^7\int_{\Omega_0}  \int_0^t  \d^{3/2} \ \left[ \bp^{7-c}   \nabla^2 v \  \bp^c \nabla \deta \right] dt' 
 \ \alpha  ^{-1} \d^{3/2} \nabla\bp^7 \operatorname{curl}  \deta dx  \,.
\end{align*} 
Following our analysis of the term $\mathcal{J}_{1}^{0,8} $ given above, we integrate-by-parts in time, and find that
\begin{align*} 
\mathcal{J}_{1(i)}^{1,7} 
&= - \sum_{c=0}^7\int_{\Omega_0}  \int_0^t  \alpha^{-2}  \d^{3/2} \ \left[ \sqrt{ \frac{\alpha (t')}{\alpha(t)}} \ah(t') \bp^{7-c} \nabla^2 \deta \ \alpha^2 \bp^c \nabla v \right] dt' 
 \ \alpha  ^{-1/2} \d^{3/2} \nabla\bp^7 \operatorname{curl}  \deta dx   \\
 & \qquad \qquad + \sum_{c=0}^7\int_{\Omega_0}   \d^{3/2} \ \left[ \ah \bp^{7-c} \nabla^2 \deta \ \bp^c \nabla \deta \right] 
 \ \alpha  ^{-1/2} \d^{3/2} \nabla\bp^7 \operatorname{curl}  \deta dx   \,.
\end{align*}
Due to the time integrability of $\alpha^{-2}$, we have that
\begin{align*}
\mathcal{J}_{1(i)}^{1,7} & \lesssim \sup_{s \in [0,t]}   \left\|  \d^{3/2} \ \left[ \ah \bp^{7-c} \nabla^2 \deta \ \alpha^2 \bp^c \nabla v \right] \right\|_0 
\| \alpha  ^{-1/2} \d^{3/2} \nabla\bp^7 \operatorname{curl}  \deta\|_0 \\
& \qquad + \sup_{s \in [0,t]}    \left\|  \d^{3/2} \ \left[ \ah \bp^{7-c} \nabla^2 \deta \ \bp^c \nabla \deta \right]\right\|_0
\| \alpha  ^{-1/2} \d^{3/2} \nabla\bp^7 \operatorname{curl}  \deta\|_0  \,.
 \end{align*} 
Since $\alpha^2 \bp^c \nabla v $ and $ \bp^c \nabla \deta$ have the same space regularity,  
it suffices to estimate the $L^2(\Omega_0)$-norm of $  \d^{3/2} \ \left[ \ah \bp^{7-c} \nabla^2 \deta \ \alpha^2 \bp^c \nabla v \right]$ as follows: 
using 
H\"{o}lder's inequality, if $c=0,1$, then we estimate $ \ah\d^{3/2} \bp^{7-c} \nabla^2 \deta \in L^2(\Omega_0)$ and $\alpha^2 \bp^c \nabla v  \in
L^ \infty (\Omega_0)$.  If $c=2,3$, then we estimate $ \ah\d \bp^{7-c} \nabla^2 \deta \in L^3(\Omega_0)$ and $\alpha^2 \bp^c \nabla v  \in
L^6 (\Omega_0)$.   If $c=4$, then we estimate $ \ah\d \bp^{3} \nabla^2 \deta \in L^6(\Omega_0)$ and $\alpha^2 \bp^4 \nabla v  \in
L^3 (\Omega_0)$.   If $c=5$, then we estimate $ \ah\d \bp^{2} \nabla^2 \deta \in L^ \infty (\Omega_0)$ and $\alpha^2 \bp^5 \nabla v  \in
L^2 (\Omega_0)$.   If $c=6$, then we estimate $ \ah\bp  \nabla^2 \deta \in L^6 (\Omega_0)$ and $\alpha^2 \d \bp^6 \nabla v  \in
L^3 (\Omega_0)$.  Finally, if  $c=7$, then we estimate $ \ah \nabla^2 \deta \in L^ \infty  (\Omega_0)$ and $\alpha^2 \d \bp^7 \nabla v  \in
L^2 (\Omega_0)$.    This shows that $\mathcal{J}_1^{1,7} \lesssim \sup_{s \in [0,t]} \P(\Ee(s)) $.

The same integration-by-parts in time argument used to estimate $ \mathcal{J} ^{0,8}_2$ is used to estimate the integral $\mathcal{J} ^{1,7}_2$.   The use
of H\"{o}lder's inequality is the same as just detailed for  $\mathcal{J}_1^{1,7} $, and we find that $\mathcal{J}_2^{1,7} \lesssim \sup_{s \in [0,t]} \P(\Ee(s)) $.
Hence, from \eqref{sss3}, we find that 
$$
\| \ah \d ^{3/2} \bp^7 \nabla  \operatorname{curl}  \deta ( \cdot , t)\|_0^2 \le \Ee(0) + \sup_{s \in [0,t]} \P(\Ee(s)) \,.
$$

 \vspace{.1 in}
\noindent
{\it The cases $2\le b\le 7$.} These cases follow identically to the cases $b=0$ and $b=1$ detailed above.

 \vspace{.1 in}
\noindent
{\it The case $b=8$.}  In this case, 
\begin{align} 
\|\ah \d^{5}\nabla^b\bp^\a \operatorname{curl}  \deta ( \cdot , t)\|_0^2
&=  \beta(t)\int_{\Omega_0} \d^{10} \nabla^8  \operatorname{curl} u_0  \  \alpha  ^{-1} \nabla^8  \operatorname{curl}  \deta dx  +
\mathcal{J}^{8,0} _1 +\mathcal{J}^{8,0}_2 \,, \label{sss4}
\end{align} 
where
\begin{align*} 
\mathcal{J}_1^{8,0} &=  \int_{\Omega_0}  \int_0^t \at \d^{5} \nabla^8  \left[ \alpha ^2 \nabla v \ \nabla \deta \right] dt' 
 \ \alpha^{-1} \d^5 \nabla^8 \operatorname{curl}  \deta dx  \,, \\
\mathcal{J}_2^{8,0} & 
 =  \int_{\Omega_0}   \int_0^t \at \int_0^{t'} \at \d^{5} \nabla^8  \left[ \alpha ^2 \nabla v \P(A)  \alpha ^2 \nabla v \right]  \, dt'' dt'
  \ \alpha^{-1}  \d^5 \nabla^8   \operatorname{curl}  \deta dx  \,.
\end{align*} 
We begin with the estimate for $\mathcal{J}_1^{8,0} $, which we write as
\begin{align*} 
\mathcal{J}_{1}^{8,0} &=
\sum_{c=0}^8 \int_{\Omega_0}  \int_0^t \alpha^{-2} \d^5\alpha ^2\nabla^{9-c}  v \ \alpha(t) ^{-1/2}     \nabla^{1+c} \deta \, \sqrt{\frac{ \alpha (t') }{\alpha (t)}}\, dt' 
 \ \alpha(t) ^{-1/2}  \d^5 \nabla ^8 \operatorname{curl}  \deta dx\,.
\end{align*}  
In the case that $c=8$, 
since $ \alpha ^{-2}$ is integrable in time, 
\begin{align*}
|\mathcal{J}_{1,c=8}^{8,0} | & \lesssim  \sup_{s \in [0,t]}  \|  \ah  \d^5 \nabla^{9} \deta\|_0  \| \alpha ^2 \nabla v\|_{L^ \infty } \|  \ah \d^5 \nabla^8 \operatorname{curl}  \deta\|_0 \lesssim \sup_{s \in [0,t]} \P(\Ee(s)) \,.
 \end{align*} 
In the case that $0 \le c\le 7$, 
we integrate-by-parts in time to obtain that  
\begin{align*} 
\mathcal{J}_{1,c}^{8,0} 
 &=
- \int_{\Omega_0}  \int_0^t  \alpha^{-2} \d^5  \alpha(t')^{-1/2} \nabla^{9-c} \deta \  \alpha ^2 \nabla^{1+c}  v \, \sqrt{\frac{ \alpha (t') }{\alpha (t)}} \,  dt' 
 \ \alpha ^{-1} \d^5 \nabla^8 \operatorname{curl}  \deta dx  \\
 & \qquad\qquad\qquad\qquad
 +  \int_{\Omega_0}   \ah  \d^5  \nabla^{9-c} \deta  \  \nabla^{1+c}  \deta 
 \ \ah \d^5 \nabla^8 \operatorname{curl}  \deta dx  \\
 & \lesssim \sup_{s \in [0,t]}   \left( \left\|   \ah \d^5  \nabla^{9-c} \deta \  \alpha ^2 \nabla^{1+c} v  \right\|_0
  + \|  \ah  \d^5  \nabla^{9-c} \deta  \  \nabla^{1+c}  \deta  \|_0\right) \|  \ah \d^5 \nabla^8 \operatorname{curl}  \deta \|_0 \,.
 \end{align*} 
 Since $ \alpha ^2 \nabla^{1+c} v$ and $ \nabla^{1+c}  \deta$ have the same space regularity, it suffices to explain the estimate for the $L^2$-norm
 of $ \ah \d^5  \nabla^{9-c} \deta \  \alpha ^2 \nabla^{1+c} v $.  As $0\le c \le 7$ changes, we use the following H\"{o}lder pairs:
 {\footnotesize
 \begin{alignat*}{2}
 & \ah \d^5  \nabla^{9} \deta  \in L^2\,,  \alpha ^2 \nabla v \in L^ \infty\,, \ \ \   &&\ah \d^4  \nabla^8 \deta  \in L^2\,,  \alpha ^2\d \nabla^2 v \in L^ \infty\,, \\
 & \ah \d^3  \nabla^7 \deta  \in L^2\,,  \alpha ^2\d^2 \nabla^3 v \in L^ \infty\,,  \ \ \   &&\ah \d^3  \nabla^6 \deta  \in L^6\,,  \alpha ^2\d^2 \nabla^4 v \in L^3\,, \\
  & \ah \d^3  \nabla^5 \deta  \in L^6\,,  \alpha ^2\d^2 \nabla^5 v \in L^3\,,  \ \ \   &&\ah \d^2  \nabla^4 \deta  \in L^\infty \,,  \alpha ^2\d^{2.5} \nabla^6 v \in L^2\,, \\
  &\ah \d  \nabla^3 \deta  \in L^\infty \,,  \alpha ^2\d^{3.5} \nabla^7 v \in L^2\,, \ \ \ && \ah   \nabla^2 \deta  \in L^\infty \,,  \alpha ^2\d^{4.5} \nabla^8 v \in L^2\,.
\end{alignat*} 
}
where we have used the fact that both $\d^{2.5} \nabla ^6 v( \cdot , t)$ and $\d^{3.5} \nabla ^7 v( \cdot , t)$ are bounded by $\Ee(t)^ {\frac{1}{2}} $ by the 
weighted embedding \eqref{w-embed}.   It follows that $\mathcal{J}_1^{8,0} \lesssim \sup_{s \in [0,t]} \P(\Ee(s)) $,
 
Again, the same integration-by-parts in time argument used to estimate $ \mathcal{J} ^{0,8}_2$ is used to estimate the integral $\mathcal{J} ^{8,0}_2$.   The use
of H\"{o}lder's inequality is the same as just detailed for  $\mathcal{J}_1^{8,0} $, and we find that $\mathcal{J}_2^{8,0} \lesssim \sup_{s \in [0,t]} \P(\Ee(s)) $.
Hence, from \eqref{sss4}, 
$$
\|\ah \d^{5}\nabla^b\bp^\a \operatorname{curl}  \deta ( \cdot , t)\|_0^2 \le \Ee(0) + \sup_{s \in [0,t]} \P(\Ee(s)) \,.
$$
We have thus established a priori estimates for weighted derivatives of $ \operatorname{curl} \deta$.
 \end{proof} 

\subsubsection{Statement and proof of the stability theorem for $\gamma=2$} Having established weighted estimates for derivatives of 
$ \operatorname{curl} \deta$, we now turn to energy estimates to establish the existence of global-in-time perturbations to the affine flow.   There are
two fundamental ideas for energy estimates: first, we use Lemma \ref{lemma_aenergy}, which shows that for a differential operator $D$, the
structure $ D a^k_i  D v^i,_k $ produces the highest-order energy contribution modulo curl estimates that have already been established.   As to the
energy estimates, in order to build the regularity of $\deta$ and $v$, one could study a succession of time-differentiated problems as was done
in \cite{CoLiSh2010, CoSh2012} or use a succession of higher-order weighted space derivatives with the power of the weight increasing with the
derivative count as was done in \cite{JaMa2015, HaJa2016}.   Because our equation for the perturbations has temporal weights, it appears more natural
to use the method of  higher-order weighted space derivatives as in  \cite{JaMa2015, HaJa2016}.

Before stating the global existence and stability theorem, we establish   two useful identities.

\begin{lemma} \label{cor1} We have that
$$
-\left[\d ^{-1} \left( \d^2 W\right),_k\right],_{r_1} = -\d^{-2} \left( \d^3 W,_{r_1}\right),_k - \left[ \d,_{r_1} W,_k - \d,_k W,_{r_1}\right] + \delta_{k r_1} W \,,
$$
and
 for $s \ge 2$, 
\begin{align} 
&
-\left[ \d ^{-1} \left(\d^2 W\right),_k\right],_{r_1 \ddd r_s}
=
-\d^{-(s+1)} \left( \d^{(s+2)} W,_{r_1\ddd r_s}\right),_k + {\frac{s+1}{2}}  \delta_{k r_s} W,_{r_1 \ddd r_{s-1}} \nonumber \\
&\qquad
+ {\frac{s}{2}}  \delta_{k r_{s-1}} W,_{r_1 \ddd r_{s-2}r_s}
+ {\frac{s-1}{2}}   \delta_{k r_{s-2}} W,_{r_1 \ddd r_{s-3} r_{s-1}r_s} + \ddd +  \delta_{k r_1} W,_{r_2\ddd r_s}\nonumber \\
&\qquad
+ \d,_k   W,_{r_1\ddd r_s} - \d,_{r_s}   W,_{r_1 \ddd r_{s-1} k} 
+ \p_{r_s} \left[ \d,_k  W,_{r_1 \ddd r_{s-1}} - \d,_{r_{s-1}}   W,_{r_1\ddd r_{s-2} k}  \right] \nonumber \\
&\qquad
+ \p_{r_{s-1} r_s} \left[ \d,_k  W,_{r_1 \ddd r_{s-2}} - \d,_{r_{s-2}}   W,_{r_1 \ddd r_{s-3}k}  \right] 
\nonumber \\
&\qquad  \qquad
+ \ddd 
+ \p_{r_2 \ddd  r_s} \left[ \d,_k   W,_{r_1 } - \d,_{r_1}   W,_{ k}  \right]  \,, \nonumber
\end{align} 
where $ \delta_{k0}=0$ and $W_{r_1 \ddd j} =0$ for $j \le 0$.
\end{lemma} 
This lemma is very similar to Lemma 4.4 in \cite{HaJa2016}.
\begin{proof} 
 For integers $\ell \ge 2$,
\begin{equation}\label{ss100}
- \left[ \d^{-\ell+1} \left( \d^\ell w\right),_i\right],_m = - \d^{-\ell} \left(\d^{\ell+1} w,_m\right),_i + {\frac{\ell}{2}} \delta_{im} w + \d,_i w,_m - \d,_m w,_i \,.
\end{equation} 
Indeed, this follows from the  following simple computation:
\begin{align*} 
\left[ \d^{-\ell+1} \left( \d^\ell w\right),_i\right],_m & = \d^{-\ell+1} (\d^\ell w),_{im} - (\ell-1) \d^{-\ell} \d,_m ( \d^\ell w),_i \\
 &=\d^{-\ell+1} (\d^\ell w,_m),_{i}  + \ell \d^{-\ell+1}( \d^{\ell+1} \d,_m w),_i- (\ell-1) \d^{-\ell} \d,_m ( \d^\ell w),_i\\
 &= \d^{-\ell} (\d^{\ell+1} w,_m),_i - \d,_i w,_m + \ell \d,_m w,_i + (1-\ell) \d,_m w,_i \\
 &\ + \ell(\ell-1) \d^{-\ell+1} \d^{\ell-2} \d,_i d,_m w - {\frac{\ell}{2}} \delta_{im} w - \ell(\ell-1) \d^{-\ell} \d^{\ell-1} \d,_i\d,_m w
\end{align*} 
with cancellations in the last equality yielding \eqref{ss100}.

By repeated application of this identity, the lemma is proved.
\end{proof}

\begin{lemma}\label{lemma_tan}
For each  anti-symmetric matrix $M:= M^i_r$, we associate the vector 
$$\M= \left( M^3_2 - M^2_3\,, M^1_3 - M^3_1\,, M^2_1 - M^1_2\right) \,.$$
Then for a differentiable scalar function $f$,
$$
\left[ \d,_i f,_r - \d,_r f,_i\right]\, M^i_r = {\frac{1}{2}} \bp f \cdot \M \,.
$$
\end{lemma} 
\begin{proof} 
Using \eqref{grad_d},  
\begin{align*} 
&2\left[ \d,_i f,_r - \d,_r f,_i\right]\, M^i_r \\
&
= (x_2 f,_3 - x_3 f,_2) (M^3_2 - M^2_3) + (x_3 f,_1 - x_1 f,_3) (M^1_3 - M^3_1) + (x_1 f,_2 - x_2 f,_1) (M^2_1 - M^1_2)  \,,
\end{align*} 
which is the desired identity.
\end{proof}

\begin{theorem}[Perturbations of simple affine motion for $\gamma=2$]\label{mainthm}  For $\gamma=2$ and  $ \epsilon >0$ taken sufficiently small, if the initial data satisfies
 $\Ee(0) <  \epsilon $, then there exists a global  solution $\eta(x,t)$ of \eqref{ceuler03D}  such that $\Ee(t) \le \epsilon $ for all $t \ge 0$.    In particular,
 the flow map $\eta \in C^0([0, \infty ), H^4(\Omega_0))$ and the moving vacuum boundary $\Gamma(t)$ is of class $H^{3.5}$ for all $t\ge 0$.
 The composition
 $\xi(x,t)= \alpha (t) \eta(x,t)$ defines a  global-in-time solution of  the Euler equations \eqref{ceuler3d}.   
  If the initial data satisfies $\Ee^9(t) \le C$, then the solution $ \xi (x,t)$ is unique.
\end{theorem} 
\begin{proof} 
{\it {Step 0. Outline of the proof.}}
The energy estimates and hence the proof of the theorem are established as follows: for $b=0,...,8$ and $\a=0,...,8-b$, we let the operator $\bp ^\a \nabla ^b$ act on  equation \eqref{momentum3d} (below), multiply by the test function
$\d^{b+1}\bp ^\a \nabla ^b v^i$ and integrate over $\Omega_0$.    As $b$ increases, the power on the weight $\d^{b+1}$ increases, and we use 
Lemma \ref{cor1} to ensure that integration-by-parts does not place any derivatives on the weight in the test function at the highest-order.   The formation 
of the fundamental energy
terms produces error terms containing derivatives of highest-order, but these errors terms are either of curl type or of tangential derivative type, and Lemmas
\ref{lemma_atan}, \ref{lemma_curl},  and \ref{lemma_tan} are used for their bounds.   The control of the remainder of the error terms simply requires the use of H\"{o}lder's inequality together
with weighted embeddings and the Sobolev embedding theorem.   Because the exact  embedding inequality may  change with the integers $b$ and $\a$, we detail
these estimates by starting with the zeroth-order energy and systematically increasing the derivative count until we reach the highest-order energy terms.

\vspace{.2in}
\noindent
{\it {Step 1. A smooth sequence of solutions.}}
We consider initial conditions (\ref{ceuler0}c) given by
\begin{equation}\label{smooth1ii}
(\eta, v)  =( e, u_0^\kappa) \ \  \ \  \text{ in } \Omega_0   \times \{t=0\}   \,,
\end{equation} 
where $u_0^\kappa$ is given by \eqref{u0kappa}.  Notice that \eqref{ceuler0} is equivalent to  \eqref{ce03d};  by the local-in-time well-posedness theorem 
 in \cite{CoSh2012}, using the smooth initial data \eqref{smooth1ii}, there exist a smooth solution $(\eta_\kappa, v_\kappa)$ on a short time-interval $[0,T]$ 
 (where $T$ depends on $\kappa$) such that for each $t \in [0,T]$,
 $\eta_\kappa( \cdot , t) \in H^9(\Omega_0)$,  $v_\kappa( \cdot , t) \in H^8(\Omega_0)$, and $\p_t v_\kappa( \cdot , t) \in H^7(\Omega_0)$.   We are thus able to consider higher-order energy estimates, as we can rigorously differentiate the sequence $(\eta_\kappa, v_\kappa)$ as many time as required.   Again, for notational
 clarity, we will drop the subscript $\kappa$ and simply write $\eta$ and $v$.

\vspace{.2in}
\noindent
{\it {Step 2.   Zeroth-order estimate.}}   
 We write the momentum equation (\ref{ceuler03D}) as
\begin{equation} 
\alpha^4\p_t v^i + 2\alpha ^3 \dot\alpha v^i +  \alpha ^{-1} \eta^i +\alpha^{-1}  \d ^{-1} a^k_i  \left(\d^2 J^{-2}\right),_k =0 \ \  \text{ in } \Omega_0   \times (0,T] \,.
\label{momentum3d}
\end{equation}

We begin with the zeroth-order  $L^2$ (physical)  energy estimate; testing \eqref{momentum3d} against $\d v^i$, we find that
\begin{align*} 
&{\frac{1}{2}}  \frac{d}{dt} \left( \| \alpha ^2 \d^{1/2} v\|_0^2 +     \| \alpha^{-1/2} \d^{1/2} \eta\|_0^2 
\right) 
+ {\frac{\dot\alpha}{2}} \|\alpha^{-1} \d^{1/2} \eta\|_0^2
-
 \alpha ^{-1}  \int_{\Omega_0} \d^2 J^{-2} v^i,_k \aki \, dx =0\,.
\end{align*} 
Since $v^i,_k \aki = \p_t J$, we have that
\begin{align*} 
&{\frac{1}{2}}  \frac{d}{dt} \left( \| \alpha ^2 \d^{1/2} v\|_0^2 +     \| \alpha^{-1/2} \d^{1/2} \eta\|_0^2 
\right) 
+ {\frac{\dot\alpha}{2}} \|\alpha^{-1} \d^{1/2} \eta\|_0^2
-
 \alpha ^{-1} \frac{d}{dt}  \int_{\Omega_0} \d^2 J^{-1}\, dx =0\,.
\end{align*} 
Commuting $\frac{d}{dt} $ with $ \alpha  ^{-1} $ and integrating in time from $0$ to $t$, we find that 
\begin{align*}
&
 \| \alpha ^2 \d^{1/2} v\|_0^2 +     \| \alpha^{-1/2} \d^{1/2} \eta\|_0^2 + \alpha ^{-1} \int_{\Omega_0} \d^2 J ^{-1} \, dx
 \lesssim \Ee(0)   \,.
 \end{align*} 
 Now, since 
 $$
\left|\deta (x,t)\right| = \left|\int_0^t  v(x,s) ds\right| \le \int_0^t\alpha^{-2}(s) ds  \sup_{s \in [0,t]} \alpha ^2(s) |v(x,s|
\lesssim  \sup_{s \in [0,t]} \alpha ^2(s) |v(x,s|  \,,
 $$
 it follows that
 \begin{align}
&\|  \d^{1/2} \deta \|_0^2
+ \| \alpha ^2 \d^{1/2} v\|_0^2 +     \| \alpha^{-1/2} \d^{1/2} \eta\|_0^2 +  \alpha ^{-1}  \int_{\Omega_0} \d^2 J ^{-1} \, dx
 \lesssim \Ee(0)   \,. \label{estd0}
 \end{align} 

 We next let tangential derivatives act on the momentum equation.  We first detail the case that only one tangential derivative enters the problem.
 Thus, we let  $\bp$ act on \eqref{momentum3d} and since
 $\bp_i \d=0$, we obtain that
\begin{align} 
&
 \alpha^4\bp v_t^i + 2\alpha ^3 \dot\alpha \bp v^i + \alpha ^{-1}  \bp\deta^i   
 +\alpha^{-1}  \d ^{-1} \bp a^k_i  \left(\d^2 J^{-2}\right),_k + \alpha^{-1}  \d ^{-1} a^k_i  \left(\d^2 \bp J^{-2}\right),_k =0 
  \,. \label{ss700}
\end{align} 
Testing \eqref{ss700} against $ \d \bp v^i$, we find that 
\begin{align*} 
&{\frac{1}{2}}  \frac{d}{dt} \left( \| \alpha ^2 \d^{1/2} \bp v\|_0^2 +     \| \alpha^{-1/2} \d^{1/2}\bp  \eta\|_0^2 
\right) 
+ {\frac{\dot\alpha}{2}} \|\alpha^{-1} \d^{1/2}\bp  \eta\|_0^2  + \mathcal{I} ^{0,\bp^1}_1 + \mathcal{I} ^{0,\bp^1}_2 =0 \,,
\end{align*} 
where
\begin{align*} 
\mathcal{I} ^{0,\bp^1}_1 =
- \alpha ^{-1}  \int_{\Omega_0} \d^2 J^{-2} \bp \aki \bp v^i,_k  \, dx \ \text{ and } \  \mathcal{I} ^{0,\bp^1}_2 = 
 -\alpha ^{-1}  \int_{\Omega_0} \d^2  \bp J^{-2}  \bp v^i,_k \aki \, dx \,.
\end{align*} 
Using the identity \eqref{a1} together with Lemma \ref{lemma_aenergy},
\begin{align*}
\mathcal{I} ^{0,\bp^1}_1
&=
\frac{\alpha^{-1} }{2} \int_{\Omega_0} \d^2 J ^{-1} \frac{d}{dt} \left( | \nabla _\eta \  \bp \eta|^2  - |\operatorname{div}  _\eta \  \bp \eta|^2
- 2|\operatorname{curl}  _\eta \  \bp \eta|^2
\right)\, dx \,,
\end{align*} 
and with \eqref{J1},
\begin{align*}
\mathcal{I} ^{0,\bp^1}_2
&= 2 \alpha ^{-1}  \int_{\Omega_0} \d^2 J^{-3} \bp \eta^r,_s a^s_r \ \bp v^i,_k  \aki \, dx
= \alpha^{-1} \int_{\Omega_0} \d^2 J^{-1} \frac{d}{dt}  | \operatorname{div} _\eta \ \bp \eta|^2\, dx
 \,.
\end{align*} 
Hence,
\begin{align*} 
&{\frac{1}{2}}  \frac{d}{dt} \left( \| \alpha ^2 \d^{1/2} \bp v\|_0^2 +     \| \alpha^{-1/2} \d^{1/2}\bp  \eta\|_0^2 
\right) 
+ {\frac{\dot\alpha}{2}} \|\alpha^{-1} \d^{1/2}\bp  \eta\|_0^2 \\
& \qquad
+
{\frac{\alpha^{-1} }{2}} \int_{\Omega_0} \d^2 J ^{-1} \frac{d}{dt} \left( | \nabla _\eta \  \bp \eta|^2  + |\operatorname{div}  _\eta \  \bp \eta|^2
- 2 |\operatorname{curl}  _\eta \  \bp \eta|^2
\right)\, dx =0
\,.
\end{align*} 
Commuting $ \frac{d}{dt} $ with $ \alpha  ^{-1} $ and integrating in time from $0$ to $t$, we then have that
\begin{align*} 
& \| \alpha ^2 \d^{1/2} \bp v\|_0^2 +     \| \alpha^{-1/2} \d^{1/2}\bp  \eta\|_0^2 
+
\alpha^{-1}  \int_{\Omega_0} \d^2  \left( | \nabla _\eta \  \bp \eta|^2  + |\operatorname{div}  _\eta \  \bp \eta|^2
\right)\, dx  \\
& \qquad
\lesssim \Ee(0) + \alpha^{-1}  \int_{\Omega_0} \d^2  | \operatorname{curl}  _\eta \  \bp \eta|^2  \, dx 
+ \int_0^t \alpha^{-2}\dot\alpha  \int_{\Omega_0} \d^2  | \operatorname{curl}  _\eta \  \bp \eta|^2  \, dx ds \\
& \qquad \qquad
+ \int_0^t \alpha^{-3} \int_{\Omega_0} J ^{-1}| \P(A)| \,  | \nabla \bp \eta |^2  \, \alpha ^2 |\nabla v|\, dx ds \,,
\end{align*} 
and thanks to Lemma \ref{lemma_curl} and the fact that
 $\left|\bp \deta (x,t)\right| 
\lesssim  \sup_{s \in [0,t]} \alpha ^2(s) |\bp v(x,s)|$,
\begin{align} 
&\|  \d^{1/2} \bp \deta\|_0^2 + \| \alpha ^2 \d^{1/2} \bp v\|_0^2 +     \| \alpha^{-1/2} \d^{1/2}\bp  \eta\|_0^2 
+     \| \alpha^{-1/2} \d \nabla_\eta \bp  \eta\|_0^2 +     \| \alpha^{-1/2} \d \operatorname{div} _\eta \bp  \eta\|_0^2  \nonumber
  \\
& \qquad
\lesssim \Ee(0) + \sup_{s \in [0,t]} \P( \Ee(s)) \,. \nonumber
\end{align} 
Finally,  the estimate \eqref{Abound} shows that we can replace $ \nabla _\eta \  \bp \eta$ above with $ \nabla  \  \bp \eta$:
\begin{align} 
&\|  \d^{1/2} \bp \deta\|_0^2 + \| \alpha ^2 \d^{1/2} \bp v\|_0^2 +     \| \alpha^{-1/2} \d^{1/2}\bp  \eta\|_0^2 
+     \| \alpha^{-1/2} \d \nabla \bp  \eta\|_0^2 +     \| \alpha^{-1/2} \d \operatorname{div} _\eta \bp  \eta\|_0^2  \nonumber
  \\
& \qquad
\lesssim \Ee(0) + \vartheta \sup_{s \in [0,t]}\Ee(s) +   \sup_{s \in [0,t]} \P( \Ee(s)) \,. \label{estd0bp}
\end{align} 
Notice that in terms of the derivative count, the difference
$
\nabla \bp \deta - \bp \nabla \deta$ scales like  $\nabla \deta $,
so that we also have that
\begin{align} 
&\|  \d^{1/2} \bp \deta\|_0^2 + \| \alpha ^2 \d^{1/2} \bp v\|_0^2 +     \| \alpha^{-1/2} \d^{1/2}\bp  \eta\|_0^2 
+     \| \alpha^{-1/2} \d \bp \nabla   \eta\|_0^2 +     \| \alpha^{-1/2} \d \operatorname{div} _\eta \bp  \eta\|_0^2  \nonumber
  \\
& \qquad
\lesssim \Ee(0) + \vartheta \sup_{s \in [0,t]}\Ee(s) +   \sup_{s \in [0,t]} \P( \Ee(s)) \,. \tag{\ref{estd0bp}'}
\end{align} 
We shall make use of the fact that we can commute gradients and tangential derivatives, modulo lower-order terms that have been estimated.

Using the same argument,  we can consider the $\bp^{\a}$-problem, for integers $a=2,..., 8$.   Using induction, assume that the $\bp^{\a-1}$-problem
 has been estimated.  Then, 
we let $\bp^{\a-1}$ act on \eqref{ss700}; it is thus only necessary to 
analyze the terms that contain the $\bp^{\a} $ derivatives,  {\it as  all other terms with at most $\bp^{\a-1}$ derivatives} can be easily  estimated by 
the $\bp^{\a-1}$-problem.
 Hence,  we let $\bp^{\a-1}$ act on \eqref{ss700} and test against $\d \bp^{\a} v^i$ to obtain
\begin{align*} 
&{\frac{1}{2}}  \frac{d}{dt} \left( \| \alpha ^2 \d^{1/2} \bp^{\a} v\|_0^2 +     \| \alpha^{-1/2} \d^{1/2}\bp^{\a}  \eta\|_0^2 
\right) 
+ {\frac{\dot\alpha}{2}} \|\alpha^{-1} \d^{1/2}\bp^{\a}  \eta\|_0^2  + \mathcal{I} ^{0,\bp^{\a}}_1 + \mathcal{I} ^{0,\bp^{\a}}_2  + \mathfrak{R}=0 \,,
\end{align*} 
where $\mathfrak{R} $ denotes a remainder consisting of lower-order terms estimated by the $\bp^{\a-1}$-problem and
\begin{align*} 
\mathcal{I} ^{0,\bp^{\a}}_1 =
 \alpha ^{-1}  \int_{\Omega_0} \left[\d^2 J^{-2} \bp^{\a} \aki \right],_k \ \bp^{\a} v^i \, dx \ \text{ and } \  \mathcal{I} ^{0,\bp^{\a}}_2 = 
 \alpha ^{-1}  \int_{\Omega_0} \left[ \d^2  \bp^{\a} J^{-2} \right],_k\, \  \bp^{\a}v^i\, \aki \, dx \,,
\end{align*} 
Using the identity \eqref{a1}, 
$$
\bp^\a \aki =\bp^{\a-1}\left( \bp\eta^j,_{s} J[A^s_j A^k_i - A^k_j A^s_i]  \right) 
= \bp^\a\eta^j,_{s} J[A^s_j A^k_i - A^k_j A^s_i]  + \mathfrak{r}\,,
$$
where $\mathfrak{r}=  \sum_{b=0}^{\a-1}  \bp^b \eta^j,_{s}  \bp^{\a-1-b} \left(J[A^s_j A^k_i - A^k_j A^s_i] \right)$.  Since the highest-order derivative in 
$\mathfrak{r}$ is given by $\bp^{\a-1}\nabla \eta$ and this term has been estimated by the $\bp^{\a-1}$-problem, we need only consider the integral
\begin{align*} 
\mathcal{I} ^{0,\bp^{\a}}_{1,\operatorname{high}} 
&=
 \alpha ^{-1}  \int_{\Omega_0} \left[\d^2 J^{-2} \ \bp^\a\eta^j,_{s} J(A^s_j A^k_i - A^k_j A^s_i)\right],_k \ \bp^{\a} v^i \, dx \\
 &=
- \alpha ^{-1}  \int_{\Omega_0} \d^2 J^{-2} \ \bp^\a\eta^j,_{s} J[A^s_j A^k_i - A^k_j A^s_i]  \ \bp^{\a} v^i,_k \, dx \,.
\end{align*} 
By  Lemma \ref{lemma_aenergy},
\begin{align*}
\mathcal{I} ^{0,\bp^{\a}}_{1,\operatorname{high}} 
&=
\frac{\alpha^{-1} }{2} \int_{\Omega_0} \d^2 J ^{-1} \frac{d}{dt} \left( | \nabla _\eta \  \bp^\a \eta|^2  - |\operatorname{div}  _\eta \  \bp^\a \eta|^2
- 2 |\operatorname{curl}  _\eta \  \bp^\a \eta|^2
\right)\, dx \,.
\end{align*} 
Then, for the integral $ \mathcal{I} ^{0,\bp^{\a}}_2 $, we expand $\bp^\a J^{-2}$ and by the same argument used for the integral $ \mathcal{I} ^{0,\bp^{\a}}_1 $,
we need only estimate the highest-order term
\begin{align*}
\mathcal{I} ^{0,\bp^1}_{2,\operatorname{high}} 
&= 2 \alpha ^{-1}  \int_{\Omega_0} \d^2 J^{-3} \bp^\a \eta^r,_s a^s_r \ \bp^\a v^i,_k  \aki \, dx
= \alpha^{-1} \int_{\Omega_0} \d^2 J^{-1} \frac{d}{dt}  | \operatorname{div} _\eta \ \bp^\a \eta|^2\, dx
 \,.
\end{align*} 

Now, we use the fact that $ \operatorname{curl} \eta =  \operatorname{curl} \deta$ and that
$$
\operatorname{curl} _\deta \, \bp^\a \eta = \operatorname{curl} \bp^\a\deta + \nabla \bp^\a \deta \ (A - \operatorname{Id} ) \,,
$$
so that with \eqref{Abound} and  Lemma \ref{lemma_curl},
$$
\| \ah \d  \operatorname{curl} _\eta \, \bp^8 \deta \|  \lesssim \Ee(0) + \vartheta \sup_{s \in [0,t]}\Ee(s) +   \sup_{s \in [0,t]} \P( \Ee(s)) \ \ \text{for} \ \a=1,2,3,4,5,6,7,8
 \,.
$$
Hence, 
the sum of the two integrals $\mathcal{I} ^{0,\bp^1}_{1,\operatorname{high}} + \mathcal{I} ^{0,\bp^1}_{2,\operatorname{high}}$,
after time integration,  provide the energy terms in the same way as for the $\bp$-problem, and we obtain that  
\begin{align} 
&\|  \d^{1/2} \bp^{\a} \deta\|_0^2 + \| \alpha ^2 \d^{1/2} \bp^{\a} v\|_0^2 +     \| \alpha^{-1/2} \d^{1/2}\bp^{\a}  \eta\|_0^2 
+     \| \alpha^{-1/2} \d \nabla \bp^{\a}  \eta\|_0^2 +     \| \alpha^{-1/2} \d \operatorname{div} _\eta \bp^{\a}  \eta\|_0^2  \nonumber
  \\
& \qquad
\lesssim \Ee(0) + \vartheta \sup_{s \in [0,t]}\Ee(s) +   \sup_{s \in [0,t]} \P( \Ee(s)) \ \ \text{for} \ \a=1,2,3,4,5,6,7,8
 \,. \label{estd0bpa}
\end{align}

 \vspace{.2in}
\noindent
{\it Step 3. First-order estimate.}   
 We next  let $ \nabla = \p_{r_1}$ act on \eqref{momentum3d} and write this as
 \begin{align} 
 \alpha^4\p_t v^i,_{r_1} + 2\alpha ^3 \dot\alpha v^i,_{r_1} + \alpha ^{-1}  \eta^i ,_{r_1}  
+ \alpha^{-1} \left[ \aki  \d ^{-1} ( \d^2 J^{-2}),_k\right],_{r_1}
 =0  \,.
\label{d1momentum3d}
\end{align} 
The analysis of this equation is similar to that of the $\bp$-problem, but we must contend with the fact that $ \nabla \d \neq 0$.  

Using the product rule, Lemma \ref{cor1},  and \eqref{grad_d}, we have that
\begin{align} 
\left[ \aki  \d ^{-1} ( \d^2 J^{-2}),_k\right],_{r_1}
&
=\aki, _{r_1}   \d ^{-1} ( \d^2 J^{-2}),_k 
+
\aki \d^{-2} \left( \d^3 J^{-2},_{r_1}\right),_k \nonumber \\
& \qquad\qquad
- a^{r_1}_i J^{-2}  
+\left[ \d,_{r_1}  J^{-2},_{ k} -  \d,_k  J^{-2},_{r_1 } \right] \aki \nonumber \\
&
=\aki, _{r_1}   \d ^{-2} ( \d^3 J^{-2}),_k
+
\aki \d^{-2} \left( \d^3 J^{-2},_{r_1}\right),_k \nonumber \\
& \qquad\qquad
- a^{r_1}_i J^{-2}   + {\frac{1}{2}} x_k \aki,_{r_1} J^{-2}
+\left[ \d,_{r_1}  J^{-2},_{ k} -  \d,_k  J^{-2},_{r_1 } \right] \aki\,. \label{ss1001}
\end{align} 
Testing \eqref{d1momentum3d} with $\d^2  v^i,_{r_1}$  yields
\begin{align} 
&{\frac{1}{2}}  \frac{d}{dt} \left( \| \alpha ^2 \d\nabla  v\|_0^2 +     \| \alpha^{-1/2} \d \nabla \eta\|_0^2 
\right) 
+  {\frac{\dot\alpha}{2}} \|\alpha^{-1} \d\nabla \eta\|_0^2  + \sum_{j=1}^5  \mathcal{I} ^1_j  
=0\,, \label{ss1000}
\end{align} 
where
\begin{alignat*}{2}
\mathcal{I}^1_1 & = -\alpha^{-1}\int_{\Omega_0}  \d^3 J^{-2}\aki, _{r_1}  v^i,_{kr_1}\, dx \,, \ \ 
&&\mathcal{I}^1_2  = -\alpha^{-1}\int_{\Omega_0}  \d^3 J^{-2}, _{r_1} v^i,_{kr_1} \aki\, dx \,,\\
\mathcal{I}^1_3 & = -\alpha^{-1}\int_{\Omega_0}  \d^2 J^{-2}     v^i,_{r_1} a^{r_1}_i \, dx\,,  \  \
&&
\mathcal{I}^1_4  =  {\frac{\alpha^{-1}}{2}} \int_{\Omega_0}  \d^2 x_k \, \aki,_{r_1} J^{-2}    v^i,_{r_1}\, dx\,, \\
\mathcal{I}^1_5 & = \alpha^{-1}\int_{\Omega_0}  \d^2 \left[ \d,_{r_1}  J^{-2},_{ k} -  \d,_k  J^{-2},_{r_1 } \right] \aki    v^i,_{r_1}\, dx \,.
\end{alignat*} 
Identity \eqref{a1} together with Lemma \ref{lemma_aenergy} shows that
\begin{align*}
 \mathcal{I}^1_1
&=
\frac{\alpha^{-1} }{2} \int_{\Omega_0} \d^3 J ^{-1} \frac{d}{dt} \left( | \nabla _\eta \  \nabla \eta|^2  - |\operatorname{div}  _\eta \  \nabla \eta|^2
- 2 |\operatorname{curl}  _\eta \  \nabla \eta|^2
\right)\, dx \,,
\end{align*} 
and \eqref{J1} shows that
\begin{align*}
\mathcal{I} ^1_2
&= 2 \alpha ^{-1}  \int_{\Omega_0} \d^3 J^{-3} \nabla \eta^r,_s a^s_r \ \nabla v^i,_k  \aki \, dx
= \alpha^{-1} \int_{\Omega_0} \d^3 J^{-1} \frac{d}{dt}  | \operatorname{div} _\eta \ \nabla \eta|^2\, dx
 \,.
\end{align*} 
Hence,
\begin{align*} 
 \mathcal{I}^1_1 + \mathcal{I} ^1_2 = \frac{\alpha^{-1} }{2} \int_{\Omega_0} \d^3 J ^{-1} \frac{d}{dt} \left( | \nabla _\eta \  \nabla \eta|^2  + |\operatorname{div}  _\eta \  \nabla \eta|^2
-  2 |\operatorname{curl}  _\eta \  \nabla \eta|^2
\right)\, dx \,.
\end{align*}

Again using \eqref{J1}, we have that
$$
\mathcal{I}^1_3  = -\alpha^{-1}\int_{\Omega_0}  \d^2 J^{-2}   \p_t J \, dx =  \alpha ^{-1}  \frac{d}{dt} \int_{\Omega_0}  \d^2 J ^{-1} \, dx \,.
$$
The estimates for the integrals $ \mathcal{I} ^1_4$ and  $ \mathcal{I} ^1_5$ rely on the tangential derivative estimates obtained in Step 2;
in particular, both of these integrals  have derivatives of highest-order, but of tangential derivative type.  Hence,
using Lemma \ref{lemma_atan},
$$
\mathcal{I}^1_4  =  {\frac{\alpha^{-3}}{2}} \int_{\Omega_0}  \d^2 \nabla  \bp \eta \, \bp \eta \, J^{-2}\alpha ^2  \nabla  v\, dx\,.
$$
Note that \eqref{estd0bp} gives us control of $\alpha^{-1/2} \d \nabla \bp  \eta( \cdot  ,t)$ in $L^2(\Omega_0)$ while $\alpha ^2 \d \nabla V$ in $L^2(\Omega_0)$
appears on the right-hand side of \eqref{ss1000}. Next, we use Lemma \ref{lemma_tan} and write $ \mathcal{I} ^1_5$ as
\begin{align*} 
\mathcal{I}^1_5  = \frac{\alpha^{-1}}{2}\int_{\Omega_0}  \d^2 \bp J^{-2} \widetilde{ a \, \nabla v}\, dx 
=- \frac{\alpha^{-1}}{2}\int_{\Omega_0}  \d^2  J^{-2} \operatorname{div} _\eta  \bp \eta \  \widetilde{ a \, \nabla v}\, dx \,,
\end{align*} 
and  we will again use the fact that \eqref{estd0bp} gives us control of $\alpha^{-1/2} \d \operatorname{div} _\eta \bp  \eta( \cdot  ,t)$ in $L^2(\Omega_0)$.
We integrate equation \eqref{ss1000} in time from $0$ to $t$, commute $\frac{d}{dt} $ with $\alpha  ^{-1} $, and find that
\begin{align*} 
& \| \alpha ^2 \d\nabla v\|_0^2 +     \| \alpha^{-1/2} \d \nabla   \eta\|_0^2 
+
\alpha^{-1}  \int_{\Omega_0} \d^3  \left( | \nabla _\eta \ \nabla  \eta|^2  + |\operatorname{div}  _\eta \  \nabla  \eta|^2
\right)\, dx +  \frac{d}{dt} \alpha ^{-1} \int_{\Omega_0}  \d^2 J ^{-1} \, dx \\
& \qquad
\lesssim \Ee(0) + \alpha^{-1}  \int_{\Omega_0} \d^2  | \operatorname{curl}  _\eta \ \nabla  \eta|^2  \, dx 
+ \int_0^t \alpha^{-2}\dot\alpha  \int_{\Omega_0} \d^2  | \operatorname{curl}  _\eta \  \nabla  \eta|^2  \, dx ds \\
& \qquad \qquad
+ \int_0^t \alpha^{-3} \int_{\Omega_0} J ^{-1}| \P(A)| \,  | \nabla^2 \eta |^2  \, \alpha ^2 |\nabla v|\, dx ds
+ \int_0^t  {\frac{\alpha^{-3}}{2}} \int_{\Omega_0}  \d^2 | \nabla  \bp \eta| \, |\bp \eta| \, J^{-2}\,  | \alpha ^2  \nabla  v|\, dxds \\
& \qquad \qquad
- \int_0^t \alpha^{-3}\int_{\Omega_0}  \d^2  J^{-1}| \operatorname{div} _\eta  \bp \eta| \ |A| \,| \nabla v| \, dx ds
 \,,
\end{align*} 
which, thanks to Lemma \ref{lemma_curl} and \eqref{estd0bp}, shows that
\begin{align*} 
& \| \alpha ^2 \d\nabla v\|_0^2 +     \| \alpha^{-1/2} \d \nabla   \eta\|_0^2 
+
\| \alpha ^{-1/2} \d^{3/2}   \nabla _\eta \ \nabla  \eta\|_0^2  + \| \alpha ^{-1/2} \d^{3/2}  \operatorname{div}  _\eta \ \nabla  \eta\|_0^2  
\\
& \qquad \lesssim \Ee(0) + \vartheta \sup_{s \in [0,t]}\Ee(s) +   \sup_{s \in [0,t]} \P( \Ee(s))
 \,.
\end{align*} 
Then,  since $\left|\nabla \deta (x,t)\right| 
\lesssim  \sup_{s \in [0,t]} \alpha ^2(s) |\nabla v(x,s)|$, and using \eqref{Abound}, we obtain that
\begin{align} 
& \| \alpha ^2 \d\nabla \deta \|_0^2  + \| \alpha ^2 \d\nabla v\|_0^2 +     \| \alpha^{-1/2} \d \nabla   \eta\|_0^2 
+
\| \alpha ^{-1/2} \d^{3/2}   \nabla ^2 \eta\|_0^2  + \| \alpha ^{-1/2} \d^{3/2}  \operatorname{div}  _\eta \ \nabla  \eta\|_0^2   \nonumber
\\
& \qquad \lesssim \Ee(0) + \vartheta \sup_{s \in [0,t]}\Ee(s) +   \sup_{s \in [0,t]} \P ( \Ee(s))
 \,.
 \label{estd1}
\end{align} 

We next study the tangential derivative $\bp^\a$ problem for $ a=1,..., 7$.   We assume that the $\bp^{\a-1}$ problem has been estimated, and
let $\bp^{\a}$ act on \eqref{d1momentum3d} and use \eqref{ss1001} to find that
  \begin{align*} 
  &
 \alpha^4\p_t \bp^\a v^i,_{r_1} + 2\alpha ^3 \dot\alpha \bp^{\a} v^i,_{r_1} + \alpha ^{-1} \bp^{\a}  \eta^i ,_{r_1}  
+ \alpha^{-1}\bp^{\a} \aki, _{r_1}   \d ^{-2} ( \d^3 J^{-2}),_k \\
& \qquad \qquad  
 + \aki \d^{-2} \left( \d^3 \bp^{\a} J^{-2},_{r_1}\right),_k  - \bp^{\a} a^{r_1}_i J^{-2}  - a^{r_1}_i \bp^{\a} J^{-2}  \\
& \qquad\qquad
 + {\frac{1}{2}} x_k \bp^\a \aki,_{r_1} J^{-2} + {\frac{1}{2}} x_k \aki,_{r_1} \bp^\a J^{-2} 
+\left[ \d,_{r_1}  \bp^\a J^{-2},_{ k} -  \d,_k  \bp^\a  J^{-2},_{r_1 } \right] \aki + \mathfrak{R}
 =0  \,,
\end{align*} 
where $\mathfrak{R}$ denotes a remainder consisting of lower-order terms with at most $\nabla \bp^{\a-1}$ derivatives.
 We test this equation against $\bp^a v^i,_{r_1}$ and obtain that
\begin{align} 
&{\frac{1}{2}}  \frac{d}{dt} \left( \| \alpha ^2 \d\bp^\a \nabla  v\|_0^2 +     \| \alpha^{-1/2} \d \bp^\a \nabla \eta\|_0^2 
\right) 
+  {\frac{\dot\alpha}{2}} \|\alpha^{-1} \d \bp^\a  \nabla \eta\|_0^2  + \sum_{j=1}^5  \mathcal{I}^{1,\bp^\a}_j  + \mathfrak{R} 
=0\,, \nonumber
\end{align} 
where again $\mathfrak{R} $ denotes a remainder consisting of lower-order terms and 
\begin{align*}
\mathcal{I}^{1,\bp^\a}_1 & =- \alpha^{-1}\int_{\Omega_0}  \d^3 J^{-2} \bp^\a \aki, _{r_1} \  \bp^\a v^i,_{kr_1}\, dx \,,\\
\mathcal{I}^{1,\bp^\a}_2 & =- \alpha^{-1}\int_{\Omega_0}  \d^3 \bp^\a  (J^{-2}), _{r_1}   \ \bp^\a v^i,_{kr_1} \aki\, dx \,,\\
\mathcal{I}^{1,\bp^\a}_3 & = -\alpha^{-1}\int_{\Omega_0}  \d^2 J^{-2}    \bp^\a  v^i,_{r_1}\ \bp^\a a^{r_1}_i \, dx\,, \\
\mathcal{I}^{1,\bp^\a}_4 & =  {\frac{\alpha^{-1}}{2}} \int_{\Omega_0}  \d^2  x_k \, \bp^\a \aki,_{r_1} J^{-2}    \bp^\a v^i,_{r_1}\, dx\,, \\
\mathcal{I}^{1,\bp^\a}_5 & = \alpha^{-1}\int_{\Omega_0}  \d^2 \left[ \d,_{r_1} \bp^\a  J^{-2},_{ k} -  \d,_k  \bp^\a J^{-2},_{r_1 } \right] \aki  \  \bp^\a v^i,_{r_1}\, dx \,.
\end{align*} 
The integrals $ \mathcal{I}^{1,\bp^\a}_1 + \mathcal{I}^{1,\bp^\a}_2 $ are estimated in the identical fashion  as 
$ \mathcal{I}^{0,\bp^\a}_1+  \mathcal{I}^{0,\bp^\a}_2$, while
$\mathcal{I}^{1,\bp^\a}_3 $ is estimated just as $\mathcal{I} ^{0,\bp^\a}_1$.  Now the integral $\mathcal{I}^{1,\bp^\a}_4 $ is estimated using Lemma \ref{lemma_atan}
which shows that the highest-order term in that integrand can be written as
$$
 \alpha^{-2.5} \int_{\Omega_0}  \alpha ^{-1/2} \d^2 \nabla  \bp^{\a+1} \eta \, \bp \eta \, J^{-2} \alpha ^2  \bp^\a \nabla  v \,dx
$$
which is indeed estimated using the fact that for $\a=0,...,6$,  $ \alpha ^{-1/2} \d \nabla  \bp^{\a+1} \eta$ is in $L^2(\Omega_0)$ by  \eqref{estd0bpa}.   
Similarly, $\mathcal{I}^{1,\bp^\a}_5$
is estimated using Lemma \ref{lemma_tan}, and with Lemma \ref{lemma_curl},  we find that 
\begin{align} 
& \|  \d\nabla\bp^\a \deta \|_0^2  + \| \alpha ^2 \d\nabla \bp^\a v\|_0^2 +     \| \alpha^{-1/2} \d \nabla  \bp^\a \eta\|_0^2 
+
\| \alpha ^{-1/2} \d^{3/2}   \nabla ^2 \bp^\a \eta\|_0^2  + \| \alpha ^{-1/2} \d^{3/2}  \operatorname{div}  _\eta \ \nabla \bp^\a \eta\|_0^2   \nonumber
\\
& \qquad \lesssim \Ee(0) + \vartheta \sup_{s \in [0,t]}\Ee(s) +   \sup_{s \in [0,t]} \P( \Ee(s)) \ \ \text{for} \ \a=0,1,2,3,4,5,6,7
 \,.
 \label{estd1bpa}
\end{align}

\vspace{.2in}
\noindent
{\it Step 4. Second-order  estimate.} For the second-order energy estimate, we let $\nabla ^2:=\p_{r_1r_2}$ act on 
(\ref{ceuler03D}b) and obtain the equation
\begin{align} 
\alpha^4\p_t v^i,_{r_1r_2} + 2\alpha ^4 \dot\alpha v^i ,_{r_1r_2}+\alpha ^{-1}  \eta^i ,_{r_1r_2}+\alpha ^{-1}  \left[a^k_i \d ^{-1} \left(\d^2 J^{-2} \right),_k\right],_{r_1r_2} =0  \,.
\label{d2momentum3d}
\end{align} 
Now, using the product rule and Lemma \ref{cor1},
\begin{align*} 
& \left[ \aki\d ^{-1} \left( \d^2 J^{-2}\right),_k\right],_{r_1 r_2}  \\
&\ \ \ = \aki \ \d ^{-3} \left[ \d^4 (J^{-2}),_{r_1 r_2} \right],_k  + \aki,_{r_1 r_2} \d^{-3} \left[ \d^4 J^{-2}\right],_k  \\
&\qquad  -2 \d,_k \aki,_{r_1 r_2}  J^{-2}  
- {\frac{3}{2}} a^{r_2}_i (J^{-2}),_{r_1} -   a^{r_1}_i (J^{-2}),_{r_2}  \\
&\qquad -\left[ \d,_k (J^{-2}),_{r_1r_2}  - \d,_{r_2} (J^{-2}),_{r_1k} \right] \aki-\left[ \d,_k (J^{-2}),_{r_1}  - \d,_{r_1} (J^{-2}),_{k} \right],_{r_2} \aki \\
&\qquad + \aki,_{ r_2} \d^{-2} \left[ \d^3 (J^{-2}),_{r_1} \right],_k  + a^{r_1}_i,_{r_2} J^{-2}  + \aki,_{r_2} \left[ \d,_k (J^{-2}),_{r_1}
- \d,_{r_1} (J^{-2}),_k\right] \\
&\qquad + \aki,_{ r_1} \d^{-2} \left[ \d^3 (J^{-2}),_{r_2} \right],_k  + a^{r_2}_i,_{r_1} J^{-2}  + \aki,_{r_1} \left[ \d,_k (J^{-2}),_{r_2}
- \d,_{r_2} (J^{-2}),_k\right] \,.
\end{align*} 
Thus, testing \eqref{d2momentum3d} against $\d^3 v^i,_{r_1r_2}$, we have that
\begin{align} 
&{\frac{1}{2}}  \frac{d}{dt} \left( \| \alpha ^2 \d^{3/2}\nabla^2  v\|_0^2 +     \| \alpha^{-1/2} \d^{3/2} \nabla^2 \deta\|_0^2 
\right) 
+ {\frac{\dot\alpha}{2}} \|\alpha^{-1} \d^{3/2} \nabla^2 \deta\|_0^2  + \sum_{j=1}^7 \mathcal{I}^2 _j =0 \,, \label{ss80}
\end{align} 
where
\begin{align*} 
\mathcal{I}^2 _1 & =  \alpha^{-1} \int_{\Omega_0} \left[ \d^4 (J^{-2}),_{r_1 r_2}\right],_k v^i,_{r_1 r_2 } \aki\, dx \,, \\
\mathcal{I}^2 _2 & =  - \alpha^{-1} \int_{\Omega_0} \left[ \d^4 J^{-2} \aki,_{r_1 r_2}\right],_k  v^i,_{r_1 r_2} \, dx \,, \\
\mathcal{I}^2 _3 & = \alpha^{-1} \int_{\Omega_0} \d^3\, x_k \    \aki,_{r_1 r_2}  J^{-2}   v^i,_{r_1 r_2 } \, dx \,,  \\
\mathcal{I} ^2_4 & = - \alpha^{-1}\int_{\Omega_0} \d^3 \left\{  
\left[ \d,_k (J^{-2}),_{r_1r_2}  - \d,_{r_2} (J^{-2}),_{r_1k} \right] +   \left[ \d,_k (J^{-2}),_{r_1}  - \d,_{r_1} (J^{-2}),_{k} \right],_{r_2}
\right\}  \aki  v^i,_{r_1 r_2 } \, dx \,, \\
\mathcal{I} ^2_5 & = \alpha^{-1}\int_{\Omega_0} \d^3 \left[ - {\frac{3}{2}} a^{r_2}_i (J^{-2}),_{r_1} -  a^{r_1}_i (J^{-2}),_{r_2}
+a^{r_1}_i,_{r_2} J^{-2} + a^{r_2}_i,_{r_1} J^{-2}\right] v^i,_{r_1 r_2 } \, dx \,, \\
\mathcal{I} ^2_6 & =  \alpha^{-1}\int_{\Omega_0} \d^3  
\left\{ \left[ \d,_k (J^{-2}),_{r_1} - \d,_{r_1} (J^{-2}),_k\right]  \aki,_{r_2}    + \left[ \d,_k (J^{-2}),_{r_2} - \d,_{r_2} (J^{-2}),_k\right]  \aki,_{r_1} \right\} 
 v^i,_{r_1 r_2}   dx  \,, \\
 \mathcal{I} ^2_7 & = \alpha^{-1}\int_{\Omega_0} \d^3  
\left\{
\aki,_{ r_2} \d^{-2} \left[ \d^3 (J^{-2}),_{r_1} \right],_k + \aki,_{ r_1} \d^{-2} \left[ \d^3 (J^{-2}),_{r_2} \right],_k
 \right\} 
 v^i,_{r_1 r_2}   dx \,.
\end{align*} 
For $j=1,2,3,4$, the 
integrals $ \mathcal{I}^2 _j$ have the highest derivative count, having three derivatives on $\eta$, while the
integrals $ \mathcal{I} _j$, $j=5,6,7$ are lower-order and are estimated using \eqref{estd1bpa}.

  We begin
with $ \mathcal{I} ^2_1$ and $ \mathcal{I}^2 _2$:
\begin{align*} 
\mathcal{I} ^2_1 & = 2 \alpha^{-1} \int_{\Omega_0} \d^4  J^{-1 } \eta^l,_{j r_1 r_2} A^j_l \,  v^i,_{r_1 r_2 k} A^k_i \, dx
-2 \alpha^{-1} \int_{\Omega_0} \left[ \d^4  \left( J^{-2 } A^j_l \right),_{r_2}  \eta^l,_{j  r_1}\right],_k \,  v^i,_{r_1 r_2 } A^k_i \, dx \\
& =  \alpha^{-1} \frac{d}{dt} \int_{\Omega_0} \d^4  J^{-1 } | \operatorname{div} _\eta \eta,_{r_1 r_2}  |^2 \, dx
-
 \alpha^{-1} \int_{\Omega_0} \d^4  \p_t\left( J^{-1 } A^j_l  A^k_i \right)  \eta^l,_{j r_1 r_2}  \,  \eta^i,_{r_1 r_2 k}  \, dx \\
 & \qquad -2 \alpha^{-1} \int_{\Omega_0} \left[ \d^4  \left( J^{-2 } A^j_l \right),_{r_2}  \eta^l,_{j  r_1}\right],_k \,  v^i,_{r_1 r_2 } A^k_i \, dx \,.
\end{align*} 

We write
\begin{align*} 
\d^4 \p_t\left( J^{-1 } A^j_l  A^k_i \right)  \eta^l,_{j r_1 r_2}  \,  \eta^i,_{r_1 r_2 k}  
& \sim \d^4 J^{-1 } \P( A)\, \nabla v\,  \nabla^3 \eta \, \nabla^3 \eta  \,, \\
-2 \left[ \d^4  \left( J^{-2 } A^j_l \right),_{r_2}  \eta^l,_{j  r_1}\right],_k \,  v^i,_{r_1 r_2 } 
&\sim  \d^4 \P(J^{-1} A) \left[ \nabla ^3 \eta \ \nabla ^2\eta + \nabla ^2\eta \ \nabla ^2\eta \ \nabla ^2\eta\right] \nabla ^2 v\,,
\end{align*} 
and then integrate $ \mathcal{I} ^2_1$ in time from $0$ to $t$ to find that\footnote{We note that in the term $\operatorname{div} _\eta \nabla ^2 \eta $, the operator $ \operatorname{div} _\eta$ acts on each component of the Hessian $
\nabla ^2 \eta $.  In particular $ |\operatorname{div} _\eta \ \nabla ^2 \eta|^2 = \operatorname{div} _\eta \eta^i,_{r_1 r_2}  
\,  \operatorname{div} _\eta\eta^i,_{r_1 r_2}$.  }
\begin{align*} 
\int_0^t \mathcal{I} ^2_1(s) ds
& =  \alpha^{-1}  \int_{\Omega_0} \d^4  J^{-1 } | \operatorname{div} _\eta \nabla ^2 \eta  |^2 \, dx 
+
\int_0^t \alpha^{-2} \dot\alpha \int_{\Omega_0} \d^4  J^{-1 } | \operatorname{div} _\eta \nabla ^2 \eta  |^2 \, dxds \\
&
+ \int_0^t  \alpha^{-2} \int_{\Omega_0} \d^4  J^{-1 } \P( A)\,  \alpha ^2 \nabla v\,   \alpha ^{-1/2} \nabla^3 \eta \,  \alpha ^{-1/2}  \nabla^3 \eta  \, dxds \\
&
+ \int_0^t  \alpha^{-2.5} \int_{\Omega_0} \d^4 \P(J^{-1} A) \left[  \alpha ^{-1/2} \nabla ^3 \eta \ \nabla ^2\eta +   \alpha ^{-1/2} \nabla ^2\eta \ \nabla ^2\eta \ \nabla ^2\eta\right]\, \alpha ^2 \nabla ^2 v \, dxds - \Ee(0) \,.
\end{align*} 
Similarly, 
\begin{align*} 
\mathcal{I}^2 _2 & =  - \alpha^{-1} \int_{\Omega_0} \d^4 J^{-1} \eta^j,_{sr_1 r_2} (A^s_{r_1} A^k_i - A^s_i A^k_{r_1}) v^i,_{r_1 r_2 k} \, dx  \\
&\qquad\qquad
 + \alpha^{-1} \int_{\Omega_0}  \d^4 \P(J^{-1} A) \left[ \nabla ^3 \eta \ \nabla ^2\eta + \nabla ^2\eta \ \nabla ^2\eta \ \nabla ^2\eta\right] \nabla ^2 v \, dx \,.
\end{align*} 
By Lemma \ref{lemma_aenergy},
$$
 -  \eta^j,_{sr_1 r_2} (A^s_{r_1} A^k_i - A^s_i A^k_{r_1}) v^i,_{r_1 r_2 k} ={\frac{1}{2}}   \frac{d}{dt} \left( |  \nabla_\eta \  \eta,_{r_1r_2} |^2 
- | \operatorname{div}_\eta \eta ,_{r_1r_2}  |^2 - 2| \operatorname{curl}_\eta \eta,_{r_1r_2}  |^2\right)  \,,
$$
and hence
\begin{align*} 
\mathcal{I} ^2_2 & =   \frac{\alpha^{-1}}{2} \frac{d}{dt}  \int_{\Omega_0} \d^4 J^{-1}
\left(  |  \nabla_\eta \nabla \eta |^2 
- | \operatorname{div}_\eta  \nabla^2 \eta |^2     
-2 | \operatorname{curl}_\eta \nabla^2 \eta |^2 \right) \, dx  \\
 & \qquad \qquad
 +  \alpha^{-1} \int_{\Omega_0}  \d^4 \P(J^{-1} A) \left[\nabla ^3 \eta \, \nabla ^3\eta\, \nabla v+ \left(\nabla ^3 \eta \, \nabla ^2\eta + \nabla ^2\,\eta \, \nabla ^2\eta \, \nabla ^2\eta\right)\nabla ^2 v  \right] \, dx \\
&=  \frac{d}{dt}  \frac{\alpha^{-1}}{2}  \int_{\Omega_0} \d^4 J^{-1}
\left(  |  \nabla_\eta \nabla^2  \eta |^2 - | \operatorname{div}_\eta \nabla ^2 \eta   |^2 - 2| \operatorname{curl}_\eta \nabla ^2\eta |^2
\right) \, dx  \\
&\qquad  \qquad
+ \frac{\alpha^{-2}\dot\alpha}{2}  \int_{\Omega_0} \d^4 J^{-1}
\left(  |  \nabla_\eta  \nabla ^2 \eta |^2 - | \operatorname{div}_\eta  \nabla^2\eta    |^2 - | \operatorname{curl}_\eta  \nabla^2 \eta  |^2
\right) \, dx  \\
& \qquad \qquad
 +  \alpha^{-1} \int_{\Omega_0}  \d^4 \P(J^{-1} A) \left[\nabla ^3 \eta \, \nabla ^3\eta\, \nabla v+ \left(\nabla ^3 \eta \, \nabla ^2\eta + \nabla ^2\,\eta \, \nabla ^2\eta \, \nabla ^2\eta\right)\nabla ^2 v  \right] \, dx  \,.
\end{align*} 
Integrating in time from $0$ to $t$, we find that
\begin{align} 
&\int_0^t \left( \mathcal{I}^2 _1 (s)+ \mathcal{I} ^2_2(s)\right) ds 
 =  \frac{\alpha^{-1}}{2}  \int_{\Omega_0} \d^4 J^{-1}
\left(  |  \nabla_\eta \nabla^2  \eta |^2  +  | \operatorname{div}_\eta \nabla ^2 \eta  |^2 -  2| \operatorname{curl}_\eta \nabla^2\eta   |^2 \right) \, dx\nonumber  \\
&\
+ \int_0^t  \frac{\alpha^{-2}\dot\alpha}{2}  \int_{\Omega_0} \d^4 J^{-1}
\left(  |  \nabla_\eta \nabla^2  \eta |^2 +  | \operatorname{div}_\eta\nabla^2 \eta   |^2  - 2 | \operatorname{curl}_\eta \nabla^2 \eta   |^2 
\right) \, dxds  \nonumber \\
&\
+ \int_0^t  \alpha^{-2}  \int_{\Omega_0}  \d^4 \P(J^{-1} A) \left[ \alpha ^{-1/2}\nabla ^3 \eta \, \alpha ^{-1/2}\nabla ^3\eta\, \nabla v+ \left(\alpha ^{-1/2}\nabla ^3 \eta \, \nabla ^2\eta + \nabla ^2\,\eta \, \nabla ^2\eta \, \nabla ^2\eta\right)\alpha ^2 \nabla ^2 v  \right] \, dx ds \nonumber\\
&\ - \Ee(0) \,. \nonumber
\end{align} 

Next, we estimate the integral $ \mathcal{I} ^2_3$. 
We use Lemma \ref{lemma_atan} to write
$x_k a^k_{( \cdot )},_{r_1 r_2}  \sim \bp \eta,_{r_1 r_2} \, \bp \eta + \bp \eta,_{r_1 } \, \bp \eta,_{r_2} $.
It then follows that the highest-order term in $\int_0^t \mathcal{I}^2 _3(s) ds $ can be written as
\begin{equation}\label{ss81}
\int_0^t \alpha^{-2.5} \int_{\Omega_0} \d^3 \  \alpha ^{-1/2} \nabla ^2 \bp \eta \  \bp \eta \, J^{-2}   \alpha ^2 \nabla ^2 v \, dxds \,.
\end{equation} 
This  integral is estimated using the fact that $\alpha ^{-1/2} \d^{3/2}\nabla ^2 \bp$ is in $L^2(\Omega_0)$ by  \eqref{estd1bpa}, the fact that 
$\alpha ^2 \nabla ^2 v$ in $L^2(\Omega_0)$ appears in \eqref{ss80}, and again that $ \alpha ^{-2.5}$ is integrable in time.
We use Lemma \ref{lemma_tan} to estimate $ \mathcal{I} ^2_4$; the highest-order term in $\int_0^t \mathcal{I}^2 _4(s) ds $ can also be written as in
\eqref{ss81} and hence estimated in the same way.  As noted above, the integrals $ \mathcal{I} _j$, $j=5,6,7$ are lower-order and are estimated using \eqref{estd1bpa}.  Hence, after time integration and thanks to Lemma \ref{lemma_curl}, we obtain that
\begin{align} 
& \| \alpha ^2 \d ^{3/2} \nabla^2 \deta \|_0^2  + \| \alpha ^2 \d ^{3/2} \nabla^2 v\|_0^2 +     \| \alpha^{-1/2} \d^{3/2} \nabla^2 \eta\|_0^2 
+
\| \alpha ^{-1/2} \d^{2}   \nabla ^3\eta\|_0^2  + \| \alpha ^{-1/2} \d^{2}  \operatorname{div}  _\eta \ \nabla^2 \eta\|_0^2   \nonumber
\\
& \qquad \lesssim \Ee(0) + \vartheta \sup_{s \in [0,t]}\Ee(s) +   \sup_{s \in [0,t]} \P( \Ee(s))
 \,.
 \label{estd2}
\end{align} 
We next study the tangential derivative $\bp^\a$ problem for $ a=1,..., 6$.   We assume that the $\bp^{\a-1}$ problem has been estimated, and
let $\bp^{\a}\p_{r_1r_2}$ act on \eqref{momentum3d} and test against $\bp^{\a}\p_{r_1r_2} v^i$.    In the same way that we obtained \eqref{estd1bpa}, we find
that 
\begin{align} 
 & \| \alpha ^2 \d ^{3/2} \nabla^2\bp^\a \deta \|_0^2  + \| \alpha ^2 \d ^{3/2} \nabla^2\bp^\a v\|_0^2 +     \| \alpha^{-1/2} \d^{3/2}\nabla^2\bp^\a \eta\|_0^2 
+
\| \alpha ^{-1/2} \d^{2}   \nabla ^3\bp^\a\eta\|_0^2 
  \nonumber
\\
& \qquad \lesssim \Ee(0) + \vartheta \sup_{s \in [0,t]}\Ee(s) +   \sup_{s \in [0,t]} \P( \Ee(s)) \ \ \text{for} \ \a=0,1,2,3,4,5,6
 \,.
 \label{estd2bpa}
\end{align} 

\vspace{.2in}
\noindent
{\it Step 5. Third-order through six-order  estimates.}  Repeating the procedure detailed above, at the third-order level, we find that
\begin{align} 
 & \|  \d ^{2} \nabla^3\bp^\a \deta \|_0^2  + \| \alpha ^2 \d ^{2} \nabla^3\bp^\a v\|_0^2 +     \| \alpha^{-1/2} \d^{2}\nabla^3\bp^\a \eta\|_0^2 
+
\| \alpha ^{-1/2} \d^{5/2}   \nabla ^4\bp^\a\eta\|_0^2 
  \nonumber
\\
& \qquad \lesssim \Ee(0) + \vartheta \sup_{s \in [0,t]}\Ee(s) +   \sup_{s \in [0,t]} \P( \Ee(s)) \ \ \text{for} \ \a=0,1,2,3,4,5
 \,.
 \label{estd3bpa}
\end{align} 
At the fourth-order level, we obtain that
\begin{align} 
 & \|  \d ^{5/2} \nabla^4\bp^\a \deta \|_0^2  + \| \alpha ^2 \d ^{5/2} \nabla^4\bp^\a v\|_0^2 +     \| \alpha^{-1/2} \d^{5/2}\nabla^4\bp^\a \eta\|_0^2 
+
\| \alpha ^{-1/2} \d^{3}   \nabla ^5\bp^\a\eta\|_0^2  
  \nonumber
\\
& \qquad \lesssim \Ee(0) + \vartheta \sup_{s \in [0,t]}\Ee(s) +   \sup_{s \in [0,t]} \P( \Ee(s)) \ \ \text{for} \ \a=0,1,2,3,4
 \,.
 \label{estd4bpa}
\end{align} 
At the fifth-order level, we find that
\begin{align} 
 & \|  \d ^{3} \nabla^5\bp^\a \deta \|_0^2  + \| \alpha ^2 \d ^{3} \nabla^5\bp^\a v\|_0^2 +     \| \alpha^{-1/2} \d^{3}\nabla^5\bp^\a \eta\|_0^2 
+
\| \alpha ^{-1/2} \d^{7/2}   \nabla ^6\bp^\a\eta\|_0^2 
  \nonumber
\\
& \qquad \lesssim \Ee(0) + \vartheta \sup_{s \in [0,t]}\Ee(s) +   \sup_{s \in [0,t]} \P( \Ee(s)) \ \ \text{for} \ \a=0,1,2,3
 \,,
 \label{estd5bpa}
\end{align} 
and at the sixth-order level,
\begin{align} 
 & \|   \d ^{7/2} \nabla^6\bp^\a \deta \|_0^2  + \| \alpha ^2 \d ^{7/2} \nabla^6\bp^\a v\|_0^2 +     \| \alpha^{-1/2} \d^{7/2}\nabla^6\bp^\a \eta\|_0^2 
+
\| \alpha ^{-1/2} \d^{4}   \nabla ^7\bp^\a \eta\|_0^2 
  \nonumber
\\
& \qquad \lesssim \Ee(0) + \vartheta \sup_{s \in [0,t]}\Ee(s) +   \sup_{s \in [0,t]} \P( \Ee(s)) \ \ \text{for} \ \a=0,1,2
 \,,
 \label{estd6bpa}
\end{align} 
while at the seventh-order level,
\begin{align} 
 & \|   \d ^{4} \nabla^7\bp^\a \deta \|_0^2  + \| \alpha ^2 \d ^{4} \nabla^7\bp^\a v\|_0^2 +     \| \alpha^{-1/2} \d^{4}\nabla^7\bp^\a \eta\|_0^2 
+
\| \alpha ^{-1/2} \d^{4.5}   \nabla ^8\bp^\a \eta\|_0^2 
  \nonumber
\\
& \qquad \lesssim \Ee(0) + \vartheta \sup_{s \in [0,t]}\Ee(s) +   \sup_{s \in [0,t]} \P( \Ee(s)) \ \ \text{for} \ \a=0,1
 \,.
 \label{estd7bpa}
\end{align} 

\vspace{.2in}
\noindent
{\it Step 6. Highest order eighth-derivative estimates.}  For the eighth-order energy estimate, we let $\nabla ^8:=\p_{r_1\ddd r_8}$ act on 
(\ref{ceuler03D}b) and obtain the equation
\begin{align} 
\alpha^4\p_t v^i,_{r_1\ddd r_8} + 2\alpha ^4 \dot\alpha v^i ,_{r_1\ddd r_8}+\alpha ^{-1}  \eta^i ,_{r_1\ddd r_8}
+\alpha ^{-1}  \left[a^k_i \d ^{-1} \left(\d^2 J^{-2} \right),_k\right],_{r_1\ddd r_8} =0  \,.
\label{d7momentum3d}
\end{align} 
Now, using the product rule and Lemma \ref{cor1},
\begin{align} 
& \left[ \aki\d ^{-1} \left( \d^2 J^{-2}\right),_k\right],_{r_1 \ddd r_8}  
= \aki,_{r_1 \ddd r_8}\left[ \d ^{-1} \left( \d^2 J^{-2}\right),_k\right] 
+\aki \left[ \d ^{-1} \left( \d^2 J^{-2}\right),_k\right],_{r_1 \ddd r_8} + \operatorname{l.o.t.}   \nonumber \\
&
= \aki,_{r_1 \ddd r_8} \d ^{-9} \left( \d^{10} J^{-2}\right),_k 
+
\aki \d^{-9} \left( \d^{10} J^{-2},_{r_1 \ddd r_8}\right),_k + 4  x_k \aki,_{r_1 \ddd r_8} J^{-2} \nonumber \\
&\qquad
- \left[\d,_k   J^{-2},_{r_1 \ddd r_8} - \d,_{r_8}   J^{-2},_{r_1 \ddd r_7 k}\right] \aki 
-\p_{r_8} \left[\d,_k   J^{-2},_{r_1 \ddd r_7} - \d,_{r_7}   J^{-2},_{r_1 \ddd r_6 k}\right] \aki 
\nonumber \\
&\qquad
- \p_{r_7r_8} \left[ \d,_k   J^{-2},_{r_1 \ddd r_6} - \d,_{r_6}   J^{-2},_{r_1 \ddd r_5 k}  \right]\aki   
- \p_{r_6 r_7r_8} \left[ \d,_k   J^{-2},_{r_1\ddd r_5} - \d,_{r_5}   J^{-2},_{r_1\ddd r_4k}  \right]\aki  
\nonumber \\
&\qquad
- \p_{r_5 \ddd r_8} \left[ \d,_k   J^{-2},_{r_1\ddd r_4} - \d,_{r_4}   J^{-2},_{r_1 r_2 r_3k}  \right] \aki  
-\p_{r_4 \ddd r_8} \left[ \d,_k   J^{-2},_{r_1 r_2 r_3} - \d,_{r_3}   J^{-2},_{r_1 r_2 k}  \right]\aki 
\nonumber \\
&\qquad
- \p_{r_3 \ddd  r_8} \left[ \d,_k   J^{-2},_{r_1 r_2 } - \d,_{r_2}   J^{-2},_{r_1 k}  \right]\aki  
-\p_{r_2 \ddd  r_8} \left[ \d,_k   J^{-2},_{r_1 } - \d,_{r_1}   J^{-2},_{ k}  \right] \aki  + \operatorname{l.o.t.} \,,  \label{ss91}
\end{align} 
where $ \operatorname{l.o.t.}$ denotes lower-order terms which contain at most seven partial derivatives of $\eta$, and can thus be estimated using
\eqref{estd6bpa}.

Testing \eqref{d7momentum3d} against $\d^9 v^i,_{r_1 \ddd r_8}$, we have that
\begin{align} 
&{\frac{1}{2}}  \frac{d}{dt} \left( \| \alpha ^2 \d^{4.5}\nabla^8  v\|_0^2 +     \| \alpha^{-1/2} \d^{4.5} \nabla^8 \deta\|_0^2 
\right) 
+ {\frac{\dot\alpha}{2}} \|\alpha^{-1} \d^{4.5} \nabla^8 \deta\|_0^2  + \sum_{j=1}^4 \mathcal{I}^8 _j  + \mathfrak{R}=0 \,, \nonumber
\end{align} 
where $\mathfrak{R}$ denotes integrals containing the lower-order terms $ \operatorname{l.o.t.}$ described above, and 
\begin{align*} 
\mathcal{I}^8 _1 & =  \alpha^{-1} \int_{\Omega_0} \left[\d^{10} (J^{-2}),_{r_1 \ddd r_8}\right],_k v^i,_{ r_1 \ddd r_8  } \aki\, dx \,, \\
\mathcal{I}^8 _2 & =   \alpha^{-1} \int_{\Omega_0} \left[ \d^{10} J^{-2} \aki,_{r_1 \ddd r_8} \right],_k v^i,_{ r_1 \ddd r_8 } \, dx \,, \\
\mathcal{I}^8 _3 & = 4\alpha^{-1}\int_{\Omega_0} \d^9\, x_k \    \aki,_{r_1 \ddd r_8}  J^{-2}   v^i,_{r_1 \ddd r_8 } \, dx \,,  \\
\mathcal{I} ^8_4 & = - \alpha^{-1}\int_{\Omega_0} \d^9 \left\{  
- \left[\d,_k   J^{-2},_{r_1 \ddd r_8} - \d,_{r_8}   J^{-2},_{r_1 \ddd r_7 k}\right] \aki  
- \left[\d,_k   J^{-2},_{r_1 \ddd r_7} - \d,_{r_7}   J^{-2},_{r_1 \ddd r_6 k}\right],_{r_8} \aki  \right. \nonumber \\
&\qquad
- \left[ \d,_k   J^{-2},_{r_1 \ddd r_6} - \d,_{r_6}   J^{-2},_{r_1 \ddd r_5 k}  \right],_{r_7r_8}\aki 
-  \left[ \d,_k   J^{-2},_{r_1\ddd r_5} - \d,_{r_5}   J^{-2},_{r_1\ddd r_4k}  \right],_{r_6r_7r_8}\aki  \nonumber \\
&\qquad
- \left[ \d,_k   J^{-2},_{r_1 \ddd r_4} - \d,_{r_4}   J^{-2},_{r_1 r_2 r_3k}  \right],_{r_5 \ddd r_8} \aki 
-  \left[ \d,_k   J^{-2},_{r_1 r_2 r_3} - \d,_{r_3}   J^{-2},_{r_1 r_2 k}  \right],_{r_4\ddd r_8}\aki   \nonumber \\
&\qquad   \left.
- \left[ \d,_k   J^{-2},_{r_1 r_2 } - \d,_{r_2}   J^{-2},_{r_1 k}  \right],_{r_3\ddd r_8} \aki  
-  \left[ \d,_k   J^{-2},_{r_1 } - \d,_{r_1}   J^{-2},_{ k}  \right],_{r_2\ddd r_8} \aki
\right\}  \aki  v^i,_{r_1 \ddd r_8} \, dx \,.
\end{align*} 

Again, with \eqref{a1},
$$
\nabla ^8 \aki =\nabla^7 \left( \nabla \eta^j,_{s} J[A^s_j A^k_i - A^k_j A^s_i]  \right) 
= \nabla^8 \eta^j,_{s} J[A^s_j A^k_i - A^k_j A^s_i]  + \mathfrak{r}\,,
$$
where $\mathfrak{r}=  \sum_{b=0}^7 c_b \nabla ^b \eta^j,_{s}  \nabla ^{7-b} \left(J[A^s_j A^k_i - A^k_j A^s_i] \right)$, with  each $c_b$,  a constant.
  Since the highest-order derivative in 
$\mathfrak{r}$ is given by $\nabla^8 \eta$ and this term has been estimated by \eqref{estd7bpa}, we need only consider the integral
\begin{align*} 
\mathcal{I} ^{8}_{1,\operatorname{high}} 
&=
- \alpha ^{-1}  \int_{\Omega_0} \d^{10} J^{-2} \ \nabla^8 \eta^j,_{s} J[A^s_j A^k_i - A^k_j A^s_i]  \ \nabla^8 v^i,_k \, dx \,,
\end{align*} 
and by  Lemma \ref{lemma_aenergy},
\begin{align*}
\mathcal{I} ^{8}_{1,\operatorname{high}} 
&=
\frac{\alpha^{-1} }{2} \int_{\Omega_0} \d^{10} J ^{-1} \frac{d}{dt} \left( | \nabla _\eta   (\nabla^8 \eta)|^2  - |\operatorname{div}  _\eta (\nabla^8 \eta)|^2
- 2 |\operatorname{curl}  _\eta (\nabla^8 \eta)|^2
\right)\, dx \,.
\end{align*} 
Then, for the integral $ \mathcal{I} ^{8}_2 $, we expand $ \nabla ^8 J^{-2}$ and by the same argument used for the integral $ \mathcal{I} ^{8}_1 $,
we need only estimate the highest-order term
\begin{align*}
\mathcal{I} ^{8}_{2,\operatorname{high}} 
&= 2 \alpha ^{-1}  \int_{\Omega_0} \d^{10} J^{-3} \nabla ^8 \eta^r,_s a^s_r \ \nabla^8  v^i,_k  \aki \, dx
= \alpha^{-1} \int_{\Omega_0} \d^{10} J^{-1} \frac{d}{dt}  | \operatorname{div} _\eta ( \nabla ^8 \eta)|^2\, dx
 \,.
\end{align*} 
Integrating in time from $0$ to $t$, we find that
\begin{align} 
&\int_0^t \left( \mathcal{I}^8 _1 (s)+ \mathcal{I} ^8_2(s)\right) ds 
 =  \frac{\alpha^{-1}}{2}  \int_{\Omega_0} \d^{10} J^{-1}
\left(  |  \nabla_\eta (\nabla^8  \eta) |^2  +  | \operatorname{div}_\eta (\nabla ^8 \eta)  |^2 -  2| \operatorname{curl}_\eta (\nabla^8\eta)   |^2 \right) \, dx\nonumber  \\
&\
+ \int_0^t  \frac{\alpha^{-2}\dot\alpha}{2}  \int_{\Omega_0} \d^{10} J^{-1}
\left(  |  \nabla_\eta (\nabla^8  \eta) |^2 +  | \operatorname{div}_\eta (\nabla^8 \eta)   |^2  -  2| \operatorname{curl}_\eta( \nabla^8 \eta)   |^2 
\right) \, dxds  \nonumber \\
&\
+ \int_0^t  \alpha^{-2}  \int_{\Omega_0}  \d^{10} \P(J^{-1} A) \alpha ^{-1/2}\nabla ^9 \eta \, \alpha ^{-1/2}\nabla ^9\eta\, \alpha ^2 \nabla v \, dx ds 
+ \int_0^t  \mathfrak{R} (s) ds
 - \Ee(0) \,, \nonumber
\end{align} 
where $ \mathfrak{R}$ consists of space integrals with lower-order terms having at most eight partial derivatives of $\eta$, and hence easily estimated
using \eqref{estd7bpa}.  

To estimate $ \mathcal{I} ^8_3$, we write 
$$ \nabla ^8 a \sim \nabla ^9 \eta \times \nabla \eta + \sum_{b=0}^7 c_b  \nabla ^{b+1} \eta\, \nabla ^{8-b} \eta\,,$$
for constants $c_b$.  Clearly $ \sum_{b=0}^7 c_b  \nabla ^{b+1} \eta\, \nabla ^{8-b} \eta$ consists of lower-order terms that can be estimated using 
\eqref{estd6bpa}.  Hence, we consider the the highest-order term in $ \mathcal{I} ^8_3$ and use Lemma \ref{lemma_atan} to write
$$
 \mathcal{I} ^8_{3, \operatorname{high} } = -
4\alpha^{-1} \int_{\Omega_0} \d^9\, \bp \nabla ^8 \eta \, \bp \eta\,   J^{-2}\,    \nabla^8 v \, dx \,.
$$
Notice that $\alpha^{-1/2}\d^{4.5} \bp \nabla ^8 \eta$ is controlled in $L^2(\Omega_0)$ by  \eqref{estd7bpa}, and hence $\int_0^t \mathcal{I} ^8_3(s)ds$
 is easily estimated by virtue of the control of $ \alpha ^2 \d^{4.5}  \nabla^8 v$ in $L^2(\Omega_0)$ and the time-integrability of $ \alpha ^{-2}$.   By
 Lemma \ref{lemma_tan}, the highest-order terms in the integral $ \mathcal{I} ^8_4$ can be written as
 $$
 \mathcal{I} ^8_{4, \operatorname{high} } = 
\alpha ^{-1}  \int_{\Omega_0} \d^9\,\P( J ^{-1} A)  \bp \nabla ^8 \eta \,    \nabla^8 v \, dx \,,
$$
and hence this can be estimated in the same way as $ \mathcal{I} ^8_3$.   With the use of Lemma \ref{lemma_curl}, we thus find our highest-order estimate to be
\begin{align} 
 & \|  \d ^{4.5} \nabla^8 \deta \|_0^2  + \| \alpha ^2 \d ^{4.5} \nabla^8 v\|_0^2 +     \| \alpha^{-1/2} \d^{4.5}\nabla^8 \eta\|_0^2 
+
\| \alpha ^{-1/2} \d^{5}   \nabla ^9 \eta\|_0^2 
  \nonumber
\\
& \qquad \lesssim \Ee(0) + \vartheta \sup_{s \in [0,t]}\Ee(s) +   \sup_{s \in [0,t]} \P( \Ee(s))  
 \,.
 \label{estd8}
\end{align}

 \vspace{.05in}
\noindent
{\it Step 7. Building the higher-order norm $\Ee(t)$.}  Together, the estimates \eqref{estd0}, \eqref{estd0bpa}, \eqref{estd1bpa}, \eqref{estd2bpa}, \eqref{estd3bpa},
\eqref{estd4bpa}, \eqref{estd5bpa},  \eqref{estd6bpa},  \eqref{estd7bpa}, and \eqref{estd8} show that
\begin{align} 
&
\sup_{s \in [0,t]}  
\sum_{b=0}^8\sum_{\a=0}^{8-b} \left( \|\d^{{\frac{b+1}{2}} }  \nabla ^b \bp^\a \deta( \cdot ,t)\|_0^2  
+  \|\alpha ^2 \d^{{\frac{b+1}{2}} }  \nabla ^b \bp^a v( \cdot ,t)\|_0^2 \right) 
 +\sum_{b=1}^9  \| \alpha^{-1/2} \d^{{\frac{b+1}{2}} } \nabla ^b \bp^{9-b} \deta( \cdot ,t)\|_0^2 \nonumber  \\
 &\qquad\qquad\qquad
 \lesssim \Ee(0) + \vartheta \sup_{s \in [0,t]}\Ee(s) +   \sup_{s \in [0,t]} \P( \Ee(s))  \,. \label{main}
\end{align} 
This bound then implies that
\begin{equation}\label{main_sub1}
\sup_{s \in [0,t]}  
\sum_{b=1}^5 \sum_{\a=0}^{10-2b}    \left\| \alpha^{-1/2} \d^b \nabla^b \bp^\a \deta(\cdot ,t)  \right\|_{4-\a/2}^2 
\lesssim \Ee(0) + \vartheta \sup_{s \in [0,t]}\Ee(s) +   \sup_{s \in [0,t]} \P( \Ee(s)) 
\,, \end{equation} 
which is proven as follows.   We begin with $b=1$ in \eqref{main} which bounds $ \ah \d  \nabla \bp^8 \deta \in L^2$ as well as 
 $ \ah \d^{1/2}   \bp^8 \deta \in L^2$.   Thus, from \eqref{w-embed}, we may conclude that 
 $$
 \| \ah \bp^8 \deta \|_0 + \| \ah \d \nabla \bp^8 \deta\|_0 \lesssim \Ee(0) + \vartheta \sup_{s \in [0,t]}\Ee(s) +   \sup_{s \in [0,t]} \P( \Ee(s))  \,.
 $$
Next, setting $b=2$ in \eqref{main}, we may conclude that  $\ah \d^{1/2} \bp^7\deta$, $\ah \d \nabla  \bp^7\deta$, and 
$\ah \d^{3/2} \nabla^2  \bp^7\deta$ are bounded in $L^2$, and hence that
\begin{align} 
\int_{\Omega_0} \d^3 \left( |\bp^7\deta|^2+
|\nabla \bp^7 \deta|^2 + |\nabla^2 \bp^7 \deta|^2\right)dx
 \lesssim \Ee(0) + \vartheta \sup_{s \in [0,t]}\Ee(s) +   \sup_{s \in [0,t]} \P( \Ee(s))  \,. \label{id1}
\end{align}  
Thu, from \eqref{w-embed}, we have that 
\begin{equation}
\int_{\Omega_0} \d \left( |\bp^7\deta|^2+ |\nabla \bp^7 \deta|^2\right)dx
 \lesssim \Ee(0) + \vartheta \sup_{s \in [0,t]}\Ee(s) +   \sup_{s \in [0,t]} \P( \Ee(s))  \,. \label{id2}
\end{equation} 
  From\footnote{Note well that the weighted embedding theorems are codimension-$1$ results; only the behavior in the direction normal to the boundary
  of the domain determines the inequality.} Theorem 5.3 in \cite{KuPe2003} and the definition of the fractional $H^s$ norm in Section \ref{sec:frac}, we 
have that
$$
 \| \ah \bp^7 \deta \|_{1/2}  \lesssim \Ee(0) + \vartheta \sup_{s \in [0,t]}\Ee(s) +   \sup_{s \in [0,t]} \P( \Ee(s))  \,.
 $$
Now, since $ \nabla (d\nabla \bp^7 \deta) = \d \nabla ^2 \bp^7 \deta + \nabla \d \, \nabla \bp^7 \deta$, we see that \eqref{id1} and \eqref{id2} together show that
\begin{equation}
\int_{\Omega_0} \d \left( |\d \nabla \bp^7\deta|^2+ |\nabla (\d\nabla  \bp^7 \deta)|^2\right)dx
 \lesssim \Ee(0) + \vartheta \sup_{s \in [0,t]}\Ee(s) +   \sup_{s \in [0,t]} \P( \Ee(s))  \,. \nonumber
\end{equation} 
Then, applying Theorem 5.3 in \cite{KuPe2003} once again, we find that
$$
 \| \ah \bp^7 \deta \|_{1/2} + \| \ah \d \nabla \bp^7 \deta \|_{1/2} \lesssim \Ee(0) + \vartheta \sup_{s \in [0,t]}\Ee(s) +   \sup_{s \in [0,t]} \P( \Ee(s))  \,.
 $$
The remaining summands in \eqref{main} are proved inductively using the identical procedure.

Using the weighted embeddings \eqref{w-embed} and the procedure just described, we also find that
\begin{equation}\label{main2}
\sum_{b=1}^4 \sum_{a=0}^{9-2b} \left( \| \alpha ^2 \d^b \nabla ^b \bp^\a v\|^2_{3.5-\a/2} +  \|  \d^b \nabla ^b \bp^\a \deta\|^2_{3.5-\a/2}\right)
 \lesssim \Ee(0) + \vartheta \sup_{s \in [0,t]}\Ee(s) +   \sup_{s \in [0,t]} \P( \Ee(s)) \,.
\end{equation} 
It follows from the higher-order Hardy-type inequality in Lemma \ref{hardy} that
\begin{align*} 
&\sum_{\a=0}^7  \left(\| \alpha^2 \bp^\a v( \cdot ,t) \|_{3.5-\a/2}^2 +  \|  \bp^\a \deta(\cdot ,t)  \|_{3.5-\a/2}^2 \right)
+\sum_{\a=0}^8   \| \alpha^{-1/2} \bp^\a \deta(\cdot ,t)  \|_{4-\a/2}^2  \\
&\qquad\qquad\qquad
 \lesssim \Ee(0) + \vartheta \sup_{s \in [0,t]}\Ee(s) +   \sup_{s \in [0,t]} \P( \Ee(s)) \,.
\end{align*}

Finally, from the identity \eqref{curl_need},
\begin{align*} 
\alpha ^2  \d^{\frac{b+2}{2}}  \nabla^b \bp^{8-b}\operatorname{curl} _\eta v ( \cdot , t) 
&= \d^{\frac{b+2}{2}}  \nabla^b \bp^{8-b}\operatorname{curl} u_0 +  \varepsilon_{ \cdot kj}
\d^{\frac{b+2}{2}}  \nabla^b \bp^{8-b} \int_0^t \at \left(  \alpha ^2  v^k,_r A^r_l  \ \alpha ^2v^l,_m A^m_j \right) \, dt' .
\end{align*} 
The $L^2(\Omega_0)$-norm is estimated using integration-by-parts in time, and follows identically the argument for the integral $ \mathcal{J} _2^{0,8}$ in 
the proof of Lemma \ref{lemma_curl}, which shows that
$$
\|\alpha ^2  \d^{\frac{b+2}{2}}  \nabla^b \bp^{8-b}\operatorname{curl} _\eta v ( \cdot , t) \|_0^2 
 \lesssim \Ee(0) + \vartheta \sup_{s \in [0,t]}\Ee(s) +   \sup_{s \in [0,t]} \P( \Ee(s)) \,.
$$

 \vspace{.05in}
\noindent
{\it Step 8. Bound for $\Ee(t)$ and global existence.}   We have established that 
$$
\sup_{s \in [0,t]} \Ee(s)
 \lesssim \Ee(0) + \vartheta \sup_{s \in [0,t]}\Ee(s) +   \sup_{s \in [0,t]} \P( \Ee(s))  \text{ for } t \in [0,T]\,,
$$
where $T$ is independent of $ \kappa$, since $\Ee(0)$ is  independent of $\kappa$.   First, we choose $\vartheta \ll 1$ sufficiently small.  Then,
by assumption $\Ee(0) \le \epsilon$, and we 
choose $ \epsilon \ll \vartheta $ sufficiently small so that  $ \Ee(t) \lesssim \epsilon$ for all $t \in [0,T]$.  Then, by the standard continuation argument,
\begin{equation}\nonumber
 \Ee(t) \lesssim \epsilon \ \ \text{ for all } t \ge 0\,.
\end{equation} 
Hence, by the  Sobolev embedding theorem, $\| \nabla \deta\|_{L^ \infty } \lesssim \epsilon $ so that  the inequalities \eqref{eta_bound2d}--\eqref{Jbound}
are significantly improved.    As $\Ee(t)$ is bounded by for $ t \ge 0$, we have that $\eta \in L^ \infty (0,\infty ; H^4(\Omega_0))$ and by applying Theorem A.2 in
\cite{CoSh2014} to our functional framework we obtain that $\eta \in C^0 (0,\infty ; H^4(\Omega_0))$.  
Since $\Gamma(t) = \eta(\Gamma_0,t)$ we then have that $\Gamma(t)$ is of Sobolev class $H^{3.5}$ for all time $t\ge 0$.

 \vspace{.05in}
\noindent
{\it Step 9. The limit as $\kappa \to 0$.} In order to produce our energy bounds, we have used the 
smooth sequence $(\eta_\kappa, v_\kappa)$ introduced in Step 1.  
Since we have established that
$\|  v_\kappa( \cdot ,t) \|_{3.5}^2 +  \|  \eta_\kappa(\cdot ,t)\|_4^2 \lesssim \epsilon $ with
$ \epsilon $ independent of $\kappa$, we have compactness, and it is a standard argument to show that 
$\eta_k \to \eta$ in  $C^2(\overline\Omega_0 \times [0,T])$ for any $T < \infty $, where $\eta$ is a solution of \eqref{ceuler03D}, and
$\Ee(t) \lesssim \epsilon $.    By Theorem 1 in \cite{CoSh2012},
if $\Ee^9(t)$ is bounded the solution is unique.

 \end{proof}

\subsection{Global existence and stability of general affine solutions for $\gamma=2$} 
We now explain the modifications required to analyze the stability of general affine solutions to \eqref{agen}.  
We shall make use of the following corollary to  Lemma \ref{lemma_affine}.
\begin{corollary}\label{cor_main}  Suppose that $\A(t)$ is a (global) solution of the system \eqref{agen} in 3-d which satisfies the
estimates of Lemma \ref{lemma_affine}.
 Let  $\bar t:= 1+t$ and define
$$
\beta=\left\{\begin{array}{ll} \bar t^{-3(\gamma-1)+1} & 1< \gamma < \tfrac 4 3 \\ \log(1+ \bar t) & \gamma= \tfrac 4 3 \\ 1 & \tfrac 4 3 < \gamma  \end{array}\right. \,,
$$
so that $\bar t ^{-2} \beta \in L^1[0, \infty )$.   We have the following asymptotic behavior:
\begin{subequations}
\label{cs1}
\begin{alignat}{2}
\A(t) & = \bar t \A_1 +  \bar \B_1(t)\,, \qquad && \frac{d^k \bar \B_1}{dt^k} (t) = O(\bar t^{-k}\beta(t)) \text{ as } t \to \infty  \,, \\
\operatorname{cof} \A(t) & = \bar t ^{2}  \operatorname{cof} \A_1 + \bar \B_2(t)\,, \qquad && \frac{d^k \bar \B_2}{dt^k} (t) = O(\bar t^{1-k}\beta(t)) \text{ as } t \to \infty  \,, \\
\det\A(t) & = \bar t ^{3}  \det \A_1 + \bar \B_3(t)\,, \qquad && \frac{d^k \bar \B_3}{dt^k} (t) = O(\bar t^{2-k}\beta(t)) \text{ as } t \to \infty  \,, \\
\A ^{-1} (t) & = \bar t ^{-1} \A_1 ^{-1}  + \bar \B_4(t)\,, \qquad && \frac{d^k \bar \B_4}{dt^k} (t) = O(\bar t^{-2-k}\beta(t)) \text{ as } t \to \infty  \,, \\
\dot\A(t) \A ^{-1} (t) & = \bar t ^{-1}  \operatorname{Id}  + \bar \B_5(t)\,, \qquad && \frac{d^k \bar \B_5}{dt^k} (t) = O(\bar t^{-2-k}\beta(t)) \text{ as } t \to \infty  \,, \\
\bar t^2 (\det\A(t)) ^{-1} & ={\frac{1}{\bar t \det \A_1}}   + \bar \B_6(t)\,, \qquad && \frac{d^k \bar \B_6}{dt^k} (t) = O(\bar t^{-2-k}\beta(t)) \text{ as } t \to \infty  \,,
\end{alignat} 
\end{subequations}
for $k=0, 1, 2$.
\end{corollary} 
By Remark \ref{rem1}, we may assume that $ \A_1$ is a diagonal matrix given by
\begin{equation}\label{cs2}
\A_1 = \operatorname{Diag} [\ss_1, \ss_2, \ss_3] \  \text{ and }\  \Lambda = \A_1^{-2} \,,
\end{equation} 
where $\ss_i>0$ for $i=1,2,3$.  
We set
\begin{align*} 
\ee  & =  \A_1^2 \eta\,, \ \vv = \A_1^2 v\,, \ \ \text{ and } \ \dee(x,t) = \ee (x,t) - \A_1^2 x
\,, \\
\AA  & =  [ \nabla \ee] ^{-1} \,, \ \JJ = \det \nabla \ee \,.
\end{align*} 
Notice that 
\begin{equation}\label{cs4}
\AA = A\, \Lambda\,, \  \JJ = J\, \det{\A_1^2}\,,  \ \text{ and } \ \aa = \JJ \, \AA \,.
\end{equation}

The norm $\Ee^K(t)$ is modified as follows:
\begin{align} 
\Ee^K(t) &=\sum_{b=0}^K\sum_{\a=0}^{K-b} \left( \|\d^{{\frac{b+1}{2}} }  \nabla ^b \bp^\a \deta( \cdot ,t)\|_0^2  
+  \| \bar t  \d^{{\frac{b+1}{2}} }  \nabla ^b \bp^a  \A v( \cdot ,t)\|_0^2 \right) 
 +\sum_{b=1}^{K+1}  \| \sqrt{{\frac{\bar t^2}{\det \A}} } \d^{{\frac{b+1}{2}} } \nabla ^b \bp^{K+1-b} \deta( \cdot ,t)\|_0^2
\,.\nonumber
\end{align} 
We continue to denote $ \Ee^8(t)$ by $\Ee(t)$.

\begin{theorem}[Perturbations of general affine motion for $\gamma=2$]\label{thm_g2_gen}For $\gamma=2$ and  $ \epsilon >0$ taken sufficiently small, if the initial data satisfies
 $\Ee(0) <  \epsilon $, then there exists a global  solution $\eta(x,t)$ of \eqref{ceuler03D}  such that $\Ee(t) \le \epsilon $ for all $t \ge 0$.    In particular,
 the flow map $\eta \in C^0([0, \infty ), H^4(\Omega_0))$ and the moving vacuum boundary $\Gamma(t)$ is of class $H^{3.5}$ for all $t\ge 0$.
 The composition
 $\xi(x,t)= \alpha (t) \Lambda  \eta(x,t)$  defines a  global-in-time solution of  the Euler equations \eqref{ceuler3d}.   
  If the initial data satisfies $\Ee^9(t) \le C$, then the solution $ \xi (x,t)$ is unique.
\end{theorem}

\subsubsection{Curl estimates for the general momentum equation \eqref{ceuler_gen}} 

The presence of the affine matrix $\A(t)$ in the momentum equation \eqref{ceuler_gen} changes the method in which we obtain the curl estimates.  For
general affine motion, we have the following
\begin{lemma}[Curl estimates]\label{lemma_curl_gen} For all $t\ge 0$, 
\begin{align*} 
  &\underbrace{ \sum_{b=0}^7\sum_{\a=0}^{7-b} }_{b+\a>0} \|\d^{{\frac{b+2}{2}} } \operatorname{curl}   \nabla ^{b} \bp^\a \dee ( \cdot ,t)\|_0^2  
 +\sum_{b=0}^8  \| \sqrt{{\frac{\bar t^2}{\det \A}} }  \d^{{\frac{b+2}{2}} }\operatorname{curl}  \nabla ^b \bp^{8-b} \dee( \cdot ,t)\|_0^2  \lesssim 
  \Ee(0)  + \sup_{s \in [0,t]} \P (\Ee(s)) \,.
 \end{align*} 
\end{lemma} 
\begin{proof} 
We shall explain the modifications of the proof of Lemma \ref{lemma_curl} that are required to prove this inequality for the general affine motion.   Recall
that for isentropic flows, the Eulerian vorticity $\omega= \operatorname{curl} u$ satisfies the evolution equation $\p_t \omega + u \cdot \nabla \omega+ 
\omega \operatorname{div} u = \omega \cdot \nabla u$.   This means that with regards to the Lagrangian variable $\xi(x,t)$, we have that
$\operatorname{curl} _\xi \p_t \xi =  \mathcal{J} ^{-1}  \nabla \xi \cdot \omega_0$,  where
$\omega_0^i = \varepsilon_{ijk} \left( \A^k_l (0) (u_0)^l,_s (\A ^{-1}) ^s_j(0)  + \dot\A^k_s(0) (\A ^{-1} )^s_j(0)    \right)$.

In components, we have that
$$
\varepsilon_{ijk} \p_t \xi^k,_r B^r_j = \mathcal{J} ^{-1} \xi^i,_k \omega_0^k \,.
$$
Now, we use the fact that $ \mathcal{J} = \det \A \, J$ and  $B^r_j = A^r_s (\A ^{-1} )^s_j$ to conclude that
$$
\varepsilon_{ijk} \p_t \xi^k,_r A^r_s ({\A }^{-1} )^s_j = {\frac{1}{\det \A}}  J ^{-1} {\A}^i_l \eta^l,_r \omega_0^r \,.
$$
Then, the identity $ \p_t \xi = \A v + \dot \A \eta$ shows that
\begin{equation}\label{ssss1}
\varepsilon_{ijk} \A^k_m v^m,_r A^r_s ({\A }^{-1} )^s_j = - \varepsilon_{ijk} \dot \A^k_s ({\A }^{-1} )^s_j +  {\frac{1}{\det \A}}  J ^{-1} {\A}^i_l \eta^l,_r \omega_0^r \,.
\end{equation} 
and hence that
\begin{align} 
\varepsilon_{ijk} \A^k_m v^m,_s  ({\A }^{-1} )^s_j 
&=- \varepsilon_{ijk} \A^k_m v^m,_r (A^r_s- \delta ^r_s) ({\A }^{-1} )^s_j  - \varepsilon_{ijk}  \dot\A^k_s ({\A }^{-1} )^s_j +  {\frac{1}{\det \A}}  J ^{-1} {\A}^i_l \eta^l,_r \omega_0^r  \nonumber \\
& =- \varepsilon_{ijk} \A^k_m v^m,_r \int_0^t \p_t A^r_s(t') dt' ({\A }^{-1} )^s_j  - \varepsilon_{ijk} \dot \A^k_s ({\A }^{-1} )^s_j +  {\frac{1}{\det \A}}  J ^{-1} {\A}^i_l \eta^l,_r \omega_0^r \,.
\nonumber
\end{align} 
Next, using the identities \eqref{cs1}, we have that
\begin{align*} 
\varepsilon_{ijk} \A^k_m v^m,_s  ({\A }^{-1} )^s_j 
&= \varepsilon_{ijk} \left( t \A_1 + \bar \B_1(t)\right)^k_ m  v^m,_s  \left( t ^{-1} \A_1 ^{-1}  + \bar \B_4(t) \right)^s_j \\
& = \varepsilon_{ijk} (\A_1)^k_m v^m,_s  ({\A_1 }^{-1} )^s_j  +\bar\F_{im}^s(t) v^m_s
\end{align*} 
and
$$
\frac{d^k\bar\F_{im}^s}{dt^k}(t)= O(t^{-1-k}) \text{ as } t \to \infty \,.
$$

We consider the third component of this curl vector corresponding to $i=3$. With \eqref{cs2}, we have that
\begin{align*} 
{\frac{\ss_2}{\ss_1}} v^2,_1 - {\frac{\ss_1}{\ss_2}} v^1,_2 &= - \bar\F_{3\, m}^s(t) v^s,_m - \varepsilon_{3jk} \A^k_m v^m,_r \int_0^t \p_t A^r_s(t') dt' ({\A }^{-1} )^s_j  \\
& \qquad - \varepsilon_{3jk} \dot \A^k_s ({\A }^{-1} )^s_j +  {\frac{1}{\det \A}}  J ^{-1} {\A}^i_l \eta^3,_r \omega_0^r \,.
\end{align*} 
We multiply this equation by $\ss_1 \ss_2$ to obtain that\
\begin{align*} 
[\operatorname{curl} \vv]^3:=
\ss_2^2 v^2,_1 - \ss_1^2 v^1,_2 &= - \ss_1 \ss_2  \bar\F_{3\, m}^s(t) v^s,_m - \ss_1\ss_2 \varepsilon_{3jk} \A^k_m v^m,_r \int_0^t \p_t A^r_s(t') dt' ({\A }^{-1} )^s_j  \\
& \qquad - \ss_1\ss_2 \varepsilon_{3jk} \dot \A^k_s ({\A }^{-1} )^s_j +  {\frac{\ss_1\ss_2}{\det \A}}  J ^{-1} {\A}^i_l \eta^3,_r \omega_0^r  \,.
\end{align*} 

As before, the precise structure of the right-hand side of the above identity does not play any role; as such, in order to simplify the presentation, we write
\begin{align*} 
\operatorname{curl} \vv & \sim \F \nabla v+ \A \nabla v \int_0^t A \, \nabla v\, A dt' \, \A ^{-1}  + \dot\A \A ^{-1} + (\det \A) ^{-1}  J ^{-1} \A \nabla \eta \omega_0 \,.
\end{align*} 
Since all of the curl estimates utilize space differentiation, the term $ \dot\A \A ^{-1} $ can be removed, and we see that
\begin{align*} 
\operatorname{curl} \vv 
&\sim\F t ^{-1}  \A ^{-1}(t \A \nabla v)+ \A \nabla v \int_0^t A \, t ^{-1} \A ^{-1}  \,(t \A \nabla v)\, A dt' \, \A ^{-1}   + (\det \A) ^{-1}  J ^{-1} \A \nabla \eta \omega_0 \,,
\end{align*} 
and hence
\begin{align} 
\operatorname{curl} \dee 
&\sim \int_0^t  \F t ^{-1}  \A ^{-1}(t \A \nabla v) dt'  + \int_0^t  (\det \A) ^{-1}  J ^{-1} \A \nabla \eta \omega_0 dt'  \nonumber \\
& \qquad
+ \int_0^t  \A \nabla v \int_0^{t'} A \, t ^{-1} \A ^{-1}  \,(t \A \nabla v)\, A dt'' \, \A ^{-1} dt'  \nonumber \\
&\sim \int_0^t  \F t ^{-1}  \A ^{-1}(t \A \nabla v) dt'  + \int_0^t  (\det \A) ^{-1}  J ^{-1} \A \nabla \eta \omega_0 dt'  \nonumber \\
& \qquad
+ \int_0^t \left[ t ^{-1}  (t \A \nabla v ) \int_0^{t'} A \, t ^{-1} \A ^{-1}  \,(t \A \nabla v)\, A dt'' \, \A ^{-1} \right]dt'   \,.
 \label{cs3} 
\end{align} 
Notice that by \eqref{cs1},  $ \int_0^t {t' }^{-1} |\A ^{-1} (t')| dt' \le C  < \infty$ for $ t \ge 0$, and so with $ t ^{-1}| \A ^{-1} |$ replacing $ \alpha ^{-2} $
and $\sqrt{{\frac{t^2}{\det \A}} }$ replacing $ \alpha ^ {-\frac{1}{2}}$,
we see that \eqref{cs3} has the same derivative count and time-integrability properties as \eqref{sss2}.

The weighted derivative estimates for $\operatorname{curl} \dee$ then follow from the same procedure as in the proof of Lemma  \ref{lemma_curl}, and
we do not repeat them here.
\end{proof}

\subsubsection{Modifications to the energy estimates for \eqref{ceuler_gen}}  
In order to make use of our curl estimates for $\dee$, we shall multiply  \eqref{vort_gen} by $\Lambda $ to obtain
$$
\Lambda ^l_i (\A^T)^i _s \A^s_j \p_t v^j + 2\Lambda ^l_i(\A^T)^i _s \dot\A^s_j v^j + {\frac{1}{\det \A}} \Lambda ^l_i \eta^i 
+{\frac{2 \det \A_1^2}{\det \A}}  \Lambda ^l_i  A^k_i  \left(\d (J\det(\A_1^2))^{-1}\right),_k  =0  \,.
$$
By \eqref{cs4}, this is equivalent to 
$$
\Lambda ^l_i (\A^T)^i _s \A^s_j \p_t v^j + 2\Lambda ^l_i(\A^T)^i _s \dot\A^s_j v^j + {\frac{1}{\det \A}} \Lambda ^l_i \eta^i 
+{\frac{2 \det \A_1^2}{\det \A}} \AA^k_l  \left(\d \JJ^{-1}\right),_k  =0  \,,
$$
which can also be written as
\begin{equation} \label{cs5}
\Lambda ^l_i (\A^T)^i _s \A^s_j \p_t v^j + 2\Lambda ^l_i(\A^T)^i _s \dot\A^s_j v^j + {\frac{1}{\det \A}} \Lambda ^l_i \eta^i 
+{\frac{ \det \A_1^2}{\det \A}} \d ^{-1}  \aa^k_l  \left(\d^2 \JJ^{-2}\right),_k  =0  \,.
\end{equation} 

Following the energy estimates in the case of simple affine motion, we let $\bp ^\a \nabla ^b$ act on \eqref{cs5}, multiply by 
$\bar t^2 \d^{b+1}\bp ^\a \nabla ^b \vv^l$ and integrate over $\Omega_0$.    Thanks to Lemma \ref{lemma_curl_gen} and Corollary \ref{cor_main}, the difficult estimates involving the  nonlinear term
${\frac{ \det \A_1^2}{\det \A}} \d ^{-1}  \aa^k_l  \left(\d^2 \JJ^{-2}\right),_k$ are done in the identical fashion.

Let us now explain how the other terms are dealt with.    As will be clear, it suffices to consider the 
 $L^2(\Omega_0)$ estimates for \eqref{cs5}.
Since $\vv^l = ( \Lambda ^{-1} )^l_r v^r$, we have that
\begin{align} 
&\int_{\Omega_0}\Lambda ^l_i (\A^T)^i _s \A^s_j \p_t v^j \ t^2 \d   \vv^l dx =
\int_{\Omega_0} \bar t^2 \d \, \Lambda ^l_i (\A^T)^i _s \A^s_j \p_t v^j \  ( \Lambda ^{-1} )^l_r v^r dx   = \int_{\Omega_0}  \bar t^2 \d \,  (\A^T)^i _s \A^s_j \p_t v^j \,  v^i dx \nonumber \\
& \qquad = \int_{\Omega_0}  \bar t^2 \d \,  \A^s_j \p_t v^j \,  \A^s_i v^i dx = \int_{\Omega_0} \bar  t^2 \d \,  \p_t( \A^s_j v^j) \,  \A^s_i v^i dx
- \int_{\Omega_0} \bar t^2 \d \, \dot \A^s_j  v^j \,  \A^s_i v^i dx \nonumber \\
&\qquad
={\frac{1}{2}}  \int_{\Omega_0} \bar  t^2 \d \,  \p_t | \A v|^2 dx - \int_{\Omega_0}  \bar t^2 \d \, \dot \A^s_j  v^j \,  \A^s_i v^i dx  \nonumber \\
&\qquad
={\frac{1}{2}}  \frac{d}{dt}  \int_{\Omega_0}  \bar t^2 \d \, | \A v|^2 dx-  \int_{\Omega_0}  \bar t \d \, | \A v|^2 dx - \int_{\Omega_0} \bar t^2 \d \, \dot \A^s_j  v^j \,  \A^s_i v^i dx
\label{cs6}
\end{align} 
Next,
\begin{align} 
&2 \int_{\Omega_0} \Lambda ^l_i(\A^T)^i _s \dot\A^s_j v^j \ t^2 \d   \vv^l dx =
2 \int_{\Omega_0} t^2 \d \,  \Lambda ^l_i(\A^T)^i _s \dot\A^s_j v^j \  ( \Lambda ^{-1} )^l_r v^r dx 
 =2 \int_{\Omega_0}  t^2 \d \,  \dot\A^s_j v^j  \A^s_i v^i dx\,.  \label{cs7} 
\end{align} 
Finally,
\begin{align} 
 & \int_{\Omega_0}(\det\A) ^{-1}   \Lambda ^l_i \eta^i \ \bar t^2 \d   \vv^l dx
 =   \int_{\Omega_0}( \det\A) ^{-1} \bar t^2 \d   \Lambda ^l_i \eta^i  \  ( \Lambda ^{-1} )^l_r v^r dx =  \int_{\Omega_0}(\det\A)^{-1}\bar  t^2 \d\, \eta^i  \  v^i dx \nonumber\\
 & \qquad   = {\frac{1}{2}}  \int_{\Omega_0}(\det\A)^{-1}  \bar t^2 \d\, \p_t| \eta|^2 dx 
 = {\frac{1}{2}}  \int_{\Omega_0}(\bar t \det\A_1)^{-1}   \d\, \p_t| \eta|^2 dx + {\frac{1}{2}}  \int_{\Omega_0} \bar \B_6(t)   \d\, \p_t| \eta|^2 dx\nonumber  \\
  & \qquad = {\frac{1}{2\det\A_1}} \frac{d}{dt}   \int_{\Omega_0}{\bar t}^{-1}   \d\, | \eta|^2 dx +
  {\frac{1}{2\det\A_1}}   \int_{\Omega_0}\bar t^{-2}   \d\, | \eta|^2 dx +\int_{\Omega_0} \bar \B_6(t)   \d\, \eta \cdot v\, dx  \,. \label{cs8}
\end{align} 
Summing over the identities \eqref{cs6}--\eqref{cs8}, we have the following sum:
\begin{align*} 
&
{\frac{1}{2}} \frac{d}{dt}\left( \| d^ {\frac{1}{2}}\bar t\A v\|_0^2+ (\det\A_1) ^{-1}  \| d^ {\frac{1}{2}} \bar t^ {-\frac{1}{2}} \eta\|_0^2  \right)
-  \int_{\Omega_0}  \bar t \d \, | \A v|^2 dx + \int_{\Omega_0}  \bar t^2 \d \, \dot \A^s_j  (\A ^{-1} )^j_r \A^r_l v^l \,  \A^s_i v^i dx \\
& \qquad 
+ (2\det\A_1) ^{-1}  \| d^ {\frac{1}{2}} \bar t^ {-1} \eta\|_0^2 +\int_{\Omega_0} \bar \B_6(t)   \d\, \eta \cdot v\, dx \,.
\end{align*} 
Now, using (\ref{cs1}c), we have that
 \begin{align*} 
 -  \int_{\Omega_0} \bar t \d \, | \A v|^2 dx + \int_{\Omega_0} \bar t^2 \d \, \dot \A^s_j  (\A ^{-1} )^j_r \A^r_l v^l \,  \A^s_i v^i dx
 = \int_{\Omega_0} \bar  t^2 \d \,( \bar \B_5 )_{rs} \A^r_l v^l \,  \A^s_i v^i dx \,.
 \end{align*} 
 Hence, the sum of  the identities \eqref{cs6}--\eqref{cs8} gives
 \begin{align*} 
&
{\frac{1}{2}} \frac{d}{dt}\left( \| d^ {\frac{1}{2}} \bar t\A v\|_0^2+ (\det\A_1) ^{-1}  \| d^ {\frac{1}{2}} \bar t^ {-\frac{1}{2}} \eta\|_0^2  \right) + (2\det\A_1) ^{-1}  \| d^ {\frac{1}{2}}\bar  t^ {-1} \eta\|_0^2 \\
& \qquad 
+\int_{\Omega_0} \bar \B_6(t)   \d\, \eta \cdot v\, dx  +  \int_{\Omega_0}  t^2 \d \,( \bar \B_5 )_{rs} \A^r_l v^l \,  \A^s_i v^i dx\,.
\end{align*} 
This sum is integrated in time, and we use the fact that $ \bar \B_5(t)$ and $  \bar \B_6(t)$ are $O(\bar t ^{-2} )$ as $t \to \infty $ and hence the time integral
of $ \bar \B_5(t)$ and $  \bar \B_6(t)$ is uniformly bounded as $t \to \infty $. 

The higher-order estimates, applying the operator $t^2 \d^{b+1}\bp ^\a \nabla ^b \vv^l$ to \eqref{cs5} and testing with 
$\bar t^2 \d^{b+1}\bp ^\a \nabla ^b \vv^l$ work in the identical fashion as it is only the commutation of time-derivatives that differs here from the simple affine case.
Hence, the energy estimates follow in the identical fashion as for the simple affine matrix.

\subsection{Global existence and stability for  all $\boldsymbol{\gamma} \boldsymbol{>} {\frac{5}{3}} $}   For $1 < \gamma \le {\frac{5}{3}} $, global
existence and stability of affine solutions was proven in \cite{HaJa2016}.  We now establish this result for all $\gamma > {\frac{5}{3}} $.

For general $ \gamma$, we 
set $\rho_0(x)=\d(x)^ {\frac{1}{\gamma -1}} $, and write the Euler equations as
\eqref{ce03d} as
\begin{equation}\label{ce3d_gamma}
 \p^2_t\xi^i+ \d(x)^ {-\frac{1}{\gamma -1}}   b^k_i  \left(\d(x)^ {\frac{\gamma}{\gamma -1}}  \mathcal{J} ^{-\gamma}\right),_k =0 \ \  \text{ in } \Omega_0  
  \times (0,T] \,.
\end{equation} 
The perturbation $\eta(x,t)$ then satisfies
\begin{align} 
&(\A^T)^i _s \A^s_j \p_t v^j + 2(\A^T)^i _s \dot\A^s_j v^j +(\det \A)^{1-\gamma} \eta^i + (\det \A)^{1-\gamma}  \d(x)^ {-\frac{1}{\gamma -1}} a^k_i   \left(\d(x)^ {\frac{\gamma}{\gamma -1}} J ^{-\gamma}\right),_k  =0 \,. \label{magoo10}
\end{align} 


\begin{definition}[Norm for the case that ${\mathbf \gamma > {\frac{5}{3}} }$ norm]
\begin{align} 
\Eeg^K(t) &=\sum_{b=0}^K\sum_{\a=0}^{K-b} \left( \|\d^{{\frac{b(\gamma-1)+1}{2\gamma-2}} }  \nabla ^b \bp^\a \dee( \cdot ,t)\|_0^2  
+  \|\bar t\, \d^{{\frac{b(\gamma-1)+1}{2\gamma-2}} }  \nabla ^b \bp^a \A v( \cdot ,t)\|_0^2 \right) \nonumber \\
 &+
 \sum_{b=1}^{K+1}  \| \bar t\, (\det \A)^{\frac{1-\gamma}{2}}  \d^{{\frac{b(\gamma-1)+1}{2\gamma-2}} } \nabla ^b \bp^{K+1-b} \dee( \cdot ,t)\|_0^2  
\,.\nonumber
\end{align} 
\end{definition} 
The curl estimates are modified as follows:
\begin{lemma}[Curl estimates for $\gamma > {\frac{5}{3}} $]\label{lemma_curlg} For $K\ge 8$,  $t\ge 0$,  and $\gamma > {\frac{5}{3}} $, 
\begin{align*} 
  &\underbrace{ \sum_{b=0}^{K-1}\sum_{\a=0}^{K-1-b} }_{b+\a>0}
   \|\d^{{\frac{b(\gamma-1)+2}{2\gamma-2}} } \operatorname{curl}   \nabla ^{b} \bp^\a \dee ( \cdot ,t)\|_0^2  
 +\sum_{b=0}^K  \|  \bar t\, (\det \A)^{\frac{1-\gamma}{2}} \d^{{\frac{b(\gamma-1)+2}{2\gamma-2}} }\operatorname{curl}  \nabla ^b \bp^{K-b} \dee( \cdot ,t)\|_0^2  \\
 &\qquad\qquad
  \lesssim 
  \Eeg^K(0)  + \sup_{s \in [0,t]} \P (\Eeg^K(s)) \,.
 \end{align*} 
\end{lemma} 

\noindent
The proof of this lemma is identical to the proof of Lemma \ref{lemma_curl_gen}.    For the energy estimates, we let $\bp ^\a \nabla ^b$ act on on  \eqref{magoo10},
 then multiply by 
$\bar t^2\d^{b+1/(\gamma-1)}\bp ^\a \nabla ^b v^i$, and integrate over $\Omega_0$.    The procedure for closing the energy estimates is again identical to the case $\gamma=2$ which was explained above; hence, we have
the following
\begin{theorem}[Perturbations of general affine motion for $\gamma> {\frac{5}{3}} $]\label{thm_gg_gen}For $\gamma> {\frac{5}{3}} $, $K \ge 8$, and  $ \epsilon >0$ taken sufficiently small, if the initial data satisfies
 $\Eeg^K(0) <  \epsilon $, then there exists a global  solution $\eta(x,t)$ of \eqref{ceuler03D}  such that $\Eeg^K(t) \le \epsilon $ for all $t \ge 0$.    In particular,
 the flow map $\eta \in C^0([0, \infty ), C^2(\Omega_0))$ and the moving vacuum boundary $\Gamma(t)$ is of class $C^2$ for all $t\ge 0$.
 The composition
 $\xi(x,t)= \alpha (t) \Lambda  \eta(x,t)$  defines a  global-in-time solution of  the Euler equations \eqref{ceuler3d}.   
  If the initial data satisfies $\Eeg^9(t) \le C$, then the solution $ \xi (x,t)$ is unique.
\end{theorem}

For  $1<\gamma \le {\frac{5}{3}} $, as $\gamma$ becomes smaller, we must increase the number, $K$, of space differentiated problems.  
 The higher-order energy $\Eeg^K(t)$ ensures that
$$
\d^{\frac{1+ K(\gamma-1)}{2\gamma -2}}  \nabla ^K v( \cdot , t) \in L^2(\Omega_0).
$$
From the estimates for $ \operatorname{curl} \dee$ given in Lemma \ref{lemma_curl}, it is necessary to have $ \nabla \p_t v ( \cdot , t) \in L^ \infty (\Omega_0)$,
which, in turn, necessitates that $ v( \cdot , t) $ is in $H^s(\Omega_0)$ for $ s > 3$ by the Sobolev embedding theorem.
In order for $v( \cdot , t) \in H^s(\Omega_0)$ for $s > 3$, according to \eqref{w-embed}, it is necessary that
$$
{\frac{1+ K(\gamma-1)}{2\gamma -2}}   < K - 3  \text{ which implies that } K > {\frac{6\gamma -5}{\gamma -1}} \,.
$$

\begin{definition}[Norm for $1< \gamma \le {\frac{5}{3}} $ in 3-d]
For $1 < \gamma \le {\frac{5}{3}} $ and for $K> {\frac{6\gamma -5}{\gamma -1}}$, we define the higher-order norm as
\begin{align}
\Eeg^K(t) &=\sum_{b=0}^K\sum_{\a=0}^{K-b} \left( \|\d^{{\frac{b(\gamma-1)+1}{2\gamma-2}} }  \nabla ^b \bp^\a \dee( \cdot ,t)\|_0^2  
+  \|  \bar t^{{\frac{3}{2}}(\gamma-1) } \,  \d^{{\frac{b(\gamma-1)+1}{2\gamma-2}} }  \nabla ^b \bp^a  \A v( \cdot ,t)\|_0^2 \right) \nonumber \\
 &+
 \sum_{b=1}^{K+1}  \|  \d^{{\frac{b(\gamma-1)+1}{2\gamma-2}} } \nabla ^b \bp^{K+1-b} \deta( \cdot ,t)\|_0^2  
\,.\nonumber
\end{align} 
\end{definition} 
We have the corresponding curl lemma, again proved in the identical way as Lemma \ref{lemma_curl_gen}.
\begin{lemma}[Curl estimates for $1< \gamma \le {\frac{5}{3}} $]\label{lemma_curlglow} For $K> {\frac{6\gamma -5}{\gamma -1}}$,  $t\ge 0$,  and 
$1< \gamma \le {\frac{5}{3}} $, 
\begin{align*} 
  &\underbrace{ \sum_{b=0}^{K-1}\sum_{\a=0}^{K-1-b} }_{b+\a>0}
   \|\d^{{\frac{b(\gamma-1)+2}{2\gamma-2}} } \operatorname{curl}   \nabla ^{b} \bp^\a \dee ( \cdot ,t)\|_0^2  
 +\sum_{b=0}^K  \|  \d^{{\frac{b(\gamma-1)+2}{2\gamma-2}} }\operatorname{curl}  \nabla ^b \bp^{K-b} \dee( \cdot ,t)\|_0^2  \\
 &\qquad\qquad
  \lesssim 
  \Eeg^K(0)  + \sup_{s \in [0,t]} \P (\Eeg^K(s)) \,.
 \end{align*} 
\end{lemma}

 For the energy estimates, we let $\bp ^\a \nabla ^b$ act on \eqref{magoo10}, then multiply by 
$\bar t^{3(\gamma-1)} \d^{b+1/(\gamma-1)}\bp ^\a \nabla ^b v^i$, and integrate over $\Omega_0$.   We then have that
\begin{theorem}[Perturbations of affine motion for $1<\gamma\le {\frac{5}{3}} $]
For $1<\gamma\le {\frac{5}{3}} $, $K  > {\frac{6\gamma -5}{\gamma -1}}$, and  $ \epsilon >0$ taken sufficiently small, if the initial data satisfies
 $\Eeg^K(0) <  \epsilon $, then there exists a global  solution $\eta(x,t)$ of \eqref{ceuler03D}  such that $\Eeg^K(t) \le \epsilon $ for all $t \ge 0$.    In particular,
 the flow map $\eta \in C^0([0, \infty ), C^2(\Omega_0))$ and the moving vacuum boundary $\Gamma(t)$ is of class $C^2$ for all $t\ge 0$.
 The composition
 $\xi(x,t)= \alpha (t) \Lambda  \eta(x,t)$  defines a  global-in-time solution of  the Euler equations \eqref{ceuler3d}.   
  If the initial data satisfies $\Eeg^L(t) \le C$ for $L > {\frac{6\gamma -5}{\gamma -1}}+1$ , then the solution $ \xi (x,t)$ is unique.
\end{theorem} 
\begin{remark} In fact, by
Lemma
\ref{lemma_affine}, our proof works for $\gamma$ in the range $ {\frac{4}{3}} < \gamma \le {\frac{5}{3}} $, while the result of \cite{HaJa2016} gives stability for general affine flows
when $1 < \gamma \le
{\frac{4}{3}} $.
\end{remark} 

\section*{Acknowledgments}
SS was supported by the National
Science Foundation grant DMS-1301380, and by the Department of Energy Advanced Simulation and
Computing (ASC) Program.   We thank the anonymous referee for many useful suggestions which have significantly improved the presentation.

\section*{Compliance with Ethical Standards}
Funding sources have been acknowledged.  There are no potential conflicts of interest.  The research did not involve human participants and/or animals.

\end{document}